\documentclass[a4paper,10pt]{amsart}
\usepackage[utf8]{inputenc}
\usepackage{lmodern}
\usepackage[T1]{fontenc}
\usepackage{microtype} 	
\usepackage[all]{xy}

\usepackage{amsmath,amssymb,amsfonts,amsthm,latexsym,mathrsfs}
\usepackage{mathtools}
\usepackage{ stmaryrd } 
\usepackage{xcolor}
\usepackage{hyperref}
\usepackage{cleveref}
\hypersetup{colorlinks=true,linkcolor=gray,citecolor=gray}
\usepackage{xspace}
\usepackage[hyperpageref]{backref} 
\newtheorem{lemma}{Lemma}[section]
\newtheorem{proposition}[lemma]{Proposition}
\newtheorem{theorem}[lemma]{Theorem}

\newtheorem{corollary}[lemma]{Corollary}

\newtheorem{maintheorem}{Theorem}
\newtheorem{mainproposition}[maintheorem]{Proposition}
\newtheorem{maincorollary}[maintheorem]{Corollary}

\theoremstyle{definition}

\newtheorem{definition}[lemma]{Definition}
\newtheorem{conjecture}[lemma]{Conjecture}
\newtheorem{question}[lemma]{Question}
\theoremstyle{remark}
\newtheorem{remark}[lemma]{Remark}
\newtheorem{example}[lemma]{Example}

\newtheorem*{remark*}{Remark}
\newtheorem*{problem*}{Problem}

\newcommand\Ima{{\rm Im}}

\newcommand{\GL}{\operatorname{GL}}
\newcommand{\SL}{\operatorname{SL}}
\newcommand{\PSL}{\operatorname{PSL}}
\newcommand{\Ma}{\operatorname{M}}

\newcommand\Aut{{\rm Aut}}

\newcommand\supp{{\rm supp}}
\newcommand\PCI{\operatorname{PCI}}

\newcommand{\ZZ}{\mathcal{Z}}
\renewcommand{\O}{\mathcal{O}}
\newcommand{\U}{{\mathcal U}}

\newcommand\Soc{\operatorname{Soc}}
\newcommand\Cen{\operatorname{Cen}}

\newcommand\mc{\mathcal}
\newcommand\ov{\overline}

\newcommand\wt{\widetilde}
\newcommand\wh{\widehat}

\newcommand{\Z}{{\mathbb Z}}

\newcommand{\Q}{{\mathbb Q}}
\newcommand{\R}{{\mathbb R}}

\newcommand\Bic{\operatorname{Bic}}

 \newcommand\restr[2]{{
   \left.\kern-\nulldelimiterspace 
   #1 
   \right|_{#2} 
   }}

\definecolor{Vino}{rgb}{0.256,0,0}

\usepackage[margin=1.4in]{geometry}

\begin{document}
\title{On integral decomposition of unipotent elements in integral group rings}
\author{Geoffrey Janssens}
\author{Leo Margolis}

\address{(Geoffrey Janssens) \newline Institut de recherche en math\'ematique et physique, Universit\'e de Louvain-La-Neuve, Chemin du Cyclotron $2$, $1348$ Louvain-la-Neuve, Belgium \newline E-mail address: {\tt geoffrey.janssens@uclouvain.be}}
\address{(Leo Margolis) \newline Universidad Aut\'onoma de Madrid, Departamento de Matem\'aticas, C/ Francisco Tomás y Valiente 7, 28049 Madrid, Spain \newline E-mail address: {\tt leo.margolis@icmat.es}}

\begin{abstract}
Jespers and Sun conjectured in \cite{JespersSun} that if a finite group $G$ has the property ND, i.e. for every nilpotent element $n$ in the integral group ring $\mathbb{Z}G$ and every primitive central idempotent $e \in \mathbb{Q}G$ one still has $ne \in \mathbb{Z}G$, then at most one of the simple components of the group algebra $\mathbb{Q} G$ has reduced degree bigger than $1$. With the exception of one very special series of groups we are able to answer their conjecture, showing that it is true --- up to exactly one exception. To do so we first classify groups with the so-called SN property which was introduced by Liu and Passman in their investigation of the Multiplicative Jordan Decomposition for integral group rings.

The conjecture of Jespers and Sun can also be formulated in terms of a group $q(G)$ made from the group generated by the unipotent units, which is trivial if and only if the ND property holds for the group ring. We answer two more open questions about $q(G)$ and notice that this notion allows to interpret the studied properties in the general context of linear semisimple algebraic groups. Here we show that $q(G)$ is finite for lattices of big rank, but can contain elements of infinite order in small rank cases. 

We then study further two properties which appeared naturally in these investigations. A first which shows that property ND has a representation theoretical interpretation, while the other can be regarded as indicating that it might be hard to decide ND. Among others we show these two notions are equivalent for groups with SN.  
\end{abstract}

\maketitle

\newcommand\blfootnote[1]{%
  \begingroup
  \renewcommand\thefootnote{}\footnote{#1}%
  \addtocounter{footnote}{-1}%
  \endgroup
}

\blfootnote{The first author is grateful to Fonds Wetenschappelijk Onderzoek vlaanderen - FWO (grant 88258), and le Fonds de la Recherche Scientifique - FNRS (grant 1.B.239.22) for financial support.

The second author was supported by the Spanish ministry of Science and Innovation under a Ramon y Cajal grant (reference RYC2021-032471-I) and a Severo Ochoa Grant CEX2019-000904-S funded by MCIN/AEI/ 10.13039/501100011033}

\vspace{-0,4cm}
\tableofcontents

\section{Introduction}
Already in the 1860's Weierstrass and Jordan introduced what students today learn to call the Jordan normal form of a matrix, cf. \cite{Hawkins} for a historic overview. Reformulating this theory in a more general context for $A$ an algebra over a field $F$ every invertible element $a \in A$ can be uniquely decomposed as $a = a_u a_s$ such that $a_u$ is unipotent, $a_s$ is semisimple and $a_u a_s = a_sa_u$. When $A$ contains a substructure $B$ of interest, e.g. when $F$ is a number field and $B$ an order in $A$, one could ask, if for every invertible $b \in B$ one can still achieve this Jordan decomposition in $B$, i.e. whether $b_n,b_s \in B$ holds. Motivated by the study of units in integral group rings Hales, Luthar and Passi asked when the above will happen for $A= \mathbb{Q}G$ the rational group algebra of a finite group $G$ and $B = \mathbb{Z}G$ the integral group ring therein. Namely they defined a finite group $G$ to have \emph{Multiplicative Jordan Decomposition}, if for every unit $u \in \mathbb{Z}G$ the elements $u_n$ and $u_s$, which a priori are defined in $\mathbb{Q}G$, actually live in $\mathbb{Z}G$, and asked which groups satisfy this property. Though quite a lot of research has been developed to this, the problem remains open in general, see \Cref{sec:MJD} for more details and references. 

A major breakthrough in this investigation came when it was observed in \cite{HPW} that a group which has Multiplicative Jordan Decomposition also has the \emph{Nilpotent Decomposition} (\emph{ND} for short). Namely, $G$ is said to have ND, if for every nilpotent element $n \in \mathbb{Z}G$ and every central idempotent $e \in \mathbb{Q}G$, the product $ne$ still lies in $\mathbb{Z}G$. This property can be reformulated in terms of the associated unipotent elements $1+n$. More precisely, denote by $\U( \Z G)$ the unit group in $\Z G$ and let 
$$ \U(\Z G)_{un} := \{ \alpha \in \U (\Z G) \mid \alpha \text{ is unipotent } \}$$
be the set of unipotent elements in $\U (\Z G)$. Furthermore, for $e$ a primitive central idempotent of $\Q G$ consider the set $\mathcal{E}_G(e) := \{ \alpha \in \U(\Z G)_{un} \mid (\alpha -1) e = \alpha -1\}$ of unipotent elements projecting trivially to all components except the $e$-th one. Denote by $\PCI( \Q G)$ the set of all the central primitive idempotents in $\Q G$. Now, by considering the group $q(G) := \langle  \U(\Z G)_{un} \rangle / \langle \mathcal{E}_G(e) \ | \ e \in \PCI(\Q G) \rangle$ one obtains the alternative characterisation:
\begin{equation}\label{ND via qG intro}
q(G) = 1 \text{ if and only if } G \text{ has ND}.
\end{equation}
Looking on the Wedderburn-Artin decomposition 
\begin{equation}\label{eq:QGWedderburn}
\mathbb{Q}G = M_{n_1}(D_1) \oplus ... \oplus M_{n_\ell}(D_\ell),
\end{equation}
 where $M_{n_i}(D_i)$ denotes the $n_i\times n_i$-matrix ring over a division algebra $D_i$, one sees that property ND will hold if at most one of the $n_i$ is bigger than $1$, as the only unipotent element in a division algebra is the trivial one. This observation during the search for groups having ND led Jespers and Sun to define a group  $G$ as having \emph{at most one matrix component}, if at most one of the $n_i$ in \eqref{eq:QGWedderburn} is bigger than $1$ \cite{JespersSun}. 
Their investigations even made them conjecture that these properties are in fact equivalent:

\begin{conjecture}[{Jespers-Sun, \cite[Conjecture 1]{JespersSun}}]\label{conj:EricWei}
 A finite group $G$ has ND if and only if $\mathbb{Q}G$ has at most one matrix component.
\end{conjecture}  

Using the perspective of $q(G)$ in \eqref{ND via qG intro} and work of Kleinert-del R\`io \cite{KleRio}, \Cref{conj:EricWei} can be elegantly reformulated in terms of unipotent elements. Namely, {\it it conjectures that $\langle \mathcal{U}(\Z G)_{un} \rangle$ is indecomposable.} From this point of view their conjecture is even more surprising.

Note that the indecomposability statement of the group generated by all unipotent elements is of interest for arithmetic subgroups of arbitrary semisimple algebraic groups. In \Cref{the obstruction for general smeisimple} we expand on this generality. In \cite[Section 6]{JespersSun} also the questions of when $q(G)$ is finite and whether there is a connection between the structure of $q(G)$ and the simple components of $\Q G$ were asked. The aim of this article is to answer all the problems above.

The latter two questions will be answered in \Cref{sec:GeneralOrders}. The study of property ND, with the solution of the conjecture above as the ending point, is done in \Cref{Section ND}. This however will require to classify in \Cref{sec:SN} the so-called SN groups, which is a problem of independent interest. Finally, in \Cref{sec:DK prop section} we show that property ND has also a representation theoretical interpretation and yields concrete structural information when considered for specific subsets of all nilpotent elements. We will now explain the main results of this article in more detail.

\subsection*{Jespers-Sun conjecture and SN groups}\addtocontents{toc}{\protect\setcounter{tocdepth}{1}}
\Cref{conj:EricWei} has been the starting point for the investigations presented here. To describe our result on it, define a finite group $G$ to be an \emph{SSN group of unfaithful type}, if there exist primes $p$ and $q$ such that $G = P \rtimes Q$ for $P$ a cyclic group of order $p$ and $Q$ a cyclic group of $q$-power order which acts non-trivially, but also not faithfully on $P$ (the reason for the name will become clear later). Then our main result, obtained in \Cref{th:MainTheoremNDInSection3}, states:

\begin{maintheorem}\label{Summary theorem ND}
Let $G$ be a finite group which is not an SSN group of unfaithful type. Then $G$ has ND if and only if either $\Q G$ has at most one matrix component or $G  = \langle a,b \mid a^4=b^8 =1, a^b =a^{-1} \rangle \cong C_4 \rtimes C_8$.
\end{maintheorem} 

Hence we show that though the conjecture of Jespers and Sun is not correct in general, we know only about one counterexample and all the other potential counterexamples lie in a very specific family of groups. Interpreting the conjecture as a statement on how far the rational group algebra $\mathbb{Q}G$  determines properties of $G$ and its integral group ring, this allows still to show a strong implication:

\begin{maincorollary}\label{cor:QGKnowsND}
Let $G$ and $H$ be groups such that $\mathbb{Q}G \cong \mathbb{Q}H$ and $G$ has ND. Then $H$ has $ND$.
\end{maincorollary} 
This result is obtained in \Cref{coro ND determined by QG}. Note, that there are many groups which have isomorphic rational group algebras \cite[Theorem 14.1.11]{PassmanAlgebraicStructure} and it is hence not a typical situation that a property of $G$ can be recovered from its group algebra over $\Q$.

A concept which turns out to be crucial to reduce our studies to certain classes of groups is that of \emph{groups with SN}. Namely, $G$ is said to have SN, if for every normal subgroup $N \unlhd G$ and subgroup $Y \leq G$ one has $N \leq Y$ or $NY \unlhd G$. This property was introduced by Liu and Passman to restrict the group-theoretical structure of those groups which have ND \cite{LiuPassman09} (the name was later coined in \cite{Liu}), so that a group which does not have SN will not have ND. Liu and Passman obtained some properties of groups with SN \cite{LiuPassman16}, but they were more interested in \emph{groups with SSN}, which are those in which every subgroup has SN, as this property is a consequence of the Multiplicative Jordan Decomposition. They were able to achieve quite explicit descriptions of all the groups having SSN and we generalize their findings in some sense, giving restrictions on the structure of groups with SN. Recall that a group is said to be \emph{Dedekind}, if all its subgroups are normal.

\begin{maintheorem}\label{Summary theorem for SN}
Let $G$ be a finite group. Then $G$ is a group with SN if and only if it is one of the following types:
\begin{enumerate}
\item[(i)] a group with SSN,
\item[(ii)] $G \cong P \rtimes H$ with $P$ an elementary abelian Sylow $p$-subgroup and $H$ a $p'$-Hall subgroup which is Dedekind with cyclic or generalized quaternion Sylow subgroups such that the action of $H$ on $P$ is irreducible and faithful,
\item[(iii)] $G$ has a unique minimal normal subgroup $S$ and $S$ is not solvable and $G/S$ Dedekind,
\item[(iv)] $\Q G$ has one matrix component.
\end{enumerate}
\end{maintheorem}
 
As remarked before, groups with SSN have been classified in \cite{LiuPassman16}, so that we have a precise group-theoretical descriptions for those groups with SN which do not have one matrix component.   

\Cref{Summary theorem for SN} is proven in \Cref{sec:SN} by dividing it into three separate subcases: nilpotent groups (which are handled in  \Cref{prop:NilpotentSNNotSSN}), solvable non-nilpotent groups (\Cref{prop:SolvSNGroups}) and non-solvable groups (\Cref{prop:NonSolvableSNGroups}). The classification of nilpotent groups with SN turns out to be the hardest of those.  With those preparations we then prove \Cref{Summary theorem ND} and \Cref{cor:QGKnowsND} in \Cref{Section ND}, where the former is obtained by dividing in the same cases. 

\subsection*{A general perspective on the obstruction to have ND}
In \Cref{sec:GeneralOrders} we investigate the group $q(G)$ which by (\ref{ND via qG intro}) measures how far a given group $G$ is from having ND or in other words it as an obstruction to have ND. Jespers and Sun \cite[Section 6]{JespersSun} formulated the following two problems about $q(G)$: 

\begin{enumerate}
\item Classify the finite groups $G$ such that $q(G)$ is finite. (\cite[Problem 1, §6]{JespersSun})
\item Find a connection between the structure of $q(G)$ and the simple components of $\Q G$. (\cite[Problem 2, §6]{JespersSun})
\end{enumerate}

In \Cref{measure subsection} we give answers to both questions. More precisely, \Cref{th: upper-bound size q(G)} and the proof of \Cref{when q(G) finite for grp rings} will show that for unipotent elements to have an integral decomposition is not truly connected to the simple components of $\Q G$. The relationship is rather a combination of the congruence level of $\Z G$ in the maximal order of $\Q G$ on the one hand and the rank of the simple matrix components of $\Q G$ on the other hand.

Besides, the second problem fits in a more general context of semisimple algebraic groups. More precisely,  let $F$ be a number field and $S$ a non-empty finite set of places of $F$ containing the Archimedean places. Furthermore, let $\mathbf{G}$ be a simply connected semisimple algebraic group. In particular, $\mathbf{G}$ is a direct product of simply connected almost-simple algebraic $F$-subgroups \cite[Theorem 2.6]{PlaRapBook}, say $\mathbf{G} = \prod_{i=1}^m \mathbf{G}_i$. Finally let $\Gamma$ be an $S$-arithmetic subgroup of $\mathbf{G}(F)$. In \Cref{if no exceptional then finite for general grp} we obtain the following:

\begin{mainproposition}
Consider the notations above and suppose $ S\text{-rank}(\mathbf{G}_i(F)) \geq 2$ for all anisotropic $\mathbf{G}_i(F).$ Then, $|q(\Gamma)|< \infty. $
In particular, in this case finiteness of $q(\Gamma)$ does not depend on the chosen $S$-arithmetic subgroup $\Gamma$.
\end{mainproposition}

In the case that $F= \Q, \mathbf{G}(F) = \SL_1(\Q G)$ and $\Gamma = \SL_1(\Z G)$ we give in \Cref{th: upper-bound size q(G)} a precise and down-to-earth upper bound. Our arguments heavily rely on solutions of the Congruence Subgroup Problem. 

Next, recall that a finite dimensional simple algebra is called {\it exceptional of type II} if it is $\Ma_2(D)$ with $D$ either $\Q$, an imaginary quadratic extension of $\Q$ or a totally definite quaternion algebra with center $\Q$. 

In \Cref{when q(G) finite for grp rings} we show the theorem below, saying that the finiteness of $q(G)$ depends on the presence of exceptional components. For instance, if $3 \nmid |G|$ and $\Q G$ has a simple component isomorphic to $M_2(\Q)$, then by \cite[Remark 6.17]{BJJKT2} $G$ is an extension of $D_8$. In \Cref{measure subsection} we formulate some condition on the exponent of the preimage of $D_8$ in $G$, called $(\star)$, which in turn we prove to have an impact on the size of $q(G)$. 

\begin{maintheorem}\label{theorem size qG intro}
Let $G$ be a finite group. Then the following hold:
\begin{enumerate}
\item[(i)]  If $\Q G$ has no exceptional components of type II, then $q(G)$ is finite.
\item[(ii)] If $G$ has order at most $16$, then $q(G)$ is finite.
\item[(iii)] If $G$ has order bigger than $16$, maps onto $D_8$ and this surjection satisfies $(\star)$, then $q(G)$ is an infinite non-torsion group.
\end{enumerate}
\end{maintheorem}
 Interestingly the group $C_4 \rtimes C_8$ in \Cref{Summary theorem ND} is the smallest group fitting in none of the cases covered by \Cref{theorem size qG intro}.

\subsection*{A representation theoretical and local look at ND}

In \Cref{sec:DK prop section} we will introduce two more properties which appeared in our investigations of the nilpotent decomposition. The first is a purely representation-theoretical property, called DK, namely that any two non-equivalent irreducible $\Q$-representations have different kernels. We show in \Cref{th:OMCImpliesDK} that this is the case for groups with at most one matrix component, but also other interesting classes of groups are show in \Cref{section DK examples} to have that property. We then study in \Cref{subsection on SN versus DK} what one could call a partial nilpotent decomposition, namely that the nilpotent decomposition does hold for those nilpotent elements of $\mathbb{Z}G$ which are the easiest to construct and which we call bicyclic nilpotent. When a decomposition does hold for all such elements, we call a group bicyclic resistant. The reason for the definition of the SN property, is essentially that a group which does not have SN is also not bicyclic resistant. We will show however that the class of not bicyclic resistant groups is bigger than the class of groups with SN, though it does not incorporate some classes relevant in the study of the ND property.
Finally we connect these two new notions by showing:

\begin{maintheorem}
Let $G$ be a finite group with SN. Then the following are equivalent:
\begin{enumerate}
\item $G$ is bicyclic resistant.
\item $G$ is supersolvable or $\Q G$ has one matrix component.
\item $G$ has DK.
\end{enumerate}
\end{maintheorem}
\noindent This result is proven in \Cref{When SN is DK theorem}.

\vspace*{.5cm}
\noindent \underline{\it Conventions and Notations.} $G$ will always denote a finite group. If $\Q G \cong \prod_{i} \Ma_{n_i}(D_i)$ is the Wedderburn-Artin decomposition of the semisimple algebra $\Q G$, then we call the factors $\Ma_{n_i}(D_i)$ \emph{simple components} of $\Q G$. Recall that $n$ is called the \emph{reduced degree} of the simple component $M_n(D)$. If a component has reduced degree $2$ or more we will speak of a \emph{matrix component}. The set $\PCI(G)$ denotes the primitive central idempotents of $\Q G$. 

Moreover we use standard group-theoretical notation: for a group $G$ we denote by $G'$ the derived subgroup of $G$, by $\ZZ(G)$ the center of $G$, by $\Phi(G)$ the Frattini subgroup of $G$, by $g^G$ the conjugacy class of an element $g\in G$ in $G$ and by $\Soc(G)$ the socle of $G$. Moreover for $g,h \in G$, we set $g^h = h^{-1}gh$ and $[g,h] = g^{-1}h^{-1}gh = g^{-1}g^h$. A cyclic group of order $n$ is denoted $C_n$, a dihedral group of order $2n$ by $D_{2n}$, an alternating group of degree $n$ by $A_n$ and $Q_{2^n}$ denotes the generalized quaternion group of order $2^n$, i.e. 
\[Q_{2^n} = \langle a, b \ | \ a^{2^{n-1}} = b^4 = 1, \ b^2 = a^{2^{n-2}}, \ a^b = a^{-1} \rangle. \]
When speaking about generalized quaternion groups, we assume them to be non-abelian, i.e. at least of order $8$. 

If $H$ is a subgroup of $G$ we denote two elements in the rational group algebra $\mathbb{Q}G$ as 
\[\widetilde{H} = \sum_{h \in H} h  \ \ \text{and} \ \ \widehat{H} = \frac{1}{|H|} \sum_{h \in H} h. \]

\vspace{0,2cm}
\noindent \textbf{Acknowledgment.} 
We thank B. Sury and Amir Behar for useful conversations concerning \Cref{unipotent in congruence of SL2Z}. We also thank Eric Jespers and Wei-Liang Sun for interesting conversations. We would also like to thank the referee for many suggestions which improved the readability of the paper.

\section{Description of groups with SN}\label{sec:SN}
Recall that a group $G$ has the {\it SN property} if for every normal subgroup $N$ of $G$ and every subgroup $Y \leq G$ either $N \subseteq Y$ or $N Y \unlhd G$. 
The main goal of this section is to prove \Cref{Summary theorem for SN}. We separate this in essentially three steps: the nilpotent groups with SN (which are handled in  \Cref{prop:NilpotentSNNotSSN}), the solvable non-nilpotent groups with SN (\Cref{prop:SolvSNGroups}) and the non-solvable groups with SN (\Cref{prop:NonSolvableSNGroups}).  The combination of these cases then gives exactly \Cref{Summary theorem for SN}. 

Recall that the group $G$ {\it has SSN} if every subgroup has SN. Such groups have been classified in \cite{LiuPassman16}. In fact we will give a precise classification of groups with SN in case $G$ is non-nilpotent.  Namely, in \Cref{Summary theorem for SN} the non-nilpotent groups with SN are exactly those from (iii) and solvable non-nilpotent groups with SN. As proven in \Cref{prop:SolvSNGroups}  the latter come in two families with one having SSN and the other being the groups from (ii). Though we have made no attempt to classify groups with one matrix component, some restrictions can be filtered out of our proofs.

\subsection{Background results on groups with SN}\addtocontents{toc}{\protect\setcounter{tocdepth}{2}}

If every subgroup of $G$ is normal, then $G$ obviously has SN. Recall that these groups are called {\it Dedekind groups} and have been classified by Baer and Dedekind:

\begin{theorem}[{\cite[Theorem 1.8.5.]{PolMilSehgalBook}}]\label{th:DedekindGroups}
$G$ is a Dedekind group if and only if it is abelian or $G \cong Q_8 \times C_2^n \times A$ for some $n \in \mathbb{N}_0$ and $A$ an abelian group of odd order.
\end{theorem}

Many basic properties of group with SN and SSN have been studied by Liu and Passman and we will use several of their results. For the convenience of the reader we collect them here as well as some other results we will need. 

\begin{lemma}[{\cite[Lemma 2.1]{LiuPassman16}\label{lem:LiuPassma2.1}}]
Let $G$ be a group with SN and $N$ a non-trivial normal subgroup of $G$. If $N$ is not cyclic, then $G/N$ is a Dedekind group. Moreover, if $H$ is a subgroup of $G$ such that $H \cap N = 1$, then $NH \unlhd G$ and $H$ is a Dedekind group. 
\end{lemma}

For the description of solvable groups with SN the following lemma will be key.

\begin{lemma}[{\cite[Lemma 2.4]{LiuPassman16}}]\label{lem:LP2.4}
Let $G$ be a group with SN, $P \in \text{Syl}_p(G)$ such that $P \unlhd G$ and $G = P \rtimes H$ for a $p'$-group $H$ which acts non-trivially on $P$. Then $P$ is elementary abelian and $H$ acts irreducibly on $P$. Moreover, if $H$ acts non-faithfully, then $G$ is an SSN group of unfaithful type.
\end{lemma}

The following is \cite[Lemma 2.5. (1)]{LiuPassman16} where it was stated for groups with SSN. However its proof only uses properties of groups with SN, so that we restate it in this form. The moreover part has been added and follows directly by using \Cref{lem:LiuPassma2.1}.

\begin{lemma}\label{lem:LP2.5}
Let $G$ be a group with SN with non-trivial normal $p$-subgroup $P_0$, say contained in the Sylow $p$-subgroup $P$ of $G$. Then $G$ contains a nilpotent $p$-complement $H$, we have $P_0 H \unlhd G$ and $G = PH$. In particular $G$ is solvable. Moreover, if $P_0$ is not cyclic then $H$ is Dedekind and $P \unlhd G$.
\end{lemma}
\begin{proof}
Suppose $P_0$ is not cyclic. Then \Cref{lem:LiuPassma2.1} implies that $G/P_0$ is Dedekind. Therefore  $P/P_0 \unlhd G/P_0$, which implies that $P \unlhd G$. The proof of the rest of the statement is completely as the proof of \cite[Lemma 2.5. (1)]{LiuPassman16}.
\end{proof}

We will also need a particular way to construct primitive central idempotents of $\mathbb{Q}G$. For this we will use the theory of strong Shoda pairs. For now we give definitions that are sufficient for this section and refer the reader to \Cref{prelim for ND section} and \cite[Chapter 3]{GRG1} for more details. Suppose $G$ is metabelian. Then all the irreducible $\Q$-representations of $G$ are monomial, i.e. they all arise as the induced representations $\lambda^G$ of a linear representation $\lambda$ of some subgroup $H$ of $G$. In that case, one considers $K = \ker(\lambda)$ and denotes by $e(G,H,K)$ the associated primitive central idempotent of $\mathbb{Q}G$. 

\begin{lemma}\label{lem:IdempotForMetabelian}\cite[Theorem 3.5.12 and Exercise 3.4.4]{GRG1}
Assume $G$ is a metabelian group and $A$ a maximal abelian subgroup of $G$ containing $G'$. Then the primitive central idempotents of $\mathbb{Q}G$ are the elements $e(G,H,K)$ where $H$ and $K$ are subgroups of $G$ such that $H$ is a maximal element of the set $\{B \leq G \ | \ A\leq B \ \ \text{and} \ \ B' \leq K \leq B \}$ and $H/K$ is cyclic. Moreover $e(G,H,K_1) = e(G,H,K_2)$ if and only if $K_1$ and $K_2$ are conjugate in $G$. 
\end{lemma}

As the construction of central idempotents from normal subgroups or Shoda pairs is not always possible or practical, we will sometimes need to work with the central idempotents coming from characters instead. We recall their construction:

\begin{theorem}\label{th:IdempotentsFromCharacters}\cite[Theorem 2.1.6]{LuxPahlings}
Let $F$ be a field of characteristic $0$ and $\chi$ the character of a simple $FG$-module $L$ with $D = \text{End}_{FG}(L)$. Then the primitive central idempotent of the Wedderburn component of $FG$ corresponding to $\chi$ is
\[\frac{\chi(1)}{[D:F]|G|}\sum_{g \in G} \chi(g^{-1})g. \]
\end{theorem}

In practice we will use the following lemma.

\begin{lemma}\label{lem:CoefficientOfProjectionByCharacter}
Let $M_k(D)$ be a simple component of the group algebra $FG$, for $F$ a field of characteristic $0$, with character $\chi$ and corresponding primitive central idempotent $e$. Moreover, let $n = \sum_{g \in G}\alpha(g)g$ be a generic element in $FG$. Then the coefficient of $ne$ at $g$ can be expressed in the two forms
\[\frac{k}{|G|} \sum_{h \in G} \alpha(gh^{-1}) \chi(h^{-1}) = \frac{k}{|G|} \sum_{h \in G} \alpha(h) \chi(g^{-1}h). \]
\end{lemma}
\begin{proof}
Note that $\text{dim}_F(eFG) = k^2[D:F]$ and $\chi(1) = k[D:F]$. So the primitive central idempotent $e$ corresponding to this component by \Cref{th:IdempotentsFromCharacters} has the form
\[e = \frac{\chi(1)}{[D:F]|G|}\sum_{g \in G} \chi(g^{-1})g = \frac{k}{|G|}\sum_{g \in G} \chi(g^{-1})g. \]
Hence for the product we have
\begin{align*}
ne &= \frac{k}{|G|}\sum_{g \in G} \sum_{h \in G} \alpha(g) \chi(h^{-1})gh \\
&= \frac{k}{|G|}\sum_{g \in G} \left(\sum_{h \in G} \alpha(gh^{-1}) \chi(h^{-1})\right)g = \frac{k}{|G|}\sum_{g \in G} \left(\sum_{h \in G} \alpha(h) \chi(g^{-1}h)\right)g.
\end{align*}
So the coefficient of $ne$ at $g$ is
\[\frac{k}{|G|} \sum_{h \in G} \alpha(gh^{-1}) \chi(h^{-1}) = \frac{k}{|G|} \sum_{h \in G} \alpha(h) \chi(g^{-1}h). \]
\end{proof}

\subsection{Nilpotent groups}

The goal of this subsection is to describe nilpotent groups with SN:

\begin{theorem}\label{prop:NilpotentSNNotSSN}
Let $G$ be a nilpotent group with SN. Then it has property SSN or $\Q G$ has at most one matrix component.
\end{theorem}

To start, one can quickly reduce to studying $p$-groups.

\begin{lemma}\label{lem:FromNilpotentToPGroup}
Let $G$ be a nilpotent group with SN which is not a Dedekind group. Then $G$ is a $p$-group.
\end{lemma}
\begin{proof}
Assume $P$ and $Q$ are a non-cyclic $p$-Sylow and a $q$-Sylow subgroup of $G$ respectively. Then $P \unlhd G$ and so by \Cref{lem:LiuPassma2.1} we know that $Q$ is a Dedekind group. But as this argument applies to every prime, this means that $G$ is Dedekind, contradicting the assumption.
\end{proof}

Recall that it was shown by Liu and Passman that $p$-groups with SSN coincide with the so-called NCN groups.

\begin{lemma}\label{lem:LiuPassma2.2}\cite[Proposition 2.2]{LiuPassman16}
Let $G$ be a $p$-group. Then $G$ has SSN if and only if every non-cyclic subgroup of $G$ is normal.
\end{lemma}

Now, the advantage of this is that non-normal subgroups of $p$-groups with SN are very restricted as the  following result shows. This is a variation of \cite[Lemma 3.2]{Liu} which includes $2$-groups.
\begin{lemma}\label{lem:OurLemma3.2}
Let $G$ be a group with SN and $Q$ a subgroup of $G$ which is not normal.
\begin{itemize}
\item[a)] If there exists $N \leq Q$ such that $N \unlhd G$, then $N$ is cyclic.
\item[b)] Now assume $G$ is a $p$-group. Then $Q$ is cyclic, elementary abelian or isomorphic to a quaternion group of order $8$. If moreover, $N \leq Q$ such that $N \unlhd G$ and $N \neq 1$, then $Q$ is cyclic or isomorphic to $Q_8$.
\end{itemize}
\end{lemma}

We will need the following well-known classical result.

\begin{lemma}\label{lem:OneCyclicSubgroup}\cite[III, Satz 8.2]{Huppert}
Let $G$ be a $p$-group which contains exactly one cyclic subgroup of order $p$. Then $G$ is cyclic or a generalized quaternion group.
\end{lemma}

\begin{proof}[Proof of \Cref{lem:OurLemma3.2}.]
Assume $N$ is a normal subgroup of $G$ contained in $Q$. If $N$ is not cyclic, then $G/N$ is Dedekind by \Cref{lem:LiuPassma2.1}, so that $Q/N \unlhd G/N$, which implies the contradiction $Q \unlhd G$. So part (a) follows.

The proof of part (b) is by two ``iterations''. First assume $N$ with the described properties exists. Then $Q$ is cyclic or a generalized quaternion group: indeed, if $n \in N$ has order $p$ and $q \in Q$ is an element of order $p$ not lying in $\langle n \rangle$, then $\langle n \rangle \langle q \rangle \unlhd G$, as $G$ is a group with SN. But this contradicts part (a) as $\langle n \rangle \langle q \rangle$ is not cyclic. So $Q$ contains exactly one subgroup of order $p$, implying $Q$ is cyclic or generalized quaternion by \Cref{lem:OneCyclicSubgroup}. Next we claim that, independently from the existence of $N$, the group $Q$ is cyclic, elementary abelian or generalized quaternion. For this assume $Q$ is maximal non-normal and let $M$ be a subgroup of $G$ containing $Q$ such that $[M:Q] = p$. By the maximality of $Q$ we get $M \unlhd G$ and so $\Phi(M) \unlhd G$, where $\Phi(M)$ denotes the Frattini subgroup of $M$. As $Q$ is a maximal subgroup of $M$, it contains $\Phi(M)$ and so $\Phi(M)$ is cyclic by part (a). If $\Phi(M) = 1$, then $Q$ is elementary abelian. If $\Phi(M) \neq 1$, then $Q$ is cyclic or generalized quaternion by the first claim proved in this paragraph.

It remains to show that in the two claims proven in the previous paragraph we can replace generalized quaternion groups by quaternion groups of order $8$. Assume first $N$ exists and $Q$ is a generalized quaternion with $n \in Q$ the unique involution. Then $Q/\langle n \rangle$ is a dihedral group of order $|Q|/2$. If $|Q|/2 \geq 8$, this implies, by the claim proven in the previous paragraph and the fact that the SN property is inherited by quotients, that $Q/\langle n \rangle \unlhd G/\langle n \rangle$. This would imply $Q \unlhd G$. So $|Q|/2 \leq 4$ which means that $Q$ is the quaternion group of order $8$. Now again ignore the existence of $N$, let $Q$ be again maximal non-normal, $M$ a normal subgroup of $G$ containing $Q$ such that $[M:Q] = p$ and assume that $Q$ is generalized quaternion. Then as before $\Phi(M) \unlhd G$ and $\Phi(M) \neq 1$, as $Q$ is not elementary abelian. It follows that the unique involution of $Q$ is central in $G$ and so the same argument as before can be used to show $|Q| = 8$.
\end{proof}

With these preparations we are ready to show that groups with SN but without SSN necessarily have one matrix component. We will separate two cases.

\begin{proposition}\label{prop:OddPGroups}
Let $G$ be a $p$-groups which has SN, but not SSN. Assume that either $p$ is odd, or $p=2$ and $G$ contains an elementary abelian subgroup $Q$ which is not normal. Then $G$ has at most one matrix component.
\end{proposition}
\begin{proof}
If $p=2$ let $Q$ be the elementary abelian subgroup of $G$ which is not normal. When $p$ is odd, by \Cref{lem:LiuPassma2.2}, $G$ contains a non-cyclic subgroup $Q$ which is not normal in $G$. Then $Q$ is elementary abelian by \Cref{lem:OurLemma3.2}. We choose $Q$ maximal with these properties, in particular for $Q \lneq M \leq G$ we have $M \unlhd G$. By \Cref{lem:OurLemma3.2} if $Q$ contains a subgroup $N$ which is normal in $G$, then $N = 1$ (note that this is trivial if $|Q| = 2$). In particular we have $\ZZ(G) \cap Q = 1$.

\textit{Claim 1:} $\ZZ(G)$ is cyclic. 

Assume first that $\ZZ(G)$ contains an elementary abelian subgroup $\langle z_1 \rangle \times \langle z_2 \rangle$. Then, by our choice of $Q$ and the fact that $\ZZ(G) \cap Q = 1$, we have $Q\langle z_1 \rangle \unlhd G$ and $Q\langle z_2 \rangle \unlhd G$. This implies that $[G,Q] \leq Q \langle z_1 \rangle$ and $[G,Q] \leq Q \langle z_2 \rangle$, respectively. So $[G,Q] \leq Q \langle z_1 \rangle \cap Q \langle z_2 \rangle = Q$, which would imply that $Q \unlhd G$. Hence $\ZZ(G)$ contains at most one subgroup of order $p$. As $\ZZ(G)$ cannot be generalized quaternion, because such a group is not abelian, the claim follows from \Cref{lem:OneCyclicSubgroup}.

We denote an element of order $p$ in $\ZZ(G)$ by $z$.

\textit{Claim 2:} $G' = \langle z \rangle$. Moreover, if $g\ \in G$ such that $z \notin \langle g \rangle$, then $g^p = 1$.

Let $g,h \in G$ such that $[g,h] \neq 1$. Assume first that $z \notin \langle g \rangle$. Note that we can assume this without changing the value of $[g,h]$ when $p$ is odd. As $G$ is a group with SN, this implies $\langle g, z \rangle = \langle g \rangle \times \langle z \rangle \unlhd G$. As $z$ is central and $G$ is a $p$-group we hence get $[G,g] \subseteq \langle g^p \rangle \times \langle z \rangle$. If $g^p \neq 1$, then we have $[G, g^p] \subseteq \langle g^{p^2} \rangle$, $[G, g^{p^2}] \subseteq \langle g^{p^3} \rangle$, ...,$[G, g^{\circ(g)/p}] = 1$, implying $g^{\circ(g)/p} \subseteq \ZZ(G)$ which contradicts Claim 1, as $z \notin \langle g \rangle$ by assumption. Hence $g^p = 1$ and $[G,g] \subseteq \langle z \rangle$, in particular $[g,h] \in \langle z \rangle$. Now suppose that $z$ lies in every non-trivial subgroup of $\langle g, h \rangle$, i.e. $\langle g,h \rangle$ is a generalized quaternion group. If $\langle g,h \rangle$ has order $8$, then $[g,h] = z$ and there is nothing more to prove. So assume $\langle g,h \rangle \cong Q_{2^m}$ for some $m \geq 4$. Say $g^{2^{m-1}} = h^4 = 1$ and $g^h = g^{-1}$. Note that $h^g = hg^2$. As $Q$ is elementary abelian and maximal non-normal, we get $\langle Q, h \rangle \unlhd G$. As also $Q  \times\langle z \rangle \unlhd G$, we have that $[h,q]$ has order at most $2$ and so $hq$ has order at most $4$ for every $q \in Q$. Note for this that $h^2 = z$ is central in $G$. But as $h^g = hg^2 \in \langle Q, h\rangle$, we have $g^2 \in \langle Q, h \rangle$, a contradiction, since $g^2$ has order at least $8$.

In particular Claim 2 implies that $G$ is metabelian. Moreover, there exists a cyclic subgroup $C$ of $G$ containing $z$ such that $A = C \times Q$ is a maximal abelian subgroup of $G$. We are now finally ready to prove that $G$ has at most one matrix component by applying \Cref{lem:IdempotForMetabelian}. So assume $\mathbb{Q}e(G,H,K)$ is a non-commutative component of $\mathbb{Q}G$. As $K$ lies in the kernel of a representation corresponding to this component, we have $G' \not\leq K$, i.e. $z \notin K$ by Claim 2. Hence, again by Claim 2, the unique maximal element of $\{B \leq G \ | \ A\leq B \ \ \text{and} \ \ B' \leq K \leq B \}$ is $A$ and $e(G,H,K) = e(G,A,K)$ with $A/K$ a cyclic group. By \Cref{lem:IdempotForMetabelian} it is thus sufficient to prove that all subgroups of $A$ which do not contain $z$ and have cyclic quotients are conjugate. We call such subgroups ``good''. Note that good subgroups are elementary abelian, so contained in $\langle z \rangle \times Q$. A good subgroup $U$ is determined by $(\langle z \rangle \times Q)/U$ and these quotients are exactly the images of the groups $\langle z q\rangle$, where $q$ runs through the elements of $Q$. Hence there are $|Q|$ good subgroups. It is clear that $Q$ itself is a good subgroup. So we need to prove that $Q$ has $|Q|$ conjugates in $G$, i.e. $[G:N_G(Q)] = |Q|$. By the maximality of $Q$, and since $z \in N_G(Q)$, we have $N_G(Q) \unlhd G$. As $G'$ is a central subgroup of order $p$, the group $G/N_G(Q)$ is elementary abelian. So we can view $V = Q \times G/N_G(Q)$ as an $\mathbb{F}_p$-vector space of dimension $|Q| + |G/N_G(Q)|$. We define a non-degenerate symplectic bilinear form
\[V \times V \rightarrow \langle z \rangle, \ \ (v,w) \mapsto [v,w]. \]
As no element of $G/N_G(Q)$ leaves all elements of $Q$ fixed under conjugation, $Q$ and $N_G(Q)$ are maximal isotropic subspaces of $V$, so that each of them has dimension $\frac{1}{2}\cdot(|Q| + |G/N_G(Q)|)$ by \cite[II, Satz 9.11]{Huppert}, i.e. $|G/N_G(Q)| = |Q|$. 
\end{proof}

By \Cref{lem:OurLemma3.2} it hence remains to study the case that $G$ is a 2-group and all the non-normal subgroups of $G$ are isomorphic to a quaternion group of order $8$. This turns out to be surprisingly hard. We would be very interested in an easier proof.

\begin{lemma}\label{lem:Q8timesQ8}
Let $G$ be a $2$-group which has SN, but not SSN. Then every involution of $G$ is central if and only if $G \cong Q_8 \times Q_8$.
\end{lemma}
\begin{proof}
Assume that every involution in $G$ is central. We first note that $G$ is not a generalized quaternion group. Indeed $Q_8$ and $Q_{16}$ have SSN \cite[Theorem 2.3, BJ6]{Liu} and if $G = \langle g, h \ | \ g^{2^n} = h^4 = 1, \ g^{2^{n-1}} = h^2, \ g^h = g^{-1} \rangle$ for some $n \geq 4$, then $G$ does not have SN. This can be observed by taking $N = \langle g^4 \rangle$, $Y = \langle h \rangle$, so that $N \not\subseteq Y$ and $NY \ntrianglelefteq G$ as $h^g = hg^2 \notin NY$. In particular, the center of $G$ is not cyclic.

By Lemmas~\ref{lem:LiuPassma2.2} and \ref{lem:OurLemma3.2} we can assume $G$ contains a non-normal subgroup $Q$ isomorphic to $Q_8$. We fix $a,b \in Q$ as generators of $Q$ and $c = a^2$. We will prove several small facts on $G$ which will lead to the proof of the lemma.
\begin{itemize}
\item[(i)] $\ZZ(G)$ has rank $2$:\\
$\ZZ(G)$ is not cyclic, as $G$ is not generalized quaternion. So assume the rank of $\ZZ(G)$ is bigger than $2$. Say $z_1, z_2 \in \ZZ(G)$ are independent elements of order $2$ such that $(\langle z_1 \rangle \times \langle z_2 \rangle) \cap Q = 1$. Then $Q \times \langle z_1 \rangle$ and $Q \times \langle z_2 \rangle$ are both normal subgroups of $G$. Hence $[G,Q] \leq Q\langle z_1 \rangle \cap Q\langle z_2 \rangle = Q$ which would imply $Q\unlhd G$.

\item[Convention:] We let $z \in G$ be an involution not lying in $Q$, so $\langle c \rangle \times \langle z \rangle$ is the unique maximal elementary abelian subgroup of $G$.

\item[(ii)] $|G/\Phi(G)| \leq 16$, i.e. $G$ is at most 4-generated:\\
This follows from (i) using \cite[Four Generator Theorem]{MacWilliams}.

\item[(iii)] The groups $\langle a \rangle$, $\langle b \rangle$ and $\langle ab \rangle$ are not normal in $G$:\\
Say $\langle a \rangle \unlhd G$. As $\langle a \rangle \not\subseteq \langle b \rangle$ and $G$ has SN this implies $\langle a \rangle \langle b \rangle = Q \unlhd G$, a contradiction. Similarly $\langle b \rangle$ and $\langle ab \rangle$ are not normal in $G$.

\item[(iv)] $a^G = \{a, a^{-1}, az, a^{-1}z \}$, $b^G = \{b, b^{-1}, bz, b^{-1}z \}$ and $(ab)^G = \{ab, (ab)^{-1}, abz, (ab)^{-1}z \}$:\\
As $G$ has SN we have $\langle a \rangle \times \langle z \rangle \unlhd G$. So by (iii) there is $g \in G$ such that $a^g = az$ or $a^g = a^{-1}z$. As $a^b = a^{-1}$ and $(az)^b = a^{-1}z$ the claim for $a^G$ follows. Similarly the conjugacy classes of $b$ and $ab$ follow from (iii).

\item[(v)] For $g \in G$ we have $g^2 \in C_G(Q)$:\\
By (iv) we have $[g,a] \in \{1,c, z, cz \}$, in any case a central element of order at most $2$. So $[g^2, a] = [g,a]^g[g,a] = 1$. Of course for $b$ and $ab$ we similarly have $[g^2,b] = [g^2, ab] =1$.

\item[(vi)] If $g \in C_G(Q)$ and $g \notin \langle c \rangle$, then $c \notin \langle g \rangle$:\\
This is clear if $g$ has order $2$. So assume the order of $g$ is $2^n$ for $n \geq 2$ and such that $g^{2^{n-1}} = c$. Then $(g^{2^{n-2}}a)^2 = c\cdot c = 1$, so $g^{2^{n-2}}a$ is an involution. Here we used that $g \in C_G(Q)$. As $g^b = g$ we have $(g^{2^{n-2}}a)^b = g^{2^{n-2}}a^{-1}$, so this involution is not central, contradicting the assumptions on $G$.

\item[(vii)] If $g \notin N_G(Q)$, but $g$ is centralizing $a$, $b$ or $ab$, then $c \notin \langle g \rangle$:\\
Say $g \in C_G(a)$, $\circ(g) = 2^n$ and assume $c \in \langle g \rangle$. As $g\notin N_G(Q)$, we must have $g \notin N_G(\langle b \rangle)$, so $b^g = b^{\pm 1}z$ by (iv). Note that $g^2 \in C_G(Q)$ by (v). So $g^{2^{n-2}}a$ is an involution with $(g^{2^{n-2}}a)^b = g^{2^{n-2}}ac$, if $n > 2$, or $(g^{2^{n-2}}a)^b \in \{gaz, gacz \}$, if $n=2$. In any case we would have a non-central involution.

\item[(viii)] If $g \notin N_G(\langle a \rangle)$ and $g \notin N_G(\langle b \rangle)$, then $g \in N_G(\langle ab \rangle)$. The same holds for every permutation of $a$, $b$ and $ab$:\\
If $g \notin N_G(\langle a \rangle)$ and $g \notin N_G(\langle b \rangle)$, then $a^g = a^{\pm 1}z$ and $b^g = b^{\pm 1}z$, so that
\[(ab)^g = a^{\pm 1}b^{\pm 1} z z \in \{ab, a^{-1}b, ab^{-1}, a^{-1}b^{-1} \}. \]
Noting that $a^{-1}b = ab^{-1} = (ab)^{-1}$ and $a^{-1}b^{-1} = ab$, the claim follows.

\item[(ix)] $N_G(Q) \unlhd G$ and $G/N_G(Q) \cong C_2 \times C_2$:\\
$N_G(Q)\unlhd G$ follows, as $N_G(Q)$ contains $z$ and is thus bigger than $Q$. By (v) the group $G/N_G(Q)$ is elementary abelian. If $G/N_G(Q) \cong C_2$ would hold, then by (viii) one of $\langle a \rangle$, $\langle b \rangle$ or $\langle ab \rangle$ would be normal in $G$, which would contradict (iii). On the other hand, if $g,h \notin N_G(\langle a \rangle)$, then $gh \in N_G(\langle a \rangle)$ by (iv). This implies that $G$ has only three non-trivial ways to act on the cyclic subgroups of the normal subgroup $Q \times \langle z \rangle$, implying $|G/N_G(Q)| \leq 4$.

\item[Convention:] By (viii) and (ix) we can choose $x,y \in G$ such that $\circ(x) \geq \circ(y)$ and $x \notin N_G(\langle a \rangle)$, $x \in N_G(\langle b \rangle)$ as well as $y \in N_G(\langle a \rangle)$, $y \notin N_G(\langle b \rangle)$. To assure the condition $\circ(x) \geq \circ(y)$ we might have to rename the elements $a$ and $b$.

\item[(x)] $C_G(Q) = \Phi(G)$ and $\{a,b,x,y \}$ is a minimal generating set of $G$:\\
First note that if $g \in G$, then $Q^g \leq Q \langle z \rangle$ by (iv). As $z$ is central, this implies that for every $h \in C_G(Q)$, the element $h$ is also centralizing $Q^g$. Hence $h^g \in C_G(Q)$, implying that $C_G(Q)$ is normal in $G$.
Now, by the action of $a$, $b$, $x$ and $y$ on $Q \times \langle z \rangle$ we see that no element of the form $a^\alpha b^\beta x^\gamma y^\delta$ with at least one of the $\alpha$, $\beta$, $\gamma$ and $\delta$ odd is centralizing $Q$. Hence the images of $a$, $b$, $x$ and $y$ in $G/C_G(Q)$ generate an elementary abelian subgroup of order $16$. By (ii) this is a maximal elementary abelian quotient of $G$, so the well-known properties of Frattini subgroups of $p$-groups, as recorded for instance in \cite[III, Section 3]{Huppert}, imply the claim.

\item[(xi)] $G^2 = \Phi(G)$. Moreover for any $g,h \in G$ we have $[g,h]^g = [g,h]$ and $(gh)^2 = g^2h^2[h,g]$:\\
The equation $G^2 = \Phi(G)$ holds in every $2$-group \cite[III, Satz 3.14(b)]{Huppert}. Let $i \in G$ be an involution so that $i \notin \langle g \rangle$. Hence $\langle g \rangle \times \langle i \rangle \unlhd G$ and so $[g,h] \in \langle g^2, i \rangle$, which implies $[g,h]^g = [g,h]$. Moreover this gives $(gh)^2 = g^2h[h,g]h = g^2h^2[h,g]$.

\item[(xii)] For $g,h \in G$ we have $[g^2,h] = [g,h]^2$ and $[g^2,h] \in \langle g^4 \rangle \cap \langle h^4 \rangle$. Furthermore, $\langle g^2 \rangle \unlhd G$:\\
In general $[g^2,h] = [g,h]^g[g,h]$, so $[g^2,h] = [g,h]^2$ holds by (xi). Moreover, if $i,j \in G$ are involutions such that $i \notin \langle g \rangle$ and $j \notin \langle h \rangle$, then $\langle g \rangle \times \langle i \rangle$ and $\langle h \rangle \times \langle j \rangle$ are normal subgroups of $G$, so that $[g,h] \in \langle g^2, i \rangle$ and $[g,h] \in \langle h^2, j \rangle$. So $[g^2,h] = [g,h]^2 \in \langle g^4 \rangle \cap \langle h^4 \rangle$. Finally, as $\langle g \rangle \times \langle i \rangle \unlhd G$ we have $[g,G] \subseteq \langle g^2, i\rangle$, so that $[g^2,G] \subseteq \langle g^4 \rangle$ by the previous, implying that $\langle g^2 \rangle \unlhd G$.

\item[(xiii)] If $g,h \in C_G(Q)$ and both have order at least $4$, then $\langle g \rangle \cap \langle h \rangle \neq 1$:\\
Say $\langle g \rangle \cap \langle h \rangle =1$. By (vi) we know that $c$ is not contained in $\langle g \rangle$ or $\langle h \rangle$. So we can assume $z \in \langle g \rangle$ and $cz \in \langle h \rangle$. Assume first that $\circ(g) = 2^n \geq 8$ and say $\circ(h) = 2^m$. By (xii) and the assumption $\langle g \rangle \cap \langle h \rangle = 1$ we have $[g^2,h] = 1$, so that $(g^{2^{n-2}}h^{2^{m-2}})^2 = zzc = c$. Hence $g^{2^{n-2}}h^{2^{m-2}}$ is an element of order $4$ in $C_G(Q)$ squaring to $c$, contradicting (vi). Now assume $g$ an $h$ are both of order $4$. As $g \in C_G(Q)$ and $G^2 = \Phi(G) = C_G(Q)$ by (x) and (xi) there are $g_1,...,g_k \in G$ such that $g = g_1^2g_2^2...g_k^2$. By the general commutator formulas and (xii) we get
\[[g,h] = [g_1^2...g_k^2,h] = ([g_1,h]^2)^{g_2^2...g_k^2}([g_2,h]^2)^{g_3^2...g_k^2}...[g_k,h]^2. \]
As $[g_i,h]^2 \in \langle h^4 \rangle = 1$ for all $i$ by (xii) we get $[g,h] = 1$ and so $(gh)^2 = g^2h^2 = zcz = c$, again contradicting (vi).

\item[Convention:] In case $C_G(Q)$ contains an element $g$ of order $4$, we set $z = g^2$. By (vi) and (xiii) this is well-defined. If there is no such element, we just keep the $z$ from before.

\item[(xiv)] There is $\tilde{z} \in C_G(Q)$ such that $z \in \langle \tilde{z} \rangle$ and $C_G(Q) = \langle c \rangle \times \langle \tilde{z} \rangle$:\\
Let $g,h\in C_G(Q)$. As $C_G(Q) = G^2$ by (x) and (xi) as in the proof of (xiii) we have $[g,h] \in \langle h^4 \rangle$, where we also need that $\langle h^2 \rangle \unlhd G$ by (xii). Hence $[C_G(Q), C_G(Q)] \leq C_G(Q)^4$, so that $C_G(Q)$ is a \emph{powerful $2$-group} (cf. the definition in \cite[I, Definition 2.1]{DdSMS}). Hence $C_G(Q)^2 = \{g^2 \ | \ g \in C_G(Q) \}$ \cite[I, Proposition 2.6]{DdSMS}. By (vi) this implies $c \notin \Phi(C_G(Q)) = C_G(Q)^2$, hence $\langle c \rangle$ is a direct factor of $C_G(Q)$ and we have $C_G(Q) = \langle c \rangle \times H$ for a subgroup $H$ containing only the involution $z$. Then $H$ is cyclic or generalized quaternion by \Cref{lem:OneCyclicSubgroup}, but as quaternion groups are not powerful, $H$ must be cyclic and the claim follows.

\item[(xv)] For $g,h \in G$ we have $[g^2, h^2] = 1$:\\
By (v) we know $g^2, h^2 \in C_G(Q)$ which is an abelian group by (xiv).

\item[(xvi)] Without breaking the conventions we can assume that $y$ has order $4$, $a^y = a$ and $z \in \langle y \rangle$:\\
We first aim to replace $y$ by an element of order $4$. By (xiv) we know $x^4, y^4 \in \langle \tilde{z} \rangle$, so, as by convention $\circ(x) \geq \circ(y)$, there is $\ell \in \mathbb{Z}$ such that $x^{4\ell}y^4 = 1$. If $x$ has order $4$, then also $y$. So assume $x$ has order at least $8$. Note that as $x^2, y^2 \in \langle c \rangle \times \langle \tilde{z} \rangle$ by (v) and (xiv) we have $y^2 \in \langle x^2, c \rangle$, so that $[x,y^2] = 1$. Hence by (xi) and (xii)
\[(x^\ell y)^4 = (x^{2\ell}y^2[y,x^\ell])^2 = x^{4 \ell}y^4[y^2,x^\ell] = x^{4\ell}y^4 = 1. \]
Note that by the defining properties of $x$ and $y$ we have $x^\ell y \notin N_G(\langle b \rangle)$, so $x^\ell y$ can not be an involution.
So, $x^\ell y$ has order $4$ and we replace $y$ by $x^\ell y$, where we replace $a$ by $ab$ if $\ell$ is odd, so that the convention is kept. In case with the new $y$, we have $a^y \neq a$, we replace $y$ by $by$. Finally, if $z \notin \langle y \rangle$, then $cz \in \langle y \rangle$ as $c \in \langle y \rangle$ is impossible by (vii). Then $ay$ satisfies all the conventions and moreover $(ay)^2 = a^2y^2 = ccz = z$, so that we choose $ay$ as the new $y$.

\item[Convention:] We choose $y$ as described in (xvi).

\item[(xvii)] $\circ(x) = 4$:\\
Assume $\circ(x) = 2^n > 4$. Then by (v) and (xiv) we have $x^4 \in \langle \tilde{z} \rangle$ and so $z \in \langle x^4 \rangle$. As $[x^2,y] \in \langle y^4 \rangle = 1$ by (xii) we then get $(x^{2^{n-2}}y)^2 = zz=1$ and $x^{2^{n-2}}y$ is an involution. But as it is not centralizing $b$ this gives a contradiction.

\item[(xviii)] $\tilde{z} = z$:\\
Assume $\circ(\tilde{z}) > 2$ holds. First note that as $x$ and $y$ have order $4$ by (xvi) and (xvii), the defining properties of $x$ and $y$ then imply that $\tilde{z}$ is not a power of $x$ or $y$. Moreover, again as $x$ has order $4$, we have $x^2 \in Z(G)$ and so $1 = [x^2,y] = [x,y]^2$ by (xii). In particular, $[x,y]$ has order at most $2$, implying that it is an element of the maximal elementary abelian subgroup $\langle c \rangle \times \langle z \rangle$ which is contained in $\langle c \rangle \times \langle \tilde{z}^2 \rangle$. So, $G/(\langle c \rangle \times \langle \tilde{z}^2 \rangle)$ is an elementary abelian group where the images of $a$, $b$, $x$, $y$ and $\tilde{z}$ are independent elements. To see the this use again the fact that $\circ(x) = \circ(y) = 4$. But this cannot be true, as $G$ is $4$-generated by (x).

\item[(xix)] We can assume $b^x = b$ and $x^2 = z$ without breaking the conventions:\\
As $x \in N_G(\langle b \rangle)$, we have $b^x = b^{\pm 1}$. If $b^x = b^{-1}$, we can replace $x$ by $ax$. Next, $x^2 = c$ is not possible by (vii). So if $x^2 \neq z$, we must have $x^2 = cz$ (note that $\circ(x) = 4$ by (xvii)). If this is the case we replace $x$ by $bx$ noting that $(bx)^2 = ccz = z$.

\item[Convention:] We choose $x$ as described in (xix).

\item[(xx)] $a^x = az$ and $b^y = bz$:\\
Assume $a^x \neq az$. Then $a^x = a^{-1}z$ by the properties of $x$ and (iv). Then using (xi) we get $(ax)^2 = a^2x^2[x,a] = czcz = 1$, so that $ax$ would be an involution not centralizing $b$. Similarly, if $b^y \neq bz$, then $by$ would be an involution not centralizing $a$.

\item[Relations] We summarize the obtained relations:
\begin{align*}
x^4 &= y^4 = 1, \ \ x^2 = y^2 = z, \\
a^x &= az, \ \ b^x = b, \ \ a^y = a, \ \ b^y = bz.
\end{align*}
We also note $(ab)^x = abz$ and $(ab)^y = abz$.

\item[(xxi)] $\langle x, y \rangle$ and $\langle bx, ay \rangle$ are normal subgroups of $G$ isomorphic to $Q_8$:\\
By the relations already obtained for $x$ and $y$ to show that $\langle x,y \rangle \cong Q_8$ it suffices to show $[x,y] = z$. As $x$ and $y$ have order $4$ the commutator $[x,y]$ has order at most $2$ by (xii) and the fact that involutions are central. We consider the other possible values for $[x,y]$. If $[x,y]=1$, then $xy$ is an involution no centralizing $b$. If $[x,y] = cz$, then by (xi) we get $(axy)^2 = a^2 (xy)^2 [xy,a] = c (x^2y^2cz)z = 1$, so that $axy$ would be an involution not centralizing $a$. Similarly, if $[x,y] = c$, then $abxy$ is an involution not centralizing $a$. Hence $[x,y] = z$ and $\langle x,y \rangle \cong Q_8$. We next observe $\langle bx, ay \rangle \cong Q_8$. This follows from calculating $(bx)^2 = (ay)^2 = cz$ as well as $(bx)^{ay} = bxcz = (bx)^{-1}$. It remains to show that both these subgroups are normal, but as $\{a,b,x,y \}$ is a generating set of $G$ by (x), it is sufficient to consider their conjugates under these four elements. A direct calculation using the relations above then gives the claim.

\item[(xxii)] $G = \langle x,y \rangle \times \langle bx, ay \rangle \cong Q_8 \times Q_8$:\\
By (xxi) both groups $\langle x,y \rangle$ and $\langle bx, ay \rangle$ are normal subgroups of $G$ and isomorphic to $Q_8$. As they have trivial intersection, $\langle x,y \rangle \times \langle bx, ay \rangle$ is a subgroup of $G$. This subgroup contains $a$, $b$, $x$ and $y$ which is a generating set of $G$ by (x).
\end{itemize} 

Finally, we show that if $G \cong Q_8 \times Q_8$, then $G$ has SN but not SSN using the notation for the elements of $G$ as in the $Q_8 \times Q_8$ we just found. First, as $a^x = az$ the subgroup $\langle a, b \rangle$ is not normal in $G$ and $G$ does not have SSN by \Cref{lem:LiuPassma2.2}. Next, note that every subgroup containing the three non-trivial involutions is normal in $G$, as $G' = \langle c \rangle \times \langle z \rangle$. Moreover, it is easy to see that a cyclic subgroup $Y$ of order $4$ is non-normal in $G$ if and only if $Y^2 = \langle c \rangle$. Hence, a general subgroup $Y$ is non-normal if and only if $\Phi(Y) = \langle c \rangle$ and $Y$ is either cyclic of order $4$ or a quaternion group of order $8$. Hence, for every normal subgroup $N$ and non-normal subgroup $Y$ the relation $N \not\leq Y$ implies that $N$ contains an involution different from $c$, giving $NY \unlhd G$. Overall, $G$ has SN. 
\end{proof}

The main result of this subsection now follows easily.
\begin{proof}[Proof of \Cref{prop:NilpotentSNNotSSN}.]
Let $G$ be a nilpotent group which has SN but not SSN. By \Cref{lem:FromNilpotentToPGroup} $G$ is a $p$-group. If $p$ is odd, or $p=2$ and $G$ contains a non-central involution, then the result is contained in \Cref{prop:OddPGroups}. Finally consider the case $p=2$ and that all involutions of $G$ are central which only applies to the group $Q_8 \times Q_8$ by \Cref{lem:Q8timesQ8}.  As $\ZZ(G)$ is not cyclic, $G$ has no faithful irreducible representations. It hence suffices to consider the components of the maximal quotients of $G$. Say $c$ and $z$ are the involutions of the direct factors. Then $G/\langle c \rangle$, $G/\langle z \rangle$ and $G/\langle cz \rangle$ are the maximal quotients. The first two of those are isomorphic to $Q_8 \times C_2 \times C_2$ which is a Hamiltonian group not contributing matrix components to $\mathbb{Q}G$. Finally, we show that the group $G/\langle cz \rangle$ has SN, but not SSN, and contains non-central involutions, so that the result then follows from \Cref{prop:OddPGroups}. To see this say $a$ is an element of order $4$ in the first direct factor and $x$ an element of order $4$ in the second. Then $(ax)^2 = cz$, so that $ax$ is mapped to a non-central involution when mapping to $G/\langle cz \rangle$. To see that this quotient does not have SSN consider the subgroup generated by the first direct factor and $x$: then $\langle x \rangle$ is mapped to a normal subgroup not contained in the image of $\langle a \rangle$, but $\langle a,x  \rangle$ is not mapped to a normal subgroup. 
\end{proof}

\begin{remark}
An alternative proof of \Cref{lem:Q8timesQ8} could be derived from the classification of $2$-groups all of whose non-normal subgroups are cyclic, elementary abelian of rank 2 or quaternion of order $8$ in \cite[Section 175]{BerkovichJankoVol4}. The group $Q_8 \times Q_8$ is one of those, but one would need to exclude the other groups appearing.
\end{remark}

\subsection{Non-nilpotent groups}

We will start with the case that $G$ is solvable. For this we need to introduce the following class of groups which will in \Cref{Section ND} distinguish themselves by being the only finite groups for which we cannot determine the equivalence between property ND and having at most one matrix component. Recall,

\begin{definition}\label{Defintion solv SSN unfaithful type}
Let $G$ be a group whose order is divisible by exactly two different primes $p$ and $q$ and let $P \in \text{Syl}_p(G)$ and $Q \in \text{Syl}_q(G)$. Assume $P$ and $Q$ are both cyclic, $P$ has order $p$ and $G = P \rtimes Q$ such that $Q$ acts non-trivially but also non-faithfully on $Q$. Then we call $G$ an \textit{SSN group of unfaithful type}.
\end{definition}
We note that the name in the previous definition is justified by \cite[Theorem 2.7]{LiuPassman16}.

When we speak of the \textit{rank} of a $p$-group $P$ we will mean the minimal number of generators of a maximal elementary abelian subgroup of $G$.

\begin{proposition}\label{prop:SolvSNGroups}
$G$ is a solvable non-nilpotent group with SN if and only if the following holds: $G$ contains a normal elementary abelian Sylow $p$-subgroup $P$ and a $p'$-Hall subgroup $H$ which is Dedekind. Each Sylow subgroup of $H$ has rank 1 and if $P$ has rank 1, then $H$ is cyclic. Moreover, 
\begin{itemize}
\item[(i)] either $G$ is an SSN group of unfaithful type
\item[(ii)] or the action of $H$ on $P$ is irreducible and faithful. In this case also no non-trivial element of $H$ is centralizing a non-trivial element of $P$.
\end{itemize}
\end{proposition}

\begin{proof}
We first show that $G$ being a solvable and non-nilpotent group with SN implies the described properties. 
Assume first that $G$ contains a normal elementary abelian subgroup of rank at least $2$ for some prime $p$. Then by \Cref{lem:LP2.5} we know that $P \unlhd G$ for $P \in \text{Syl}_p(G)$ and $G$ contains a nilpotent $p'$-Hall subgroup $H$. By \Cref{lem:LP2.4} the action of $H$ on $P$ is irreducible and faithful and $P$ is elementary abelian. It remains to show that the Sylow subgroups of $H$ all have rank 1. The action of $H$ on $P$ corresponds to a faithful and irreducible representation of $H$ over $\mathbb{F}_p$. Let $\chi$ be the character of this representation, $M$ the $\mathbb{F}_pG$-module and $F$ a field extension of $\mathbb{F}_p$ which is a splitting field for $H$. Then by \cite[Theorem 9.21]{Isaacs} the character of the module $F \otimes_{\mathbb{F}_p} M$ is a sum of certain Galois-conjugate characters of an irreducible $F$-character $\eta$ of $H$. If $H$ contains an elementary-abelian subgroup $Q$ of rank at least 2, then by the structure of Dedekind groups, $Q$ is central in $H$ and $\eta$ has a non-trivial kernel on $Q$. But this kernel is then also contained in the kernel of $\chi$, contradicting the fact that the action of $H$ on $P$ is faithful. Note also that if $D$ is a representation corresponding to $\eta$ and $h \in H \setminus \{ 1 \}$, then $D(h)$ has no eigenvalue 1. This follows again from the structure of Dedekind groups, as this is true for faithful characters of the quaternion group of order $8$ and cyclic groups. Hence there is no $g \in P \setminus \{1\}$ such that $g^h = g$.

We can hence assume that every elementary abelian normal subgroup of $G$ has rank 1. Let $P_0$ be such a normal $p$-subgroup, so $P_0$ is a cyclic group of order $p$, and let $H$ be a nilpotent $p'$-Hall subgroup of $G$ which exists by \Cref{lem:LP2.5}. We first consider the case that some Sylow subgroup of $H$ acts trivially on $P_0$. Let $Q \in \text{Syl}_q(G)$ be a such a Sylow subgroup, i.e. $[P_0, Q] = 1$. Then $P_0 \times Q \unlhd G$, as $G$ has SN, and so $Q \unlhd G$, as it is characteristic in $P_0 \times Q$. So by \Cref{lem:LP2.5} we have $G = QR$ for $R$ a $q'$-Hall subgroup of $G$. Moreover, the action of $R$ on $Q$ is not faithful, as $P_0$ acts trivially on $Q$. We conclude by \Cref{lem:LP2.4} that $G$ is an SSN group of unfaithful type.

So we can assume that every Sylow subgroup of $H$ acts non-trivially on $P_0$. Let $Q \in \text{Syl}_q(H)$. We first show that the rank of $Q$ is 1. Assume it is not and that $\langle x \rangle \times \langle y \rangle$ is an elementary abelian group of rank $2$ contained in $Q$. As $\text{Aut}(P_0)$ is cyclic, some element of $\langle x \rangle \times \langle y \rangle$ must act trivially on $P_0$, say this is $x$. Then $P_0 \times \langle x \rangle \unlhd G$, as $G$ has SN and so $\langle x \rangle \unlhd G$, as $\langle x \rangle$ is a characteristic subgroup of $P_0 \times \langle x \rangle$. Hence, again using the SN property of $G$, also $\langle x \rangle \times \langle y \rangle \unlhd G$, but this contradicts our assumption that $G$ contains no normal elementary abelian subgroup of rank at least $2$. Hence $Q$ has rank $1$. So, every Sylow subgroup of $H$ is cyclic or generalized quaternion by \Cref{lem:OneCyclicSubgroup}. 

We show that $H$ contains no quaternion group. Indeed, assume $Q = \langle g,h \ | \ g^{2^n} = g^4 = 1, g^{2^{n-1}} = h^2, g^h = g^{-1} \rangle$ is a subgroup of $H$ for some $n \geq 2$. As $\text{Aut}(P_0)$ is cyclic, some element of order $4$ in $Q$ must act trivially on $P_0$. As in the previous paragraph, this element must generate a normal subgroup of $G$. When $n \geq 3$ the only normal subgroup of $Q$ of order $4$ is $\langle g^{2^{n-2}} \rangle$, so it must act trivially. When $n=2$ we can assume this without loss of generality. 
Now $G/\langle g^{2^{n-2}}, h\rangle$ is a Dedekind group, so that the image of $Q$ in this quotient acts trivially on the image of $P_0$. Hence $g$ acts trivially on $P_0$ and $h$ non-trivially.
As before we get then $\langle g \rangle \unlhd G$ and the SN property implies $\langle g \rangle \langle h \rangle = Q \unlhd G$. But this cannot be as $h$ acts non-trivially on $P_0$. We conclude that all the Sylow subgroups of $H$ are cyclic.

We now show that under all the assumptions $G$ has a normal Sylow $p$-subgroup. Let $\{q_1,q_2,...,q_k \}$ be the prime divisors of $|H|$ with $q_1 < q_2 < ... < q_k$. Note that as a Sylow $q_i$-subgroup of $H$ acts non-trivially on $P_0$ we have $q_i \mid (p-1)$ and so $q_i < p$ for each $i$. By successively applying the famous corollary of Burnside's $p$-complement theorem on cyclic Sylow subgroups for minimal primes \cite[IV, Satz 2.8]{Huppert}, we obtain that $G$ contains a normal $q_1$-complement $H_1$, which contains a normal $q_2$-complement $H_2$,..., which contains a normal $q_k$-complement $P$ which must be a Sylow $p$-subgroup of $G$. Note that in each step the normal complement found is characteristic, so that all these groups are also normal in $G$, in particular $P$. So $G = P \rtimes H$. It follows from \Cref{lem:LP2.4} that the structure of $G$ is as claimed. Moreover, as all the Sylow subgroups of $G$ are cyclic and the action of each Sylow subgroup of $H$ on $P$ is now faithful we get $[g,h] \neq 1$ for all $h \in H \setminus \{1\}$ and $g \in P \setminus \{1\}$.

Finally, we also show that the described groups are groups with SN. This is clear for the SSN groups of unfaithful type, as these have even SSN by \cite[Theorem 2.7]{LiuPassman16}. So assume $G$ is as described in (ii). As the action of $H$ on $P$ is irreducible and faithful, $P$ is the unique minimal normal subgroup of $G$. So if $N \unlhd G$ and $N \neq 1$, then $P \leq N$. Let moreover $Y \leq G$. If $N$ is not a subgroup of $Y$, then $NY/N \unlhd G/N$, as $G/N$ is a quotient of $H$ and hence a Dedekind group. This implies $NY \unlhd G$ and so $G$ indeed has SN. 
\end{proof}

Finally, using the methods from \cite[section 3]{LiuPassman16} we readily classify the non-solvable groups with SN.

\begin{proposition}\label{prop:NonSolvableSNGroups}
Let $G$ be a non-solvable group. Then $G$ has SN if and only if $G$ has a unique minimal normal subgroup $S$ such that $S$ is non-abelian and $G/S$ Dedekind. Moreover in that case $S = \Soc(G)$ is a direct product of isomorphic finite simple groups and if $S$ is simple, then $G$ is an almost simple group.
\end{proposition}
\begin{proof}
Let $G$ be a non-solvable group with SN. Recall that minimal normal subgroups are direct products of isomorphic simple groups, see \cite[(8.3)]{Aschbacher}. By \Cref{lem:LP2.5} we cannot have an abelian minimal normal subgroup. Hence every minimal normal subgroup is the direct product of non-abelian simple groups, the socle $S$ of $G$ is non-abelian and the Fitting subgroup trivial. In particular, $S$ equals the generalized Fitting subgroup which is the unique largest normal semisimple subgroup.

Next, write $S$ as the internal direct product $\prod_{i=1}^n A_i$ with $A_i$ a minimal normal subgroup of $G$. Using \Cref{lem:LiuPassma2.1}, we see that $G/A_1$
is Dedekind which is only possible if all $A_i=1$ for $i \neq 1$. Thus $S$ is the unique minimal normal subgroup. Moreover $S$ is the direct product of isomorphic non-abelian simple groups. Furthermore, as $S$ is non-abelian, \Cref{lem:LiuPassma2.1} also yields that $G /S$ is Dedekind.

Conversely, suppose $S \leq G $ is the unique minimal normal subgroup of $G$, that $S$ is non-abelian and $G/S$ is Dedekind. Now, every $N \unlhd G$ contains the unique minimal normal subgroup $S$. Hence if $Y \leq G$, then $S \leq NY$ and so $NY / S \leq G/S$ is normal. Therefore, $NY \unlhd G$ as needed.

Finally, suppose that $G$ has SN and $S = \Soc(G)$ is simple. As $S$ is also the generalized Fitting subgroup, it contains its own centralizer. In other words, $C_G(S) = \mathcal{Z}(S)$ is trivial and hence $G$ acts faithfully on $S$. Thus one may identify $G$ with a subgroup of $\Aut (S)$ with $S$ simple, i.e. $G$ is almost simple. 
\end{proof}

\section{Groups with the ND property and the Jespers-Sun conjecture}\label{Section ND}

Let $G$ be a finite group and $n \in \Z G$ nilpotent. Then $G$ {\it has ND} if $ne \in \Z G$ for every $n$ and every primitive central idempotent $e$ of $\Q G$. In this section we answer property ND for all finite groups which are not as in \Cref{Defintion solv SSN unfaithful type}. In particular for such groups we show that there is a unique counterexample to Jespers-Sun's \Cref{conj:EricWei}. 

\begin{theorem}\label{th:MainTheoremNDInSection3}
Let $G$ be a finite group which is not an SSN group of unfaithful type. Then $G$ has ND if and only if $\Q G$ has at most one matrix component or $G \cong \langle a,b \mid a^4=b^8 =1, a^b =a^{-1} \rangle$.
\end{theorem}

Any group with ND is necessarily a group with SN. This follows almost directly from the definition and is recorded in \cite[Proposition 2.5]{LiuPassman09}, where this follows from the proof, and more explicitly in \cite[Proposition 3.4]{JespersSun}. So to prove \Cref{th:MainTheoremNDInSection3} we will use \Cref{Summary theorem for SN}. Furthermore we will have to distinguish the case where $G$ is nilpotent or not. In particular the above result is the combination of \Cref{th:NilpotentCase} and \Cref{ND for non-nilp SSN grps}. 

One may now draw easily interesting consequences from \Cref{Summary theorem ND}. For example if $G$ is not metacyclic, then Jespers-Sun's conjecture is actually correct. In the philosophy ``what does a group ring $RG$ know about $G$?'' we can give a positive answer to a variation of the Jespers-Sun Conjecture:

\begin{corollary}\label{coro ND determined by QG}
Let $G$ and $H$ be groups such that $\mathbb{Q}G \cong \mathbb{Q}H$ and $G$ has ND. Then $H$ has $ND$.
\end{corollary} 
\begin{proof}
Assume first that $G$ has at most one matrix component. Then $\mathbb{Q}G \cong \mathbb{Q}H$ implies, that so has $H$ and so $H$ has ND. By \Cref{th:MainTheoremNDInSection3} it remains to consider the cases that $G$ is the non-abelian group $C_4 \rtimes C_8$ or an SSN group of unfaithful type. Assume first that $G \cong C_4 \rtimes C_8$. Then $G/G' \cong C_8 \times C_2 \cong H/H'$ \cite[Theorem 2.8]{AngelAngel}. It follows that if $\langle c \rangle = H'$, then either there is an element $h \in H$ of order $4$ or $16$ such that $c \in \langle h \rangle$ or there is no element at all squaring to $c$. As a generalized quaternion group of order $32$ has derived subgroup of order $8$, it follows that $\ZZ(G)$ has rank bigger than one and with the previous $H \cong G$ or it is one of the groups
\[H_1 = \langle a,b \ | \ a^{16} = b^2 =1, a^b = a^9 \rangle, \ \ H_2 = \langle a,b,c \ | \ a^8=b^2=c^2 = 1, a^b = ac, [a,c]=[b,c]=1 \rangle.\]
None of these groups maps onto a quaternion group, so both $\mathbb{Q}H_1$ and $\mathbb{Q}H_2$ do not have a simple component isomorphic to the rational quaternion algebra, while $\mathbb{Q}G$ does. We conclude $G \cong H$.

So assume $G$ is an SSN group of unfaithful type, say $G \cong C_p \rtimes C_{q^k}$ for some primes $p$ and $q$ and a positive integer $k$. Then $\mathbb{Q}G \cong \mathbb{Q}H$ implies $|H| = p\cdot q^k$. Moreover the maximal commutative direct summand of $\mathbb{Q}G$ is isomorphic to $\mathbb{Q}(G/G') \cong \mathbb{Q}C_{q^k}$. So $H/H' \cong C_{q^k}$, which implies also $H' \cong C_p$. Hence $H \cong C_p \rtimes  C_{q^k}$. To show that $G \cong H$ it remains to show that the action on the derived subgroup has the same order or equivalently $\ZZ(G) \cong \ZZ(H)$. It is easy to calculate that the number of conjugacy classes of cyclic subgroups of $G$ and $H$ is the the same if and only if $\ZZ(G) \cong \ZZ(H)$. As this number coincides with the number of simple components of a rational group algebra \cite[Corollary 7.1.12]{GRG1}, the result follows.
\end{proof}

\subsection{Background on describing simple components via Shoda pairs}\label{prelim for ND section}

As could be expected from the content of \Cref{conj:EricWei}, we need to recall some methods to construct primitive central idempotents of $\Q G$. These methods were introduced by Olivieri-del R\'io-Sim\'on \cite{OlRiSi}, see \cite[Chapter 3]{GRG1} for a good introduction. To start, recall that if $ H\unlhd G$, then $\wh{H}$ is a central idempotent in $\Q G$. Now, set $\epsilon(H,H) = \wh{H}$ and for a strict normal subgroup $K$ of $H$ define
\begin{equation}\label{Def epsilon idempotents}
\epsilon (H,K) =  \prod\limits_{M/K \in \mathcal{M}(H/K)} (\wh{K} - \wh{M}) = \wh{K} \prod_{M/K \in  \mathcal{M}(H/K)} (1 - \wh{M}),
\end{equation}
where $ \mathcal{M}(H/K)$ denotes the set of the non-trivial minimal normal subgroups of $H/K$. In both cases the construction results in a central idempotent in $\Q H$. 
Next, with $K \unlhd H$ one associates the element 
\begin{equation}\label{Def e(G,H,K)}
e(G,H,K) = \sum_{t \in \mathcal{T}} \epsilon(H,K)^{t},
\end{equation}
where $\mathcal{T}$ is a right transversal of $\Cen_G(\epsilon(H,K))$ in $G$. The element $e(G,H,K)$ is central in $\Q G$ and is a primitive idempotent when $(H,K)$ is a {\it Strong Shoda pair} of $G$. A tuple $(H,K)$ is called a strong Shoda pair when $K \leq H \unlhd N_G(K)$, $H/K$ is cyclic and a maximal abelian subgroup of $N_G(K)/K$, and the $G$-conjugates of $\epsilon(H,K)$ are orthogonal.

To a central idempotent $e$ we will also need the associated homomorphism 
\begin{equation}\label{map associated to central id}
\varphi_e: G \rightarrow Ge, \ \ g \mapsto ge.
\end{equation}

The following is a combination of \cite[Proposition 3.4.1, Theorems 3.4.2 \& 3.5.5 and Problem 3.5.1]{GRG1}.

\begin{theorem}[\cite{OlRiSi}]\label{form of SSP idempotent result}
With notations as above, $e(G,H,K)$ is a primitive central idempotent of $\Q G$ if $(H,K)$ is a strong Shoda pair. Moreover, in that case $\Cen_G(\epsilon(H,K)) \cong N_G(K)$ and $\ker(\varphi_{e(G,H,K)}) = \text{core}_G(K) = \bigcap_{g \in G} K^g$.
\end{theorem}

We also need the $\Q$-dimension of the simple algebra associated to a strong Shoda pair which directly follows from the known description of $\Q G e(G,H,K)$.

\begin{lemma}\label{dim of simple of SSP}
Let $(H,K)$ be a strong Shoda pair of $G$. Then
$$\dim_{\Q} \Q G e(G,H,K) =  [G:H] [G:N_G(K)] \phi([H:K]),$$
where $\phi(\cdot)$ denotes the phi-Euler function.
\end{lemma}
\begin{proof}
Following \cite[Theorem 3.5.5]{GRG1}, $\Q G e(G,H,K) \cong M_{[G : N_G(K)]}(\Q(\zeta_{[H:K]}) \star N_G(K)/H)$ for some crossing that can be made explicit (see \cite[Remark 3.5.6]{GRG1}). Therefore, one has that
$$\begin{array}{ll}
dim_{\Q} \Q G e(G,H,K) & = [G : N_G(K)] ^2 \phi([H:K]) [N_G(K) : H] \\
& = [G:H] [G:N_G(K)] \phi([H:K]).
\end{array}$$
\end{proof}

\subsection{Nilpotent case}\label{sec:NDForNilpotent}

In this section we completely solve \Cref{conj:EricWei} for nilpotent groups. It turns out that in this class the conjecture is almost true - there is exactly one counterexample. From the results of the previous section, we know that we need to consider the question only for nilpotent groups with SSN. For many of those Jespers and Sun did prove their conjecture \cite[Corollary 4.12]{JespersSun}, but as it turns out, quite some work remains. Overall in this section we get:

\begin{theorem}\label{th:NilpotentCase}
Let $G$ be a nilpotent group. Then $G$ has ND if and only if either $G$ has one matrix component or $G \cong \langle a,b \mid a^4=b^8 =1, a^b =a^{-1} \rangle \cong C_4 \rtimes C_8$. 
\end{theorem}

The identifier of the exception appearing in the theorem in the SmallGroupsLibrary \cite{SmallGroupLibrary} is $[32,12]$.

It turns out that from our results in the previous section and the previous work of others, mostly Liu and Jespers-Sun, there is one series of groups we need to address which we define now. For a prime $p$ and positive integers $n \geq 1$ and $m \geq 2$ define the group
\begin{equation}\label{definition BJ1 class}
G(p,m,n) = \langle a,b \ | \ 1 = a^{p^m} = b^{p^n}, a^b = a^{1+p^{m-1}} \rangle. 
\end{equation}
Note that the center of $G(p,m,n)$ is $\langle a^p \rangle \times \langle b^p \rangle$. When working with a group $G(p,m,n)$ we will always assume that it has generators and relations exactly as given in \eqref{definition BJ1 class}.

\begin{theorem}\label{th:counterex}
The group $G(2,2,3)$ has (ND), but $\mathbb{Q}G(2,2,3)$ has more than one matrix component. Consequently, \Cref{conj:EricWei} is not correct.
\end{theorem}

Before we proceed to prove the theorem we record an easy property of nilpotent $2\times 2$-matrices.

\begin{lemma}\label{lem:FormulaNilpotent2x2Matrices}
Assume $A = \begin{pmatrix} x & y \\ z & w \end{pmatrix}$ is a $2\times 2$-matrix over a commutative domain $R$. Then $A$ is nilpotent if and only if $x = -w$, $x^2 = -yz$. 
\end{lemma}
\begin{proof}
As $A$ is nilpotent, i.e. $A^n=0$ for some $n$, the multiplicativity of the determinant gives $\det(A)=0$, implying $xw = yz$. 
Hence,
\[A^2 = \begin{pmatrix} x^2 + yz & xy + yw \\ xz + zw & w^2 + yz \end{pmatrix} = \begin{pmatrix} x^2 + xw & y(x + w) \\ z(x + w) & w^2 + xw \end{pmatrix} = \text{tr}(A)A. \]
So, $0 = \text{tr}(A)^{n-1}A$, which gives $\text{tr}(A) = 0$. Hence, $x=-w$ and from $xw=yz$ we also get $x^2=-yz$.
\end{proof}

\begin{proof}[Proof of \Cref{th:counterex}]
Let
\[G(2,2,3) = G = \langle a, b \ | \ a^4 = b^8 = 1, \ a^b = a^{-1} \rangle. \]
We note that $G' = \langle a^2 \rangle$ and $\ZZ(G) = \langle a^2, b^2 \rangle$. Then $G/G' \cong C_2 \times C_8$, so the algebra $\mathbb{Q}G$ has a direct summand 
\[\mathbb{Q}[C_2 \times C_8] \cong 4\mathbb{Q} \oplus 2\mathbb{Q}(i) \oplus \mathbb{Q}(\zeta_8).\]
Moreover $G/\langle a^2b^2 \rangle \cong Q_8$ and $G/\langle b^2 \rangle \cong D_8$, so that $\mathbb{Q}G$ has also direct summands $\mathbb{H}_\mathbb{Q}$, the standard rational quaternions, and $M_2(\mathbb{Q})$. Moreover $G/\langle a^2b^4 \rangle$ is a group of order $16$ sometimes denoted by $D_{16}^+$. The rational group algebra of this group has one matrix component isomorphic to $M_2(\mathbb{Q}(i))$. We mention that it has been used in \cite{BachleMaheshwaryMargolis} to solve another problem on integral group rings. Overall
\[ \mathbb{Q}G \cong 4\mathbb{Q} \oplus 2\mathbb{Q}(i) \oplus \mathbb{Q}(\zeta_8) \oplus \mathbb{H}_\mathbb{Q} \oplus M_2(\mathbb{Q}) \oplus M_2(\mathbb{Q}(i)). \]
This can also be easily checked using GAP \cite{GAP} and the wedderga package therein \cite{wedderga}. 

Furthermore the following representations correspond to the non-commutative components (in the order as above):
\begin{align}\label{eq:G223Reps}
&G \rightarrow \mathbb{H}_\mathbb{Q}, \ \  a  \mapsto i, \ \ b \mapsto j, \nonumber \\
&G \rightarrow M_2(\mathbb{Q}), \ \ a  \mapsto \begin{pmatrix} 0 & -1 \\ 1 & 0 \end{pmatrix}, \ \ b \mapsto \begin{pmatrix} -1 & 0 \\ 0 & 1 \end{pmatrix}, \\
&G \rightarrow M_2(\mathbb{Q}(i)), \ \ a  \mapsto \begin{pmatrix} -i & 0 \\ 0 & i \end{pmatrix}, \ \ b \mapsto \begin{pmatrix} 0 & 1 \\ i & 0 \end{pmatrix} \nonumber.
\end{align}

Here we denote by $i$ and $j$ the standard generators of $\mathbb{H}_\mathbb{Q}$.

We first derive the properties which are equivalent to having a nilpotent element in $\mathbb{Z}G$. Let $n$ be a generic nilpotent element in $\mathbb{Z}G$ and write
\[n = \sum_{g \in G} \alpha(g) g. \]
In the quotient $G/G'$ the element $n$ must map to $0$ which is equivalent to the fact that for each $h \in G$ one has $\sum_{g \in G'} \alpha(hg) = 0$. This is in turn equivalent to
\begin{equation}\label{eq:gSignga^2}
 \alpha(g) = - \alpha(ga^2) \  \forall g \in G, 
\end{equation}
since $G' = \langle a^2 \rangle$. As in the three non-commutative components the element $a^2$ is always send to $-1$, this will always give a factor $2$ in the considerations below and reduces the number of indeterminants by half. So we will always replace an expression of shape $\alpha(g) - \alpha(ga^2)$ by $2\alpha(g)$.

Next we consider the $\mathbb{H}_\mathbb{Q}$-component where the projection of $n$ must equal $0$. With the representation above $n$ is sent to
\begin{align*}
2&(\alpha(1) - \alpha(b^2) + \alpha(b^4) - \alpha(b^6) \\
+&i(\alpha(a) - \alpha(ab^2) + \alpha(ab^4) - \alpha(ab^6)) \\
+&j(\alpha(b) - \alpha(b^3) + \alpha(b^5) - \alpha(b^7)) \\
+&ij(\alpha(ab) - \alpha(ab^3) + \alpha(ab^5) - \alpha(ab^7))).
\end{align*} 
Setting this equal to $0$ gives the equations
\begin{align}
\alpha(1) &= \alpha(b^2) + \alpha(b^6) - \alpha(b^4), \nonumber \\ 
\alpha(a) &= \alpha(ab^2) + \alpha(ab^6) - \alpha(ab^4), \nonumber \\ 
\alpha(b) &= \alpha(b^3) + \alpha(b^7) - \alpha(b^5), \label{eq:H} \\ 
\alpha(ab) &= \alpha(ab^3) + \alpha(ab^7) - \alpha(ab^5). \nonumber
\end{align}
Together with \eqref{eq:gSignga^2} this can be written compactly as 
\begin{align}\label{eq:gCoefModb4}
\alpha(g) + \alpha(gb^4) = \alpha(gb^2) + \alpha(gb^6) \ \forall g \in G.
\end{align}

We denote the representation of $n$ in the  $M_2(\mathbb{Q})$-component by $\begin{pmatrix} x_1 & y_1 \\ z_1 & w_1 \end{pmatrix}$ and in the $M_2(\mathbb{Q}(i))$-component by $\begin{pmatrix} x_2 & y_2 \\ z_2 & w_2 \end{pmatrix}$.

From the representation given above we get
\[x_1 = 2(\alpha(1) + \alpha(b^2) + \alpha(b^4) + \alpha(b^6) - \alpha(b) - \alpha(b^3) - \alpha(b^5) - \alpha(b^7)) \]
which using \eqref{eq:H} transforms to 
\begin{equation}\label{eq:x1}
x_1 = 4(\alpha(b^2) + \alpha(b^6) - \alpha(b^3) - \alpha(b^7)).
\end{equation}
Similarly we get
\begin{equation}\label{eq:w1}
w_1 = 4(\alpha(b^2) + \alpha(b^6) + \alpha(b^3) + \alpha(b^7)).
\end{equation}
As one of our conditions on $n$ is $x_1 = -w_1$ by \Cref{lem:FormulaNilpotent2x2Matrices} this gives, using also \eqref{eq:H},
\begin{equation}\label{eq:bgroup}
\alpha(b^2) = -\alpha(b^6), \ \alpha(1) = -\alpha(b^4).
\end{equation}
Furthermore we compute
\[y_1 = 2(-\alpha(a) - \alpha(ab^2) - \alpha(ab^4) - \alpha(ab^6) -\alpha(ab) - \alpha(ab^3) - \alpha(ab^5) - \alpha(ab^7)) \]
which we convert using \eqref{eq:H} to
\begin{equation}\label{eq:y1}
y_1 = 4(-\alpha(ab^2) - \alpha(ab^6) - \alpha(ab^3) - \alpha(ab^7)).
\end{equation}
Similarly
\begin{equation}\label{eq:z1}
z_1 = 4(\alpha(ab^2) + \alpha(ab^6) - \alpha(ab^3) - \alpha(ab^7)).
\end{equation}

We now calculate the $\mathbb{Q}(i)$-representation. We get
\begin{align*}
x_2 = & 2(\alpha(1) - \alpha(b^4) + \alpha(ab^2) - \alpha(ab^6) \\
& + i(-\alpha(a) + \alpha(ab^4) + \alpha(b^2) - \alpha(b^6))).
\end{align*}
Using \eqref{eq:H} and \eqref{eq:bgroup} this becomes
\begin{align*}
x_2 = & 2(-2\alpha(b^4) + \alpha(ab^2) - \alpha(ab^6)  \\
+ & i(2\alpha(ab^4) - 2\alpha(b^6) - \alpha(ab^2) - \alpha(ab^6))).
\end{align*}
Similarly,
\begin{align*}
w_2 = & 2(-2\alpha(b^4) - \alpha(ab^2) + \alpha(ab^6)  \\
+ & i(-2\alpha(ab^4) - 2\alpha(b^6) + \alpha(ab^2) + \alpha(ab^6))).
\end{align*}
With the condition $x_2 = -w_2$ from \Cref{lem:FormulaNilpotent2x2Matrices} this gives $\alpha(b^4) = 0 = \alpha(b^6)$ and together with \eqref{eq:bgroup} we conclude
\begin{equation}\label{eq:1eq0}
\alpha(1) = \alpha(b^2) = \alpha(b^4) = \alpha(b^6) = 0
\end{equation}
and
\begin{align}
x_2 = & 2(\alpha(ab^2) - \alpha(ab^6) \label{eq:x2}  \\
+ & i(2\alpha(ab^4)- \alpha(ab^2) - \alpha(ab^6))) \nonumber
\end{align}
as well as
\begin{align}
w_2 = & 2( - \alpha(ab^2) + \alpha(ab^6) \label{w2}  \\
+ & i(-2\alpha(ab^4) + \alpha(ab^2) + \alpha(ab^6))). \nonumber
\end{align}
We compute the other coefficients as
\begin{align*}
y_2 = & 2(\alpha(b) - \alpha(b^5) + \alpha(ab^3) - \alpha(ab^7) \\
& + i(\alpha(b^3) - \alpha(b^7) - \alpha(ab) + \alpha(ab^5))
\end{align*}
which by \eqref{eq:H} transforms to
\begin{align}
y_2 = & 2(-2\alpha(b^5) + \alpha(b^3) + \alpha(b^7) + \alpha(ab^3) - \alpha(ab^7) \label{eq:y2} \\
& + i(2\alpha(ab^5) + \alpha(b^3) - \alpha(b^7) -\alpha(ab^3) -\alpha(ab^7) )). \nonumber
\end{align}
Similarly, also by \eqref{eq:H},
\begin{align}
z_2 = & 2(2\alpha(ab^5) - \alpha(b^3) + \alpha(b^7) - \alpha(ab^3) - \alpha(ab^7) \label{eq:z2} \\
& + i(-2\alpha(b^5) + \alpha(b^3)  + \alpha(b^7) - \alpha(ab^3) + \alpha(ab^7) )). \nonumber
\end{align}
Moreover, from \eqref{eq:x1} and \eqref{eq:1eq0} we have
\begin{equation}\label{eq:x1new}
x_1 = -4(\alpha(b^3)+\alpha(b^7)).
\end{equation} 

We now compute the quadratic equations from \Cref{lem:FormulaNilpotent2x2Matrices}. Then
\begin{equation}\label{eq:x1^2}
x_1^2 = 16(\alpha(b^3)^2 + 2\alpha(b^3)\alpha(b^7) + \alpha(b^7)^2)
\end{equation}
and from \eqref{eq:y1} and \eqref{eq:z1} we get
\begin{equation}\label{eq:-y1z1}
-y_1z_1 = -16(-\alpha(ab^2)^2 - 2\alpha(ab^2)\alpha(ab^6) - \alpha(ab^6)^2 + \alpha(ab^3)^2 + 2\alpha(ab^3)\alpha(ab^7) + \alpha(ab^7)^2).
\end{equation}
The analogues equations for the $M_2(\mathbb{Q}(i))$-component give by \eqref{eq:x2}
\begin{align}
x_2^2 = & 8(-2\alpha(ab^4)^2 + 2\alpha(ab^2)\alpha(ab^4) + 2\alpha(ab^6)\alpha(ab^4) - 2\alpha(ab^2)\alpha(ab^6) \label{eq:x2^2} \\
& + i(2\alpha(ab^2)\alpha(ab^4) - 2\alpha(ab^6)\alpha(ab^4) - \alpha(ab^2)^2 + \alpha(ab^6)^2 )) \nonumber
\end{align}
and by \eqref{eq:y2} and \eqref{eq:z2}
\begin{align}
-y_2z_2 = & -8(2\alpha(b^3)\alpha(b^5) - 2\alpha(b^7)\alpha(b^5) - \alpha(b^3)^2 + \alpha(b^7)^2 \nonumber \\
 & +2\alpha(ab^3)\alpha(ab^5) -2 \alpha(ab^7)\alpha(ab^5) - \alpha(ab^3)^2 + \alpha(ab^7)^2 \label{eq:-y2z2} \\
& + 2i(\alpha(ab^5)^2 - \alpha(ab^3)\alpha(ab^5) - \alpha(ab^7)\alpha(ab^5) \nonumber \\
 & + \alpha(b^5)^2 - \alpha(b^3)\alpha(b^5) - \alpha(b^7)\alpha(b^5) \nonumber \\
& + \alpha(b^3)\alpha(b^7) + \alpha(ab^3)\alpha(ab^7) )). \nonumber
\end{align}

We now show certain congruences modulo $2$ which will provide the key for the final argument.
First note that the imaginary part of $-y_2z_2$ is divisible by $16$. So this is also true for the imaginary part of $x_2^2$ implying $-\alpha(ab^2)^2 + \alpha(ab^6)^2 \equiv 0 \bmod 2$ which means
\begin{equation}\label{eq:congab^2}
\alpha(ab^2) \equiv \alpha(ab^6) \mod 2.
\end{equation}
We next show that $\alpha(b^3) \equiv \alpha(b^7) \bmod 2$ and also $\alpha(ab^3) \equiv \alpha(ab^7) \bmod 2$. Assume that $\alpha(b^3) \not\equiv \alpha(b^7) \bmod 2$. Then one of them is even and the other is odd which implies by \eqref{eq:x1^2} that $\frac{x_1^2}{16} \equiv 1 \bmod 4$. Note that \eqref{eq:congab^2} implies that 
\[\alpha(ab^2) + 2\alpha(ab^2)\alpha(ab^6) + \alpha(ab^6)^2 = (\alpha(ab^2) + \alpha(ab^6))^2 \equiv 0 \mod 4. \]
So from \eqref{eq:-y1z1}
\[\frac{-y_1z_1}{16} \equiv -(\alpha(ab^3)^2 + 2\alpha(ab^3)\alpha(ab^7) + \alpha(ab^7)^2) = -(\alpha(ab^3) + \alpha(ab^7))^2 \mod 4 \]
which can only be congruent to $0$ or $-1$ modulo $4$, contradicting $x_1^2 = -y_1z_1$. Hence
\begin{equation}\label{eq:congb3}
\alpha(b^3) \equiv \alpha(b^7) \mod 2.
\end{equation}
We now consider the real parts of $x_2^2$ and $-y_2z_2$. The real part of $x_2^2$ is divisible by $16$. 
So by \eqref{eq:-y2z2} and \eqref{eq:congb3}
\[0 \equiv Re(-y_2z_2) \equiv -8(-\alpha(ab^3)^2 + \alpha(ab^7)^2) \mod 16\]
which implies
\begin{equation}\label{eq:congab3}
\alpha(ab^3) \equiv \alpha(ab^7) \mod 2.
\end{equation}
Together with \eqref{eq:gSignga^2}, \eqref{eq:gCoefModb4} and \eqref{eq:1eq0} the congruences \eqref{eq:congab^2}, \eqref{eq:congb3} and \eqref{eq:congab3} can be compactly written as 
\begin{align}\label{eq:congbplusgb4}
\alpha(g) + \alpha(gb^4) \equiv 0 \mod 2 \ \ \forall g \in G.
\end{align}
These are all the equations and congruences we need.

Let now $e\in \PCI(\Q G) $. If $e$ corresponds to a component which is not a matrix component, then $ne = 0$ which is clearly an element in $\mathbb{Z}G$. To analyze the other elements of $\PCI(\Q G)$ we will deploy \Cref{lem:CoefficientOfProjectionByCharacter}. Let first $e \in \PCI(\Q G)$ be the element corresponding to the $M_2(\mathbb{Q})$-representation and let $\chi$ be its character. Then from the representation given above we get
\[ \chi(g) = \left\{ \begin{array}{lll} 2, & g \in \langle b^2 \rangle, \\  -2, & g \in a^2\langle b^2 \rangle,\\ 0, & \text{else}. \end{array}\right.\]
So by \Cref{lem:CoefficientOfProjectionByCharacter} we can compute the coefficient of $ne$ at a generic element $g \in G$ in the following way, where we use first the values of $\chi$, then \eqref{eq:gSignga^2} and then \eqref{eq:gCoefModb4}:
\begin{align*}
\frac{k}{|G|}\sum_{h \in G}&\alpha(gh^{-1})\chi(h^{-1}) \\
 &= \frac{2}{32}\left(2(\alpha(g) + \alpha(gb^2) + \alpha(gb^4) + \alpha(gb^6) - \alpha(ga^2) - \alpha(ga^2b^2) -\alpha(ga^2b^4) - \alpha(ga^2b^6))\right)\\
&= \frac{1}{8}\left(2(\alpha(g) + \alpha(gb^2) + \alpha(gb^4) + \alpha(gb^6))\right) = \frac{1}{4}\left(2(\alpha(g)+\alpha(gb^4))\right) = \frac{1}{2}\left(\alpha(g) + \alpha(gb^4)\right).
\end{align*}
By \eqref{eq:congbplusgb4} all these numbers are integers and hence $ne \in \mathbb{Z}G$.

Finally let $e \in \PCI(\Q G)$ be the element corresponding to the component  $M_2(\Q (i))$ and $\chi$ its character. The argument will be similar to the previous case. Note that we consider $\chi$ as a character of a $\mathbb{Q}$-representation, so that each of the entries in the representation given in \eqref{eq:G223Reps} corresponds to a $2\times 2$-matrix and $i$ corresponds to a matrix with trace $0$, e.g. to its rational canonical form $\begin{pmatrix} 0 & -1 \\ 1 & 0 \end{pmatrix}$. Hence
\[ \chi(g) = \left\{ \begin{array}{lll} 4, & g \in \langle a^2b^4 \rangle, \\  -4, & g \in \{a^2, b^4 \}, \\ 0, & \text{else}. \end{array}\right.\]
We again use \Cref{lem:CoefficientOfProjectionByCharacter} to compute the coefficient of $ne$ at $g$, where we use first the values of $\chi$ and then \eqref{eq:gSignga^2}:
\begin{align*}
\frac{k}{|G|}\sum_{h \in G}&\alpha(gh^{-1})\chi(h^{-1}) = \frac{2}{32}\left(4(\alpha(g) + \alpha(ga^2b^4) - \alpha(ga^2) - \alpha(gb^4))\right) \\
&= \frac{1}{4}\left(2(\alpha(g) + \alpha(gb^4))\right) = \frac{1}{2}\left(2\alpha(g)+\alpha(gb^4)\right).
\end{align*}
Hence again by \eqref{eq:congbplusgb4} all the coefficients of $ne$ are integers.
Overall we conclude that $G$ has the (ND) property.
\end{proof}

Our next goal is to show that $G(2,2,3)$ is in fact the only nilpotent counterexample to \Cref{conj:EricWei}. In the nilpotent case, by \Cref{prop:NilpotentSNNotSSN} we need to understand nilpotent groups with SSN and the groups $G(p,m,n)$ from \eqref{definition BJ1 class} in particular. 
The proof for this class of groups will proceed through several lemmas which separate the cases which remain open. All of them will be handled by a similar construction which will be made concrete in all the cases. It is inspired by an argument in \cite{Liu}.

\begin{lemma}\label{lem:ThersyConstruction}
Let $p$ be a prime, $r,s \in \mathbb{Z}G$, $y \in \ZZ(\mathbb{Z}G)$ and $e \in \mathbb{Q}G$ a central idempotent such that the following hold:
\begin{itemize}
\item[(i)] $r^2 = s^2 = rs = sr = 0$,
\item[(ii)] $er = r$, $es = 0$,
\item[(iii)] $y(r+s)/p \in \mathbb{Z}G$,
\item[(iv)] $yr/p \notin \mathbb{Z}G$.
\end{itemize}
Then $G$ does not have ND.
\end{lemma} 
\begin{proof}
By (i) we have $rs = sr$, so that by (i) $(y(r+s)/p)^2 = y^2(r+s)^2/p^2 = 0$. So by (iii) $y(r+s)/p$ is a nilpotent element in $\mathbb{Z}G$ and moreover $ey(r+s)/p = yr/p$ by (ii). So by (iv) $(y(r+s)/p$ is non-zero and $G$ does not have ND.
\end{proof}

The property of having one matrix component was systematically studied for groups with SSN by Jespers and Sun. We record the result relevant for this section which motivates the following lemmas.

\begin{lemma}\label{lem:GpmnOneMatrixComponent} \cite[Lemmas 4.2, 4.3]{JespersSun}
For $G = G(p,m,n)$ the algebra $\mathbb{Q}G$ has one matrix component if and only if $n = 1$ or $p=m=n=2$. 
\end{lemma}

The proofs of the next four lemmas all employ \Cref{lem:ThersyConstruction}. The first one will be especially detailed to facilitate the understanding of the arguments later also.

\begin{lemma}\label{lem:GpmnEvenCase}
Let $G = G(2,m,n)$ with $n \geq 2$ and $(m,n) \notin \{(2,2), (2,3) \}$. Then $G$ does not have ND.
\end{lemma}

\begin{proof}
In this situation it was already shown by Liu that $G$ does not have ND when either $m = 2$ and $n \geq 4$ or $m = 3$ \cite[Lemma 2.8]{Liu}. While Liu's statement of the lemma is different, the proof shows exactly that these groups do not have (ND). As our proof is uniform for all cases, this will include a repetition of Liu's result. The goal is to use \Cref{lem:ThersyConstruction} and the items (i)-(iv) refer to this lemma. We set $p=2$.

Let 
\begin{align*}
r &= a(a^{p^{m-1}}+b)(1-a^{p^{m-1}})(1+b^p)\widetilde{b^{p^2}}, \\
s &= a(a^{p^{m-2}}+b)(1-a^{p^{m-1}})(1-b^p)\widetilde{b^{p^2}}, \\
y &= 1 + a^{p^{m-2}}.
\end{align*}

Note that
\[(1+b^p)(1-b^p)\widetilde{b^{p^2}} = (1-b^{p^2})\widetilde{b^{p^2}} = 0, \]
so that $rs = sr = 0$ follows using that $b^p$ is central in $G$. Furthermore,
\begin{align*}
a&(a^{p^{m-1}} + b)a(a^{p^{m-1}}+b) = a^p(a^{p^{m-1}} + ba^{p^{m-1}})(a^{p^{m-1}}+b) \\
 &= a^p(1 + ba^{p^{m-1}} + b + b^pa^{p^{m-1}}) = a^p((1-b^p) + (1+a^{p^{m-1}})(b+b^p)).
\end{align*}
As 
\begin{equation}\label{eq:p2ZeroProds}
(1-b^p)(1+b^p)\widetilde{b^{p^2}} = 0 = (1+a^{p^{m-1}})(1-a^{p^{m-1}}),
\end{equation}
this implies $r^2 = 0$. Moreover,
\begin{align*}
a&(a^{p^{m-2}}+b)a(a^{p^{m-2}}+b) = a^p (a^{p^{m-2}} + ba^{p^{m-1}})(a^{p^{m-2}} + b) \\
 &= a^p (a^{p^{m-1}} + ba^{p^{m-2}} + ba^{p^{m-1} + p^{m-2}} + b^pa^{p^{m-1}}) \\
  &= a^p(a^{p^{m-1}}(1+b^p) + ba^{p^{m-2}}(1 + a^{p^{m-1}}))
\end{align*}
which by \eqref{eq:p2ZeroProds} also implies $s^2 = 0$. So (i) follows.

To construct the idempotent for (ii), note again that $b^p$ is central in $G$ so that if $f \in \PCI(\mathbb{Q}G)$, then $b^pf = \zeta I$ for a certain root of unity $\zeta$, where $I$ denotes the identity matrix of some size. The same argument applies to $a^{p^{m-1}}$. Now let $e \in \PCI(\mathbb{Q}G)$ which is the sum of all primitive central idempotents $f$ such that $fa^{p^{m-1}}$ has order $p$ and $fb^p$ is the identity. Note that as $G' = \langle a^{p^{m-1}} \rangle$, this implies that $fb$ is not central. So for each such $f$ we have $rf \neq 0$. It is clear that $se = 0$ as $(1-b^p)e=0$. Furthermore, if $f'$ is any primitive central idempotent such that $f'b^p$ is not the identity, then $f'(1+b^p) \widetilde{b^{p^2}} = 0$. Similarly, if $f'a^{p^{m-1}}$ does not have order $p$, it must be the identity, so $f'(1-a^{p^{m-1}}) = 0$. We conclude $r(1-e) = 0$ and so $re = r$ and (ii) holds.

Next,
\begin{align*}
y(r+s) &= ya(1-a^{p^{m-1}})\widetilde{b^{p^2}}((a^{p^{m-1}} + b)(1+b^p) + (a^{p^{m-2}} + b)(1-b^p)) \\
&= ya(1-a^{p^{m-1}})\widetilde{b^{p^2}}(a^{p^{m-1}} + b^pa^{p^{m-1}} + a^{p^{m-2}} - b^pa^{p^{m-2}} + 2b) \\
&= ya(1-a^{p^{m-1}})\widetilde{b^{p^2}}((a^{p^{m-2}}(1 + a^{p^{m-2}}) + b^pa^{p^{m-2}}(-1 + a^{p^{m-2}}) + 2b).
\end{align*}
Using that $1 + a^{p^{m-2}} \equiv -1 + a^{p^{m-2}} \equiv 1 - a^{p^{m-2}} \bmod 2$ this means
\begin{align*}
y(r+s) &\equiv a(1+a^{p^{m-2}})(1-a^{p^{m-2}})(1-a^{p^{m-1}})\widetilde{b^{p^2}}(a^{p^{m-2}} + b^pa^{p^{m-2}}) \\
&= a(1 - a^{p^{m-1}})(1-a^{p^{m-1}})\widetilde{b^{p^2}}(a^{p^{m-2}} + b^pa^{p^{m-2}}) \\
&\equiv a(1 + a^{p^{m-1}})(1-a^{p^{m-1}})\widetilde{b^{p^2}}(a^{p^{m-2}} + b^pa^{p^{m-2}}) = 0 \mod 2.
\end{align*}
Hence $y(r+s)/p \in \mathbb{Z}G$ and (iii) holds.

Finally, it is easy to see that the coefficient of $ab$ in the element 
\[yr = (1+a^{p^{m-2}})a(a^{p^{m-1}} + b)(1-a^{p^{m-1}})(1+b^p)\widetilde{b^{p^2}}\]
is $1$, so that $yr/p \notin \mathbb{Z}G$. This proves (iv) and the fact that $G$ does not have (ND) hence follows.
\end{proof}

For the case of odd primes we will use the following technical lemma.

\begin{lemma}\label{lem:GpmnTechnicalLemmaOddp}
Let $p$ be an odd prime, $k$ a positive integer and 
\[X =  \langle x,z \ | \ x^{p^2} = 1, z^{p^k} = 1, [x,z] = 1 \rangle \cong C_{p^2} \times C_{p^k}. \]
 Then there exist $\alpha, \beta \in \mathbb{Z}X$ such that in $\mathbb{Z}X$ the following congruences hold:
\[(1-x^p)(1-x^pz) + (1-x^{-p})(1-x^{-p}z) \equiv (1-x)^{p+1}\alpha  \mod p\]
and 
\[x(1-x^p)(1-x^pz) + x^{-1}(1-x^{-p})(1-x^{-p}z) \equiv (1-x)^{p+1}\beta  \mod p.\]
\end{lemma}

\begin{proof}
To simplify notation we write for a positive integer $\ell$:
\[\gamma(x,\ell) = 1 + x + x^{2} + ... + x^{\ell}.\]
We will several times use the equation 
\begin{equation}\label{eq:gamma}
(1-x^{\ell}) = (1-x)\gamma(x, \ell-1).
\end{equation}
We will also use that 
\begin{equation}\label{eq:powercongruence}
(1-x^p) \equiv (1-x)^p \mod p.
\end{equation}
We proceed to show the first congruence using first $1-x^{-p} = -x^{-p}(1-x^p)$ and later \eqref{eq:gamma} and finally \eqref{eq:powercongruence}. We also use $x^{-p} = x^{p^2-p}$ and $x^{-3p} = x^{p^2-3p}$.
\begin{align*}
(&1-x^p)(1-x^p z) + (1-x^{-p})(1-x^{-p} z) \\
=& (1-x^p)(1-x^p z) - x^{-p}(1-x^p)(1-x^{-p} z) \\
=& (1-x^p)(1-x^p z - x^{-p}(1-x^{-p} z)) \\
=& (1-x^p)((1-x^{-p}) - x^p z (1-x^{-3p})) \\
=& (1-x^p)((1-x)\gamma(x, p^2-p-1) - x^p z(1-x)\gamma(x, p^2-3p-1)) \\
=& (1-x^p)(1-x) (\gamma(x, p^2-p-1) - x^p z\gamma(x, p^2-3p-1)) \\
\equiv & (1-x)^{p+1}(\gamma(x, p^2-p-1) - x^p z \gamma(x, p^2-3p-1)) \mod p.
\end{align*}
Note that when $p=3$, then in the fourth line $1-x^{-3p} = 0$, so that $p^2-3p-1$ does not appear later in this case. This shows the first congruence.

Next, we show the second congruence, also using $1-x^{-p} = -x^{-p}(1-x^p)$, \eqref{eq:gamma} and \eqref{eq:powercongruence}  and also that $x^{-3p-2} = x^{p^2-3p-2}$ for $p \neq 3$ and $x^{-3p-2} = x^{p^2-2}$ for $p=3$:
\begin{align*}
& x(1-x^p)(1-x^p z) + x^{-1}(1-x^{-p})(1-x^{-p} z) \\
=& x(1-x^p)(1-x^p z) - x^{-1}x^{-p}(1-x^p)(1-x^{-p} z) \\
=& (1-x^p)(x(1-x^p z) - x^{-p-1} (1-x^{-p} z)) \\
=& (1-x^p)(x - x^{-p-1} - x^{p + 1}z + x^{-2p-1} z) \\
=& (1-x^p)(x(1-x^{-p-2}) - x^{p + 1} z(1-x^{-3p-2})) \\
=& (1-x^p)(1-x)(x\gamma(x, p^2-p-3) - x^{p + 1} z\gamma(x, p^2-3p-3)) \\
\equiv& (1-x)^{p+1}(x\gamma(x, p^2-p-3) - x^{p + 1}\gamma(x, p^2-3p-3)) \mod p,
\end{align*}
where in the last two lines in case $p=3$ the expression $p^2-3p-3$ has to be replaced by $p^2-3$. This shows the second congruence.
\end{proof}

The next three lemmas will now take care of the remaining cases for the groups $G(p,m,n)$.

\begin{lemma}\label{lem:GpmnOddCase1}
Let $p$ be odd, $m \geq 3$ and $n \geq 2$. Then $G = G(p,m,n)$ does not have ND. 
\end{lemma}
\begin{proof}
Set
\begin{align*}
r &= (b - a^{p^{m-2}})a(1-a^{p^{m-1}}) \widetilde{b^{p^2}} \prod_{i=0}^{p-2}(1-b^pa^{ip^{m-1}}), \\
s &= (b- a^{-p^{m-2}})a(1-a^{-p^{m-1}}) \widetilde{b^{p^2}} \prod_{i=2}^p (1- b^pa^{ip^{m-1}}),\\
y &= (1-a^{p^{m-2}})^{p(p-1)-1}.
\end{align*}
We will again show (i)-(iv) from \Cref{lem:ThersyConstruction} which will imply that $G$ does not have (ND). We analyze the irreducible representations of $G$ in which $r$ and $s$ are not mapped to $0$. Note that as $a^p$ and $b^p$ are central, they are mapped to a central matrix under every irreducible representation. Let $R$ be an irreducible $\mathbb{Q}$-representation and $e$ the corresponding primitive central idempotent of $\mathbb{Q}G$. Denote by $I$ the identity matrix. If $R(b^{p^2}) \neq I$, then $e\widetilde{b^{p^2}} = 0$, as the sum over all the powers of a primitive $p^\ell$-th root of unity equals $0$ when $\ell \geq 1$. Moreover $R(b^p) = I$ implies $e(1-b^p) = 0$. Hence, as both $r$ and $s$ contain the factor $(1-b^p)\widetilde{b^{p^2}}$ the the inequality $er \neq 0$ implies that $R(b^p) = \zeta I$ for $\zeta$ a primitive $p$-th root of unity while $es \neq 0$ implies $R(b^p) = \zeta'I$ for $\zeta'$ a primitive $p$-th root of unity. Moreover, if $R(a^{p^{m-1}}) = I$, then $e(1-a^{p^{m-1}}) = 0$. As $a^{p^{m-1}}$ has order $p$, we conclude that $er \neq 0$ implies $R(a^{p^{m-1}}) = \xi I$ while $es \neq 0$ implies $R(a^{p^{m-1}}) = \xi'I$ for certain primitive $p$-th roots of unity $\xi$ and $\xi'$. The factor $\prod_{i=1}^{p-2} (1-b^p a^{ip^{m-1}})$ in $r$ means that $er \neq 0$ implies $R(b^p) \neq R(a^{ip^{m-1}})^{-1}$ for every $1 \leq i \leq p-2$. From the fact that both $b^p$ and $a^{p^{m-1}}$ are mapped to elements of order $p$, we conclude that $er \neq 0$ means $R(b^p) = R(a^{p^{m-1}})$, i.e. $eb^p = ea^{p^{m-1}}$. Similarly $es \ne 0$ implies $eb^p = ea^{-p^{m-1}}$. It follows that $r$ and $s$ live in different components of $\mathbb{Q}G$, so that $rs = sr = 0$ and also (ii) holds.

We next show that $r^p = 0$. The non-central factors of $r$ give
\[((b-a^{p^{m-2}})a)^p = a^p(b - a^{p^{m-2}})(ba^{(p-1)p^{m-1}} - a^{p^{m-2}})(ba^{(p-2)p^{m-1}} - a^{p^{m-2}})...(ba^{p^{m-1}} - a^{p^{m-2}}). \]
From the paragraph before we know that when $er \neq 0$, then $R(a^{p^{m-1}}) = \zeta I$ for some primitive $p$-th root of unity $\zeta$. So then
\[e((b-a^{p^{m-2}})a)^p = ea^p(b - a^{p^{m-2}})(b\zeta^{-1} - a^{p^{m-2}})(b\zeta^{-2} - a^{p^{m-2}})...(b\zeta - a^{p^{m-2}}). \]
We claim that the coefficient of $a^{ip^{m-2}}$ in the expression 
\[(b - a^{p^{m-2}})(b\zeta^{-1} - a^{p^{m-2}})(b\zeta^{-2} - a^{p^{m-2}})...(b\zeta - a^{p^{m-2}})\]
is $0$ for every $1 \leq i \leq p-1$. Indeed, using $\prod_{j=0}^{p-1}(X-\zeta^j) = X^p - 1$, as an equation in the polynomial ring $\mathbb{Z}[X]$, up to the factor $b^i$ and possibly a sign this coefficient is the same as in $\prod_{j=0}^{p-1} (a^{p^{m-2}} - \zeta^j) = (a^{p^{m-2}})^p - 1$. So,
\[e((b-a^{p^{m-2}})a)^p = ea^p(b^p - a^{p^{m-1}}) = 0, \]
where the last equality follows from the previous paragraph. Hence, $r^p = 0$. A similar calculation shows also $s^p$, so that (i) follows.

We proceed to show (iii). 
We first calculate
\begin{align*}
y&(r+s) = y\widetilde{b^{p^2}}(1-b^p)\prod_{i=2}^{p-2}(1-a^{ip^{m-1}}b^p) \\
 & \cdot ((b-a^{p^{m-2}})a(1-a^{p^{m-1}})(1-a^{p^{m-1}}b^p) + (b-a^{-p^{m-2}})a(1-a^{-p^{m-1}})(1-a^{-p^{m-1}}b^p)) \\
 &= y\widetilde{b^{p^2}}(1-b^p)\prod_{i=2}^{p-2}(1-a^{ip^{m-1}}b^p) \\
 &\cdot (ba((1-a^{p^{m-1}})(1-a^{p^{m-1}}b^p) + (1-a^{-p^{m-1}})(1-a^{-p^{m-1}}b^p)) \\
 &  -a(a^{p^{m-2}}(1-a^{p^{m-1}})(1-a^{p^{m-1}}b^p)+a^{-p^{m-2}}(1-a^{-p^{m-1}})(1-a^{-p^{m-1}}b^p))).
\end{align*}
So the last two lines mean that, taking into account the factor $y$, it is enough to show that
\begin{equation}\label{eq:cong1pOdd}
y ((1-a^{p^{m-1}})(1-a^{p^{m-1}}b^p) + (1-a^{-p^{m-1}})(1-a^{-p^{m-1}}b^p)) \equiv 0 \mod p
\end{equation}
and 
\begin{equation}\label{eq:cong2pOdd}
y(a^{p^{m-2}}(1-a^{p^{m-1}})(1-a^{p^{m-1}}b^p) + a^{-p^{m-2}}(1-a^{-p^{m-1}})(1-a^{-p^{m-1}}b^p)) \equiv 0 \mod p.
\end{equation}
Note that 
\[y(1-a^{p^{m-2}})^{p+1} = (1-a^{p^{m-2}})^{p(p-1)-1+p+1} = (1-a^{p^{m-2}})^{p^2} \equiv (1 - a^{p^m}) = 0 \mod p\]
 by \eqref{eq:powercongruence}. So to prove \eqref{eq:cong1pOdd} and \eqref{eq:cong2pOdd} it is enough to show that the factors to the right of $y$ contain a factor $(1-a^{p^{m-2}})^{p+1}$ modulo $p$. This follows by applying \Cref{lem:GpmnTechnicalLemmaOddp} with $x = a^{p^{m-2}}$ and $z = b^p$. This shows (iii).

Finally, to get (iv) note that the coefficient of $ba$ in $yr$ is $1$. This follows as none of the products one can get by factoring out the element $yr$ gives $ba$ except the trivial one which in turns follows as the powers appearing for $a$ and $b$ are otherwise not big enough to sum up to $p^m$ or $p^n$ respectively when $a$ and $b$ are both taken in as a factor.
\end{proof}

\begin{lemma}\label{lem:GpmnOddCase2}
Let $p$ be odd, $m =2$ and $n > m$. Then $G = G(p,m,n)$ does not have ND. 
\end{lemma}
\begin{proof}
Again we will show (i)-(iv) from \Cref{lem:GpmnTechnicalLemmaOddp}, this time using the elements:
\begin{align*}
r &= (a-b^{p^{n-2}})b(1-b^{p^{n-1}})\prod_{i=0}^{p-2}(1-a^pb^{ip^{n-1}}), \\
s &= (a-b^{-p^{n-2}})b(1-b^{-p^{n-1}})\prod_{i=2}^p(1-a^pb^{ip^{n-1}}), \\
y & = (1-b^{p^{n-2}})^{p(p-1)-1}.
\end{align*}
We again first analyze the properties of primitive central idempotents which do not map $r$ or $s$ to $0$. Note that $a^p$ and $b^{p^{n-1}}$ both have order $p$. Let $e \in \PCI(\mathbb{Q}G)$. The factor $(1-a^p)$ in both $r$ and $s$ means that $er \neq 0$ implies that $ea^p$ has order $p$ and also $es \neq 0$ means that $ea^p$ has order $p$. From the factor $(1-b^{p^{n-1}})$ in $r$ and the factor $(1-b^{-p^{n-1}})$ in $s$ we also get that $er \neq 0$ implies that $eb^{p^{n-1}}$ has order $p$ and $es \neq 0$ implies the same. Finally, the rest of the factors appearing on the right from $b$ then mean that $er \neq 0$ implies $ea^p = eb^{p^{n-1}}$ and $es \neq 0$ implies $ea^p = eb^{-p^{n-1}}$. This in particular gives that $rs = sr = 0$ and (ii).

Computing $r^p$ and $s^p$ is also very similar to the previous case. Namely the non-central part of $r$ gives
\[((a-b^{p^{n-2}})b)^p  = b^p(a-b^{p^{n-2}})(aa^p-b^{p^{n-2}})(aa^{2p}-b^{p^{n-2}})...(aa^{(p-1)p} - b^{p^{n-2}}).\]
As $ea^p$ has order $p$ when $er \neq 0$, the coefficient of $(b^{p^{n-2}})^i$ in the expression 
\[e(a-b^{p^{n-2}})(aa^p-b^{p^{n-2}})(aa^{2p}-b^{p^{n-2}})...(aa^{(p-1)p} - b^{p^{n-2}})\]
is $0$ for all $1 \leq i \leq p-1$, so that $e((a-b^{p^{n-2}})b)^p = eb^p(a^p - b^{p^{n-1}}) = 0$, implying $r^p = 0$. Similarly $s^p = 0$. Overall, we obtain (i).

Now,
\begin{align*}
y&(r+s) = y(1-a^p)\prod_{i=2}^{p-2}(1-a^pb^{ip^{n-1}}) \\
\cdot& ((a-b^{p^{n-2}})b(1-b^{p^{n-1}})(1-a^pb^{p^{n-1}}) + (a-b^{-p^{n-2}})b(1-b^{-p^{n-1}})(1-a^pb^{-p^{n-1}})) \\
=& \ y(1-a^p)\prod_{i=2}^{p-2}(1-a^pb^{ip^{n-1}}) \\
\cdot &(ab((1-b^{p^{n-1}})(1-a^pb^{p^{n-1}}) + (1-b^{-p^{n-1}})(1-a^pb^{-p^{n-1}})) \\
-& b(b^{p^{n-2}}(1-b^{p^{n-1}})(1-a^pb^{p^{n-1}}) + b^{-p^{n-2}}(1-b^{-p^{n-1}})(1-a^pb^{-p^{n-1}}))).
\end{align*}
As $y(1-b^{p^{n-2}})^{p+1} = (1-b^{p^{n-2}})^{p(p-1)-1+p+1} = (1-b^{p^{n-2}})^{p^2} \equiv 0 \mod p$, it is hence enough to show that the expressions
\begin{equation}\label{eq:con1OddpCase2}
(1-b^{p^{n-1}})(1-a^pb^{p^{n-1}}) + (1-b^{-p^{n-1}})(1-a^pb^{-p^{n-1}})
\end{equation}
and 
\begin{equation}\label{eq:con2OddpCase2}
b^{p^{n-2}}(1-b^{p^{n-1}})(1-a^pb^{p^{n-1}}) + b^{-p^{n-2}}(1-b^{-p^{n-1}})(1-a^pb^{-p^{n-1}})
\end{equation}
when considered modulo $p$ both contain a factor $(1-b^{p^{n-2}})^{p+1}$. This follows by applying \Cref{lem:GpmnTechnicalLemmaOddp} for $x = b^{p^{n-2}}$ and $z = a^p$. So we have (iii). Moreover analyzing $yr$ we see that the coefficient of $ab$ equals $1$, so also (iv) follows and $G$ does not have (ND) in this case.
\end{proof}

\begin{lemma}\label{lem:GpmnOddCase3}
Let $p$ be odd. Then $G = G(p,2,2)$ does not have ND. 
\end{lemma}
\begin{proof}
We note that the case $p = 3$ was considered in \cite[Lemma 2.2]{LiuPassman10}, but we will use different arguments. We will again show (i)-(iv) from \Cref{lem:GpmnTechnicalLemmaOddp} using
\begin{align*}
r =& (b-a)(1-a^p)\prod_{i=0}^{p-2}(1-a^{ip}b^p), \\
s =& (b-a^{-1})(1-a^{-p})\prod_{i=2}^p(1-a^{ip}b^p), \\
y =& (1-a)^{p(p-1)-1}.
\end{align*}
The proof of the fact that $r^p = s^p = 0$ is different from the cases before, so that we postpone this to the end. We start again by analyzing the properties of primitive central idempotents not annihilating $r$ and $s$. Similarly as in the other cases we get that $er \neq 0$ implies that $ea^p$ and $eb^p$ have order $p$ and $ea^p = eb^p$ holds. Also, $es\neq 0$ implies that $ea^p$ and $eb^p$ have order $p$ and $ea^p = eb^{-p}$. So $rs = sr = 0$ and (ii) follow.

Similarly as we had to show \eqref{eq:cong1pOdd} and \eqref{eq:cong2pOdd} before we now need to obtain that
\[(1-a^p)(1-a^pb^p) + (1-a^{-p})(1-a^{-p}b^p) \]
and
\[a(1-a^p)(1-a^pb^p) + a^{-1}(1-a^{-p})(1-a^{-p}b^p) \]
both contain a factor $(1-a)^{p+1}$ when considered modulo $p$. This follows by applying \Cref{lem:GpmnTechnicalLemmaOddp} for $x = a$ and $z = b^p$ and so we obtain (iii). It is also easy to see that the coefficient of $b$ in $yr$ is $1$, so that (iv) holds.

It remains to prove $r^p = s^p = 1$. We first claim that there is exactly one $e \in \PCI(\mathbb{Q}G)$ such that $er \neq 0$. To see this, note that, since the center of $G$ is $\langle a^p \rangle \times \langle b^p \rangle$ and isomorphic to an elementary abelian group of rank $2$, for each possible kernel different from the derived subgroup $\langle a^p \rangle$ there can be only one component with center $\mathbb{Q}(\zeta)$ for dimension reasons. Here $\zeta$ denotes a primitive $p$-th root of unity. Also, there is exactly one $ \in \PCI(\Q G)e$ such that $es \neq 0$. We will work in these unique components using explicit representations to see that $r^p = s^p = 0$. Set 
\[A = \begin{pmatrix} 0 & 0 & \cdots & 0 & 1 \\ \zeta & 0 &  \cdots & 0 & 0 \\ 0 & 1 &  \cdots & 0 & 0 \\ \vdots & & \ddots &  & \vdots \\ 0 & 0 & \cdots & 1 & 0 \end{pmatrix}\]
and
\[C = \begin{pmatrix} 1 & 0 & 0 & \cdots & 0 \\ 0 & \zeta^{-1} & 0 & \cdots & 0 \\ 0 & 0 & \zeta^{-2} & \cdots & 0 \\ \vdots & & & \ddots & \vdots \\ 0 & 0 & 0 &  \cdots & \zeta \end{pmatrix}.\]

Then $A^p = \zeta I$, where $I$ denotes the identity matrix, and $C^{-1}AC = A^{p+1}$. Hence each map $R_i: G \rightarrow M_p(\mathbb{Q}(\zeta))$ sending $a$ to $A$ and $b$ to $A^iC$ for $0 \leq i \leq p-1$ is a representation of $G$, since also $(A^iC)^p = A^{ip}$ is an element of order $p$. In fact, all these representations are irreducible: the degree of the representations is the smallest non-trivial divisor of the order of $G$, so any non-trivial decomposition would involve only linear representations. But the linear representations contain the derived subgroup $\langle a^p \rangle$ in their kernel, hence the character values of $a^p$ under the sum of $p$ linear representations is $p$, while the character value of $a^p$ under every of the representations $R_i$ is $p\zeta$. We note also that when $D$ is a diagonal matrix with one of the diagonal entries being $0$, then $(A^iD)^p=0$. This follows, since the characteristic polynomial of $A^iD$ equals $-X^p + \zeta^id_1d_2...d_p$, where $d_1,...,d_p$ are the diagonal elements of $D$, so that the only eigenvalue of $A^iD$ is $0$ when one of the $d_j$'s equals $0$.

The $e \in \PCI(\Q G)$ which satisfies $er \neq 0$ corresponds to the representation $R_1$, as in general $R_i(a^{ip}) = R_i(b^p)$ and as we saw before $ea^p = eb^p$. Now $R_1(a-b) = AC-A = A(C-I)$ and $C-I$ is a diagonal matrix containing $0$ on the diagonal. So, $r^p = 0$ follows. Similarly $R_{p-1}$ is the representation corresponding to the primitive central idempotent not annihilating $s$ and $D_{p-1}(b-a^{-1}) = A^{p-1}C - A^{-1} = A^{-1}(A^pC-I)$ and as also $A^pC-I$ is a diagonal matrix containing $0$ on the diagonal, we get $s^p= 0$ by the paragraph before. This finishes the proof of (i) in this case and the theorem follows.
\end{proof}

We are finally ready to prove the main theorem of this section.

\begin{proof}[Proof of \Cref{th:NilpotentCase}]
By \Cref{Summary theorem for SN} it remains to consider nilpotent group which have SSN. As observed in \cite[Section 4]{LiuPassman16} these were in fact classified in \cite{BozikovJanko}. They fall into nine categories. For eight of these categories it is shown in \cite[Section 4]{JespersSun} that if $G$ lies in one of them, it has ND if and only if $\mathbb{Q}G$ has at most one matrix component. The last category which remains open in general are the groups $G(p,m,n)$.

We already know by \Cref{th:counterex} that $G(2,2,3)$ does have ND. So to exclude the cases of $G$ having one matrix component by \Cref{lem:GpmnOneMatrixComponent} we can assume that $n \geq 2$ and $(m,n) \notin \{(2,2),(2,3) \}$ if $p=2$. Then $G$ does not have ND for any of the remaining cases by Lemmas~\ref{lem:GpmnEvenCase}, \ref{lem:GpmnOddCase1}, \ref{lem:GpmnOddCase2}, \ref{lem:GpmnOddCase3}.
\end{proof}

\subsection{Non-nilpotent groups}

The goal of this section is to handle the non-nilpotent part of \Cref{Summary theorem ND}. Namely, we show:

\begin{theorem}\label{ND for non-nilp SSN grps}
Let $G$ be a finite group which is not nilpotent and not an SSN group of unfaithful type. Then $G$ has ND if and only if it has one matrix component.
\end{theorem}

In case that $G$ has more than one matrix component the proof of \Cref{ND for non-nilp SSN grps} will in fact construct an explicit nilpotent element $n\in \Z G$ and central idempotent $e$ such that $ne \notin \Z G$. In \Cref{subsection on SN versus DK} we will dig deeper into this and it will turn out that the existence of these elements is connected to the kernels of the irreducible $\Q$-representations of $G$.

\begin{proof}[Proof of \Cref{ND for non-nilp SSN grps} for $G$ solvable.]
For every finite group $G$ whenever $\Q G$ has at most one  matrix component, then $G$ has ND. Conversely, assume that $G$ has ND and hence property SN. As $G$ is assumed to not be an SSN group of unfaithful type,  \Cref{prop:SolvSNGroups} says that $G \cong P \rtimes H$ for an elementary abelian $p$-group $P$, where the action is faithful, irreducible and $[x,h] \neq 1$ for every non-trivial $x \in P$ and  $h \in H$. Moreover $H$ is Dedekind. In other words, by \Cref{th:DedekindGroups}, either $H$ is abelian or $H \cong Q_8 \times D$ with $D$ an abelian group of odd order. In the latter case we denote by $c \in Q_8$ the unique (central) element of order $2$.

 {\it Claim 1: } $P$ is the unique maximal abelian subgroup. Moreover, it contains no non-trivial subgroup which is normal in $G$. Also, a subgroup $N$ of $G$ is normal in $G$ iff  $P \subseteq N$. If $H$ is abelian, then $G$ is metabelian with $G' = P$. If $H$ is non-abelian, then $G' = \langle P, c \rangle$ and $G'' = P.$  \vspace{0,1cm}
 
\noindent First notice that since $P$ is elementary abelian and the action of $H$ on $P$ is irreducible it cannot contain a subgroup normal in $G$. Moreover, as $[x,h] \neq 1$ for all non-trivial $x \in P, h \in H$, $P$ is indeed the unique maximal abelian subgroup. Now, if $G/P \cong H$ is abelian, then $G' \subseteq P$ and thus by the first part $G' = P$. If $H \cong Q_8 \times D, $ we directly see that $G' \subseteq \langle P, c \rangle$. Using that $G' \cap P \subseteq P$ is normal in $G$ one has that $P \subset G'$ and hence $G' = \langle P, c \rangle$. Analogously we see that $G'' \subseteq P$ and in fact $G'' = P$ as $G''$ is normal. Finally, consider $N \unlhd G$. Then $N \cap P \subseteq P$ is normal in $G$, hence $N \subseteq P$ as $H$ contains no normal subgroups. Conversely, if $P \subseteq N$ then $N/P \leq G/P$. As mentioned above $G/P$ is Dedekind and thus $N/P$ is normal as claimed.

Next note that $G$ is strongly monomial, being abelian-by-supersolvable, and hence by \cite[Theorem 3.5.10.]{GRG1} all primitive central idempotents of $\Q G$ are of the form $e(G,N,K)$ for some strong Shoda pairs $(N,K)$. As recorded in \cite[Lemma 2.4.]{JespersSun}, $\Q[G]e(G,N,K)$ is commutative if and only if $G' \subseteq K$.\vspace{0.1cm}

{\it Claim 2:} The tuples in $\{ (P, K) \mid [P:K] = p \}$ are strong Shoda pair of $G$. Conversely, if $(N,K)$ is a strong Shoda pair with $N \unlhd G$, then $P \subseteq K$ or $N=P$. \vspace{0,1cm}

\noindent Let $K \unlhd N \unlhd G$. Then \cite[Corollary 3.5.11]{GRG1} tells that $(N,K)$ is a strong Shoda pair exactly when $N/K$ is cyclic and $N/K$ is a maximal abelian subgroup of $N_G(K)/K$. In particular, $N' \subseteq K$. 

\noindent To start notice that $\langle P, h \rangle /K$ is non-abelian for every $K \lneq P$ and non-trivial $h \in H$. This follows from \cite[(24.6), pg 112]{Aschbacher} asserting that  $P = [P, \langle h \rangle]$. Consequently, $P/K$ is maximal abelian in $N_G(K)/K$ and hence $(P,K)$ is a strong Shoda pair when $[P:K]=p$.  Next, by the first claim $P \leq N$ when $N \unlhd G$. Suppose $P \nsubseteq K$. If $N/P$ is abelian, then $N' \leq P \cap K \lneq P$. As $N' \unlhd G$, the first claim yields $N' = 1$ and so $N= P$. If $N/P$ is non-abelian, then $H$ is non-abelian and $c \in N$. So in that case $G' = \langle P, c \rangle  \leq N$, which entails that $G'' = P \leq N' \leq K$, a contradiction. This proves the second claim.

By \cite[Problem 3.4.4.]{GRG1}, the number of simple components $\Q Ge(G,P,K)$ with $[P:K]=p$, denoted $s$, is equal to the number of orbits of $H$ acting on $\mathcal{S} := \{ K \mid  [P:K] = p \}.$ To count the latter we decompose $\mathcal{S} = \bigcup_{d \mid |H|} \mathcal{S}_d$ with $\mathcal{S}_d := \{ K \in \mathcal{S} \mid |N_H(K)|=d \}.$ Note that the action of $H$ on $\mathcal{S}$ preserves each $\mathcal{S}_d$ and denote by $s_d$ the number of $H$-orbits thereon. Thus $s = \sum_{d \mid |H|} s_d$ and 
\begin{equation}\label{formula s_d}
s_d = \frac{1}{|H|}\sum_{K \in \mathcal{S}_d} |N_H(K)| = \frac{d\, .\, |\mathcal{S}_d|}{|H|}.
\end{equation}

Next let $\mathcal{T}_K$ be a left transversal for $N_H(K)$ in $H$. Then as every $K \in \mathcal{S}_d$ is a maximal subgroup, one has by \Cref{form of SSP idempotent result} that 
$$e_K := e(G,P,K) = \sum_{h \in \mathcal{T}_K} (\widehat{K}^h - \widehat{P}).$$
Now consider any non-trivial nilpotent element of the form 
\[x = (1-y) g \widetilde{H}\]
with $1 \neq y \in H$ and $g \in G \setminus N_G(H)$ (which exists as $H$ is non-normal). \vspace{0.1cm}

{\it Claim 3:} If $x e_K \in \Z G$ for all $K \in \mathcal{S}$, then $s = 1$. 

\noindent First write $y^g = t . v$ with $t \in P$ and $v \in H$. As $x$ is non-trivial, also $t \neq 1$ and $x = g (1 -y^{g}) \widetilde{H}= g (1- t) \widetilde{H}$. Therefore $ g^{-1}x e_{K} = (1-t) e_K \widetilde{H} \in \Z G.$ Because $K \subseteq P \unlhd G,$ one has that ${\supp}( (1-t)e_K )\subseteq P$ and so $(1-t)e_K \in \Z P$. Next note that $(1-t) \widehat{P} =0$ and $(1-t)\widehat{K}^h=0$ exactly when $t \in K^h$. Therefore $(1-t)e_K \in \Z P$ exactly means that
$$\frac{|\{ h\in \mathcal{T}_K \mid t \notin K^{h} \}|}{|K|} \in \Z.$$
Recall that $\operatorname{core}_G(K) = \ker \left( g \mapsto g e_K \right)$, by \Cref{form of SSP idempotent result}, which is trivial in this case. In particular $|\{ h\in \mathcal{T}_K \mid t \notin K^{h} \}| \neq 0.$
Therefore, if $K \in \mathcal{S}_d$ then $|K| \leq |\mathcal{T}_K| = |H|/d$.  Now \eqref{formula s_d} entails that $s_d |K| \leq |S_d|$ which sums up to $s |K| \leq |S| = \frac{|P|-1}{p-1}$, since $|S|$ is the number of maximal dimensional subspaces in the $\mathbb{F}_p$-vector space $P$.  As $[P:K]=p$ the latter inequality simplifies to $s (p-1) \leq p - \frac{1}{|K|} \leq p-1$, hence $s \leq 1.$ In fact $s=1$ by the second claim.

Next notice that by \cite[Lemma 3.3]{JespersSun}, $\Q H$ has no nonzero nilpotent elements. Therefore, when decomposing $\Q G$ as
$$\Q G \cong \Q G \widehat{P} \oplus \Q G (1-\widehat{P})$$
the piece $\Q G \widehat{P} \cong \Q G/P \cong \Q H$ has no matrix components. So, it remains to prove that $\Q G (1-\widehat{P})$ is simple. 
By \Cref{dim of simple of SSP}
$$\dim_{\Q }(\Q G e_K) = |H| (p-1) [G:N_G(K)].$$
Furthermore, $[G:N_G(K)] = [H: N_H(K)] = |\mathcal{T}_d|$ if $K \in \mathcal{S}_d$. By the third claim there is a unique $d \mid |H|$ such that $s = s_d$ and $s=1$. So  \eqref{formula s_d} translates to $|\mathcal{T}_d| = |S| = \frac{|P|-1}{p-1}.$ Altogether,
$$\dim_{\Q}(\Q G\widehat{P}) + \dim_{\Q }(\Q G e_K) = |H| +  |H| (|P| -1) = |H| . |P| = \dim_{\Q}(\Q G).$$
Thus $\Q G (1 - \widehat{P}) = \Q G e_K$ is indeed simple, finishing the proof.
\end{proof}

It now remains to consider finite non-solvable groups. In this case we prove that none of the groups as in \Cref{prop:NonSolvableSNGroups} have ND. This will be done by proving that for every such group there is always a bicyclic nilpotent element which does not have ND. 

\begin{proof}[Proof of \Cref{ND for non-nilp SSN grps} for $G$ non-solvable.]
If $G$ has SN, then it does not have ND by \cite[Proposition 3.4]{JespersSun}. So assume $G$ has SN. By \Cref{prop:NonSolvableSNGroups} we know $G$ has a unique minimal normal subgroup $S$ which is a direct product of isomorphic non-abelian simple groups and that $G/S$ is Dedekind.

Let $y \in G$ be an element of order $2$ and $x \in G$ such that $y^x \notin \langle y \rangle$. Such $x$ and $y$ exist, as $S$ has even order by the Feit-Thompson Theorem. Hence we can construct the non-trivial nilpotent element $n = (1-y)x(1+y)$. If we write $n$ in the shape as in \Cref{lem:CoefficientOfProjectionByCharacter}, then
\[\alpha(g) = \left\{ \begin{array}{lll} 1, & g = x \ \ \text{or} \ \ g = xy, \\  -1, & g = yx \ \ \text{or} \ \ g = yxy, \\ 0, & \text{else}. \end{array}\right. \]
So by \Cref{lem:CoefficientOfProjectionByCharacter} the coefficient of $ne$ at $x$ is

\begin{align}\label{eq:CoefInNonsolvableProduct} 
\frac{k}{|G|}\sum_{h \in G} \alpha(h) \chi(x^{-1}h) &= \frac{k}{|G|} \chi\left(1+y - x^{-1}yx(1+y)\right)  = \frac{k}{|G|} \chi\left(1-[x,y]\right),
\end{align}
where in the last step we used $y=y^{-1}$ as well as the fact that $y$ and $x^{-1}yx$ are conjugate and so have the same character value, which allows us to cancel them.
If $S$ is not contained in the kernel of $\chi$, then the value appearing in \eqref{eq:CoefInNonsolvableProduct} is not $0$. Moreover we have 
\[\left|\frac{k}{|G|} \chi\left(1-[x,y]\right)\right| \leq \frac{2\chi(1)k}{|G|} = \frac{2\text{dim}_\mathbb{Q}(e\mathbb{Q}G)}{|G|}.\]
So if the last is a rational number smaller than $1$ for $\chi$ corresponding to a faithful representation, the product $ne$ cannot lie in $\mathbb{Z}G$.

Now set $f = \widehat{S}$. Then $\mathbb{Q}G = f\mathbb{Q}G \oplus (1-f)\mathbb{Q}G$ and the direct summand $(1-f)\mathbb{Q}G$ corresponds to all the irreducible faithful representations of $G$. None of the indecomposable direct summands in $(1-f)\mathbb{Q}G$ is a division algebra, as $\operatorname{SL}(2,5)$ is the only non-solvable group which is a finite subgroup of a division algebra \cite[2.1.4]{ShirvaniWehrfritz}. If $(1-f)\mathbb{Q}G$ is decomposable, then one of its indecomposable direct summands must have dimension smaller than $\frac{|G|}{2}$, as $f\mathbb{Q}G$ has positive dimension and $|G| = \text{dim}(f\mathbb{Q}G) + \text{dim}((1-f)\mathbb{Q}G)$. Now the number of simple components of $\mathbb{Q}G$ equals the number of conjugacy classes of cyclic subgroups of $G$ \cite[Corollary 7.1.12]{GRG1}. The components in $f\mathbb{Q}G$ correspond to conjugacy classes in $G/S$, i.e. classes not lying in $S$ except for the class of the trivial element. But as $S$ certainly contains at least three conjugacy classes, we conclude that $(1-f)\mathbb{Q}G$ has at least two indecomposable summands.
\end{proof}

\section{On a measure for unipotents to have an integral decomposition}\label{sec:GeneralOrders}

In \cite[Section 6]{JespersSun} it was observed by Jespers-Sun that one can measure how far a given finite group $G$ is from not having ND via a certain group denoted $q(G)$, whose definition only depends on $\U (\Z G)$. In loc.cit. also two rather general problems about $q(G)$ were presented: to classify the groups $G$ for which $q(G)$ is finite and to establish a connection between the structure of $q(G)$ and the simple components of $\Q G$. We present answers to the two problems. Namely we will show that $q(G)$ is a finite group when no simple component of $\Q G$ is exceptional, and infinite when it has a simple component isomorphic to $M_2(\Q)$ and a further group-theoretical condition holds. We end by defining and pointing out that the obstruction might also be of interest for arithmetic subgroups of general semisimple algebraic groups.

\subsection{The measure and link to elementary subgroups}\label{measure subsection}

Consider the Wedderburn-Artin decomposition 
\begin{align}\label{eq:WedderburnInSection4}
\Q G \cong \bigoplus_{e \in \PCI (\Q G)} \Ma_{n_e}(D_e),
\end{align}
 where $\Q Ge \cong \Ma_{n_e}(D_e)$ with $D_e$ a finite-dimensional division algebra over $\Q$. Moreover let $\U(\Z G)_{un} := \{ \alpha \in \U (\Z G) \mid \alpha \text{ is unipotent } \}$ be the set of unipotent units in $\U (\Z G)$. For every $e \in \PCI(\Q G)$ consider the subset 
$$\mathcal{E}_G(e) := \{ \alpha \in \U(\Z G)_{un} \mid (\alpha -1) e = \alpha -1\}$$
of unipotent elements such that $\Q Ge$ is the only component to which the element projects non-trivially. 

Denote by $\SL_1(\Z G)$ the group of elements in $\U(\Z G)$ whose projections to every simple component of $\Q G$ all have reduced norm $1$ over the centre of that component (cf. \cite[p. 67]{GRG1} for the definition). Note that $\langle  \U(\Z G)_{un} \rangle $ and $\langle \mathcal{E}_G(e) \rangle$ are normal subgroups of $\U (\Z G)$ which are contained in $\SL_1(\Z G)$. The measure is the following quotient group:

\begin{equation}\label{def: q(G) group}
q(G) := \langle  \U(\Z G)_{un} \rangle / \langle \mathcal{E}_G(e) \ | \ e \in \PCI(\Q G) \rangle.
\end{equation}

As noticed in \cite[Section 6]{JespersSun}, $G$ has ND if and only if $q(G) = 1.$ As such it indeed measures how far $G$ is from having ND. Furthermore in loc.cit. the authors asked when this group is finite and how its structure is connected to the simple components of $\Q G$. To answer this we will investigate certain concrete subgroups of $\langle  \U(\Z G)_{un} \rangle$.

Let $\O$ be an order in a division algebra $D$ of finite dimension over $\Q$ and $J$ a non-zero ideal in $\O$. Then we set
$$E_n(J) := \langle e_{ij}(r) \mid 1 \leq i \neq j \leq n, \, r \in J \rangle, $$
where $e_{ij}(r)$ is the elementary matrix in $\GL_n(\O)$ which has $1$ on the diagonal and $r$ in the $(i,j)$-entry. Next, partition $\PCI(\Q G)$ into the two subsets 
$\PCI(\Q G)_{div} := \{ e \in \PCI(\Q G) \mid \Q Ge \text{ is a division algebra}\}$
and its complement $\PCI(\Q G)_{\geq 2}$ of primitive central idempotents yielding simple components of reduced degree at least $2$. In this definition we also view a field as a division algebra.

Classical results imply directly the following useful fact, where the index at the right hand side can be infinite.

\begin{proposition}\label{th: upper-bound size q(G)}
With notation as in \eqref{eq:WedderburnInSection4}, let $f = \sum_{e^{\prime} \in \PCI(\Q G)_{div}} e^{\prime}$ and for every $e \in \PCI (\Q G)_{\geq 2} $ fix a maximal order $\O_e$ in $D_e$. Then, there exists a subgroup $U_e$ of $\langle  \U(\Z G)_{un} \rangle$ which is of the form $1-e + E_{n_e}(J_e)$ for some non-zero ideal $J_e$ of $\O_e$. Consequently,
$$|q(G)| \leq [\SL_1(\Z G) (1-f) : \prod_{e \in \PCI (\Q G)_{\geq 2}} U_e] \leq \prod_{e \in \PCI (\Q G)_{\geq 2}} [\SL_{n_e}(\O_e) : E_{n_e}(J_e)].$$
\end{proposition}
\begin{proof}
For every $e \in \PCI (\Q G)_{\geq 2}$ one can choose an idempotent $f_e$ in $\Q G$ such that $ef_e$ is non-central in $\Q G e$. Consider the associated generalized bicyclic units $\text{GBic}^{\{f_e \}}(\Q G)$, see \cite[Section 11.2]{GRG1} for definition. Then following \cite[Theorem 6.3]{JJS} the group $\text{GBic}^{\{f_e \}}(\Q G)$ contains a subgroup $U_e$ of the form  $1-e + E_{n_e}(J_e)$ for some non-zero ideal $J_e$ of $\O_e$. As $\text{GBic}^{\{f_e \}}(\Q G)$ is a subgroup of $\langle  \U(\Z G)_{un} \rangle$ the previous implies the first part of the statement. 

Next note that $\SL_1(\Z G) (1-f)$ is the projection of $\SL_1(\Z G)$ onto the simple components of $\Q G$ of reduced degree at least $2$. Also notice that a unipotent unit $\alpha$ in $\Z G$ projects in every simple component to a unipotent element and in particular has reduced norm $1$ there. Thus $\langle  \U(\Z G)_{un} \rangle$ can be viewed as a subgroup of $\SL_1(\Z G) (1-f)$. Moreover, by definition $\SL_1(\Z G) (1-f)$ is a subgroup of $\prod_{e \in \PCI (\Q G)_{\geq 2}} \SL_{n_e}(\O_e).$ Furthermore, $U_e \leq \mathcal{E}_G(e)$ as elementary matrices are unipotent. Altogether, as we are both increasing the nominator and decreasing the denominator, this yields the desired inequalities.
\end{proof}

Whether the elementary subgroups $E_n(I)$, for $n \geq 2$, are of finite index in $SL_n(\O)$ is related to the celebrated answers on the Subgroup Congruence Problem. In particular it depends on the so called $S$-rank of $\SL_n(D)$ where $S$ is the set of Archimedian places of $\mathcal{Z}(D)$. More precisely, if this invariant is at least $2$, then $E_n(I)$ will be of finite index in $\SL_n(\O)$ \cite{BassMilnorSerre, SerreSL2, Vas73, Tits, BakReh}. These facts lead to call a finite dimensional simple algebra {\it exceptional} if it is of one of the following types:

\begin{enumerate}
\item[I:] a non-commutative division algebra which is not a totally definite quaternion algebra,

\item[II:] $\Ma_2(D)$ with $D$ either $\Q$, an imaginary quadratic extension of $\Q$ or a totally definite quaternion algebra with center $\Q$. 
\end{enumerate}

As recorded in \cite[Lemma 6.9]{BJJKT}, the exceptional simple algebras $\Ma_n(D)$ with $n \geq 2$ are exactly those for which the $S$-rank of $\SL_n(D)$ is $1$. If $n = 1$, there is no non-trivial unipotent element in $\SL_n(\O)$. Thus the terminology ``exceptional'' refers to the fact that subgroups generated by unipotent elements in $\SL_n(\O)$ are not sufficient to describe $\SL_1(\Z G)$ up to commensurability. 

By \cite{EKVG} $\Q G$ has an exceptional component of reduced degree $2$ if and only if $G$ maps onto a list of $56$ groups. The table in the appendix of \cite{BJJKT} demonstrates that $\Ma_2(\Q)$ is the most recurrent exceptional component of that type as for only $19$ of the $56$ groups no $\Ma_2(\Q)$ is implied. Now suppose that there exists a primitive central idempotent $e$ of $\Q G$ such that $\Q Ge \cong \Ma_2(\Q)$. If $|G|$ is not divisible by $3$, then \cite[Remark 6.17]{BJJKT2} tells that $Ge \cong D_8 = \langle a,b \ | \ a^4=b^2=1, a^b = a^{-1} \rangle$. In other words, in that case $G$ is an extension of the form 
$$1 \rightarrow Q \rightarrow G \rightarrow D_8 \rightarrow 1.$$
Thus there exist some $g,h\in G$ such that $ge = a$ and $he = b$. As shown below when $Q$ is big enough, under the following condition $G$ is very far from having ND:
$$
(\star) \, \, \, \frac{o(h)}{o(hQ)} \leq 2. 
$$ 

\begin{theorem}\label{when q(G) finite for grp rings}
Let $G$ be a finite group. Then the following hold:
\begin{enumerate}
\item[(i)]  If $\Q G$ has no exceptional components of type II, then $q(G)$ is finite.
\item[(ii)] If $G$ has order at most $16$, then $q(G)$ is finite.
\item[(iii)] If $G$ has order bigger than $16$, maps onto $D_8$ and this surjection satisfies $(\star)$, then $q(G)$ is an infinite non-torsion group.
\end{enumerate}
\end{theorem}

The above answers both questions of Jespers-Sun formulated in \cite[Section 6]{JespersSun}. More precisely, \Cref{th: upper-bound size q(G)} and the proof of \Cref{when q(G) finite for grp rings} will show that for nilpotent elements to have an integral decomposition is not truly connected to the simple components of $\Q G$. The relationship is rather a combination of the congruence level of $\Z G$ in the maximal order of $\Q G$ on the one hand and the rank of the simple matrix components of $\Q G$ on the other hand.

The algebra $\Q G$ has a component $\Ma_2(\Q)$ if and only if $G$ maps onto $D_8$ or $S_3$ \cite[Remark 6.17]{BJJKT2}. Thus the mapping onto $D_8$ in \Cref{when q(G) finite for grp rings} is implied, if $\Q G$ has a $\Ma_2(\Q)$ component and $3 \nmid |G|$. The latter restriction appears due to the use of results in \cite[Section 10]{JJS}, but we expect it is not needed. As explained in the examples below, some variant of the condition $(\star)$ is however certainly necessary.

\begin{example}
\begin{enumerate}
\item All the groups in the family
$$G(2,2,n) = \langle a,b \ | \ 1 = a^{4} = b^{2^n}, a^b = a^{-1} \rangle$$
have a matrix component $\Ma_2(\Q)$ since $G(2,2,n)/\langle b^2 \rangle \cong D_8$. The surprising $G(2,2,3)$ has order $32$ and satisfies $o(b) = 4 o(bQ)$. Furthermore, by \Cref{th:counterex} it has ND (i.e. $q(G)=1$). Thus $G(2,2,3)$ is minimal with respect to being in none of the cases described in \Cref{when q(G) finite for grp rings}. 
\item Examples of groups satisfying $(\star)$ are split extensions of $D_8$ or more generally split extension of $D_{2^n}$ with $n\geq 3$. These groups even satisfy $o(h)= o(hQ)$, giving a wide class of examples where $q(G)$ is an infinite (non-torsion) group. However, when $o(h) = 2o(hQ)$ there also  exist examples that are not split extension of $D_{2^n}$, such as $G(2,2,2) = \langle a ,b \mid a^4 = b^4 =1, a^b = a^{-1}\rangle$ or the families of groups with the property that all simple matrix components of $\Q G$ are of the form $\Ma_2(\Q)$. Such groups have been classified in \cite{JesRioCrelle} and in case of $2$-groups are given by seven possible families (see \cite[Section 10.3]{JJS}), some of which are split extensions of $D_8$ while others are not. The groups in all these families except the last have exponent $4$, so they certainly satisfy $(\star)$. The last series is a split extension of a generalized quaternion group of order 16 and also satisfies $(\star)$.
\item It follows from \cite[Section 5.2]{JespersSun} that the SSN groups of unfaithful type defined in \Cref{Defintion solv SSN unfaithful type} have no exceptional components of type II. Thus by \Cref{when q(G) finite for grp rings} for these groups $q(G)$, the obstruction to ND, is finite. In  \Cref{sec:DK prop section} we will see more fine properties of that class of groups, which all indicate the difficulty to understand those. 

\end{enumerate}
\end{example}

Now denote for a positive integer $n$ by $\Gamma(n)$ the principal congruence subgroup of level $n$ in $\SL_2(\Z)$, which is the kernel of the reduction modulo $n$ map. Concretely,
$$\Gamma(n) = \left\{ \begin{pmatrix}
1 +  n k_{11} & n k_{12} \\ n k_{21} & 1 + n k_{22}
\end{pmatrix} \in \SL_2(\Z) \mid k_{ij} \in \Z \right\}.$$

We will need the following results which seems to be known to experts, but which we could not find explicitly in the literature.

\begin{lemma}\label{unipotent conjugate}
Let $u$ be a unipotent matrix in $\SL_2(\Z)$. Then $u$ is conjugate inside $\SL_2(\Z)$ to a matrix of the form $\begin{pmatrix}
1 & m \\ 0 & 1
\end{pmatrix}$ with $m \in \Z$.
\end{lemma}
\begin{proof}
Denote $u-1 = \begin{pmatrix}
x & y \\ z & w
\end{pmatrix}$ which is by definition a nilpotent matrix.  We will prove that $u-1$ is conjugate inside $\SL_2(\Z)$ to a matrix of the form $\begin{pmatrix}
0 & m \\ 0 & 0
\end{pmatrix}$ with $m \in \Z$.

To start, by \Cref{lem:FormulaNilpotent2x2Matrices}, $x = -w$ and $x^2 = -yz$. Hence if $z=0$ or $y=0$, then $u-1$ is of the form $\begin{pmatrix}
0 & y \\ 0 & 0
\end{pmatrix}$, respectively $\begin{pmatrix}
0 & 0 \\ z & 0
\end{pmatrix} = S^{-1} \begin{pmatrix}
0 & -z \\ 0 & 0
\end{pmatrix}S$ with $S = \begin{pmatrix}
0 & 1 \\ -1 & 0
\end{pmatrix}.$ Hence we may assume that $y \neq 0$ in which case $u-1 = \begin{pmatrix}
x & y \\ \frac{-x^2}{y} & -x
\end{pmatrix}.$ We will now give the required conjugating matrix explicitly.

Define $\ov{x} = \frac{x}{gcd(x,y)}$ and $\ov{y} = \frac{y}{gcd(x,y)}$. Also take $a,b \in \Z$ such that $a \ov{x} + b \ov{y}=1$. Now define 
$S = \begin{pmatrix}
\ov{y} & a \\ - \ov{x} & b
\end{pmatrix}.$
Note that $S^{-1} = \begin{pmatrix}
b & -a \\ \ov{x} & \ov{y}
\end{pmatrix}$ is in $\SL_2(\Z)$. Thus the following claim would finish the proof of the first part of the statement.

{\it Claim:} $S^{-1} (u-1) S =  \begin{pmatrix}
0 & m \\ 0 & 0
\end{pmatrix}$ for some $m \in \Z$.

\noindent It suffices to prove that $\begin{pmatrix}
\ov{y} \\ -\ov{x}
\end{pmatrix}$ is an eigenvector of $u-1$ with eigenvalue $0$ and that $\begin{pmatrix}
a \\ b
\end{pmatrix}$ is a generalized eigenvector. The former is directly verified using the definition of $\ov{x}$ and $\ov{y}$. For the latter simply note that the columns of $u-1$ are linearly dependent and that the $\Q$-spans of $\begin{pmatrix}
\ov{y} \\ -\ov{x}
\end{pmatrix}$ and $\begin{pmatrix}
y \\ -x
\end{pmatrix}$ are equal. In other words both columns of $u-1$ are eigenvectors and hence every vector is a generalized eigenvector. 
\end{proof}

For a group $\Gamma$ and a subgroup $H\leq \Gamma$ we will denote by $cl_\Gamma(H)$ the normal closure of $H$ in $\Gamma$. This will only be needed in the next lemma and the following proof of \Cref{when q(G) finite for grp rings}.

\begin{lemma}\label{unipotent in congruence of SL2Z}
Let be $H$ a finite index subgroup of $\Gamma(n)$ for some $n$ where $n$ is largest such that $H \leq \Gamma(n)$. Then the quotient $H/ \langle B \in H \mid B \text{ is unipotent } \rangle$ is infinite provided $n \geq 6$. In that case, it is a non-torsion group.

Moreover, if $n$ is a positive integer smaller than or equal to $5$, then the subgroup $cl_{SL_2(\Z)}\left(\langle \begin{pmatrix} 1 & n \\ 0 & 1 \end{pmatrix} \rangle\right)$ has finite index in $\SL_2(\mathbb{Z})$.
\end{lemma}
\begin{proof}
Consider the normal subgroup 
$$N = \langle B \in H \mid B \text{ is unipotent } \rangle$$ of $H$. As $\Gamma(n)$ is a normal subgroup of $\SL_2(\Z)$ \Cref{unipotent conjugate} yields that every generator of $N$ is conjugate to an element of the form $\begin{pmatrix}
1 & m \\ 0 & 1
\end{pmatrix}$ with $m \equiv 0 \mod n$.  
Thus $N$ is a subgroup of $cl_{\SL_2(\Z)}\left( \langle \begin{pmatrix}
1 & n \\ 0 & 1
\end{pmatrix}\rangle\right)$.  Denote the image of this normal closure in $\PSL_2(\Z)$ by $K(n)$.  It is known, e.g. see \cite[Chapter VIII, Section 13]{NewmanBook}, that $\PSL_2(\Z) / K(n)$ is isomorphic to the triangle group $\langle x_1, x_2 \mid x_1^2 = x_2^3 = (x_1 x_2)^n = 1 \rangle$ with parameters $(2,3,n)$. Moreover, this group is infinite if and only if $n \geq 6$. This directly yields the last part of the lemma. Furthermore, since $\Gamma(n)$ is of finite index in $\SL_2(\Z)$ the quotient $\Gamma(n)/ cl_{\SL_2(\Z)}\left( \langle \begin{pmatrix} 1 & n \\ 0 & 1 \end{pmatrix} \rangle \right)$ is infinite if $n \geq 6$. Subsequently, as $H$ is of finite index in $\Gamma(n)$, the quotient $H/N$ is infinite in that case and the description of the conjugacy classes of torsion elements in $\PSL_2(\Z) / K(n)$, see \cite[Theorem 2.10]{Magnus}, also yields that $H/N$ is non-torsion.  
\end{proof}

We can now prove the main result of this section.

\begin{proof}[Proof of \Cref{when q(G) finite for grp rings}.]
Assume the notation of \eqref{eq:WedderburnInSection4}. Take $e \in \PCI (\Q G)_{\geq 2}$ and let $\O_e$ be any maximal order in $D_e$ where  $\Q G e \cong \Ma_{n_e}(D_e)$. Now consider the subgroup $U_e = 1 - e + E_n(J_e)$ of $\langle  \U(\Z G)_{un} \rangle$ given by \Cref{th: upper-bound size q(G)}. As mentioned earlier, if $\Q G e$ is not exceptional, then $E_{n_e}(J_e)$ is a finite index subgroup of $\SL_{n_e}(\O_e)$ (e.g. see \cite{Vas72, BakReh}). Therefore if $\Q G$ has no exceptional components of type II, then the right most bound in \Cref{th: upper-bound size q(G)} is finite, yielding the first part. 

Now suppose that the order of $G$ is bigger than $16$, $G$ maps onto $D_8$ and this surjection satisfies $(\star)$. Then there exists a primitive central idempotent $e$ of $\Q G$ such that $Ge \cong D_8$ and $\Q Ge \cong \Ma_2(\Q)$. Let $\varphi_e : G \rightarrow Ge$ and $Q = \ker(\varphi_e).$ Therefore we can take $g,h \in G$ such that $Ge \cong \langle \varphi_e(g) \rangle \rtimes \langle \varphi_e(h) \rangle$ with $o(gQ) = 4$, $o(hQ) = 2$ and $[gQ,hQ] = (gQ)^2$. Denote $a := \varphi_e(g)$ and $b := \varphi_e(h)$.

One has a ring monomorphism
$$\phi: \Z Ge \rightarrow \Ma_2(\Z)$$ defined by 
\begin{align*}
\quad a \mapsto \left( \begin{array}{cc}
0 & 1  \\
-1 & 0
\end{array}\right),
& \quad ab \mapsto \left( \begin{array}{cc}
1 & 0  \\
0 & -1
\end{array}\right),
\quad b \mapsto \left( \begin{array}{cc}
0 & 1  \\
1 & 0
\end{array}\right)
\end{align*}
with image 
$$\Ima(\phi)= \left\{\left( \begin{array}{cc}
a & b  \\
c & d
\end{array}\right) \in  \Ma_2(\Z) \mid a\equiv d \text{ and } b \equiv c \bmod 2 \right\}.$$

That the image is the latter can be directly verified and is also recorded in \cite[Proposition 8.1]{JJS}. The triple $(gh,h,Q)$ obtained above satisfies the conditions from \cite[Definition 10.5 \& Theorem 10.6]{JJS} and thus one has the associated non-trivial group of $H$-units $\mathcal{H}(gh,h,Q)$. As $G = \langle gh, h, Q \rangle$ we can use \cite[Theorem 10.8]{JJS}, saying that $\mathcal{H}(gh,h,Q) = \SL_1(\Z G) \cap 1-e + \Q Ge.$ In other words $\mathcal{H}(gh,h,Q)$ is the largest subgroup of $\SL_1(\Z G)$ fully contained in $\SL_{2}(\Z)$ (with contained we mean that the subgroup projects trivially on all the other simple components of $\Q G$). Furthermore, $\mathcal{H}(gh,h,Q) = 1 - e + V_{m_e}$ with $m_e = 2 |Q|$ and where 
$$V_{m_e} = \left\{ \left(\begin{array}{cc}
1 + m_e \, l_2 & m_e \, t_1  \\
 m_e \,t_2 & 1 + m_e\, l_1
\end{array} \right) \in \SL_2(\Z) \mid l_1 \equiv l_2 \text{ and } t_1 \equiv t_2 \bmod 2 \right\}$$
 is a subgroup of index $2$ in $\Gamma(m_e)$ and $V_{m_e}$ is a normal subgroup of $\U (\Z G e)$. By \Cref{unipotent in congruence of SL2Z}, the subgroup 
\begin{equation}\label{unipotent in the H-unit sbgrp}
N_e := \langle u \in V_{m_e} \mid u \text{ is unipotent } \rangle
\end{equation}
 is a subgroup of infinite index in $V_{m_e}$ if $m_e \geq 6$. In other words, $N_e$ is of infinite index when $|Q| > 2$. This inequality is satisfied as $16 < |G| = 8 |Q|$. Remark that by the above $N_e = \varphi_e(\langle \mathcal{E}_G(e^{\prime}) : e^{\prime} \in \PCI(\Q G) \rangle)$.

Next we investigate the image of the subgroup $\Bic(G)$ of $\langle  \U(\Z G)_{un} \rangle$ which is generated by the bicyclic units, i.e. all elements of the form $1 + (1-t) v \wt{t}$ or $1 + \wt{t} v (1-t)$ with $t,v \in G$. Concretely, consider the element $u := 1 + \wt{h} \, gh (1- h^{-1}).$ 
One directly sees that 
$$ue  = e + \frac{o(h)}{o(hQ)} (1 + b) \, ab (1- b).$$
A direct computation yields that
$$\phi(ue) = \begin{pmatrix}
1+ \frac{2o(h)}{o(hQ)}  & -\frac{2o(h)}{o(hQ)} \\
\frac{2o(h)}{o(hQ)} & 1 -\frac{2o(h)}{o(hQ)}
\end{pmatrix} \in \Gamma\left(\frac{2o(h)}{o(hQ)}\right).$$
Using the procedure from the proof of \Cref{unipotent conjugate}, we find that $S^{-1} \phi(ue) S = \begin{pmatrix}
1 & -\frac{2o(h)}{o(hQ)} \\ 0 & 1
\end{pmatrix}$ with $S = \begin{pmatrix}
1 & -2 \\ 1 & -1
\end{pmatrix}.$ However, $S$ is not an element in $\varphi_e(\SL_1(\Z G))$, therefore we will take normal closures to finish the argument. More precisely we will make use of the following general group theoretical fact which is easy to prove.

{\it Claim: } Let $K \leq H \leq \Gamma \leq \wh{\Gamma}$ with $K$ normal in $\wh{\Gamma}$ and $H$ normal in $\Gamma$. 
If $[H:K]$ and $[\wh{\Gamma}: \Gamma]$ are finite, then also $[cl_{\wh{\Gamma}}(H): K]$ is finite.

We apply this to $K = N_e$ defined in (\ref{unipotent in the H-unit sbgrp}) which is a normal subgroup in $\mathcal{H}(gh,h,Q)e = V_{m_e}$. As unipotent matrices stay unipotent under conjugation, it is also normal in $\wh{\Gamma} = \U(\Z G e)$. By \cite[Proposition 8.2]{JJS} $\U(\Z G e)$ has index $3$ in $\SL_2(\Z)$.
 We also take $H = \varphi_e(\langle  \U(\Z G)_{un} \rangle)$ which is normal in $\Gamma = \SL_1(\Z G)e$. Now, since $\Z G$ is an order contained in the order $\prod_{f \in \PCI( \Q G)_{div}} \Z G f \times \prod_{e \in \PCI (\Q G)_{\geq 2}} \Ma_{n_e}(\mc{O}_e)$, it follows that the corresponding $\SL_1$ are subgroups of finite index \cite[Lemma 4.6.9 \& Proposition 5.5.1]{GRG1}. Therefore $\Gamma = \SL_1(\Z G)e$ is of finite index in $\wh{\Gamma}=\U(\Z Ge)$. Now suppose that $N_e$ would be of finite index in $\varphi_e(\langle  \U(\Z G)_{un} \rangle$. Then by the above claim $N_e$ would also be of finite index in $cl_{\U(\Z Ge)}(\varphi_e(\langle  \U(\Z G)_{un} \rangle).$ The latter group contains $cl_{\U(\Z Ge)}(\langle \phi(ue) \rangle)$ which is isomorphic to  $cl_{\U(\Z Ge)^S}\left(\langle  \begin{pmatrix}
1 & -4 \\ 0 & 1
\end{pmatrix}\rangle \right)$ which is of finite index in $\SL_2(\Z)$ by \Cref{unipotent in congruence of SL2Z}. Consequently also $N_e$ is of finite index in $\SL_2(\Z)$ which, as noticed earlier, is a contradiction since $m_e \geq 6$.

Altogether we have obtained that $[\varphi_e(\langle  \U(\Z G)_{un} \rangle) : \varphi_e(\langle \mathcal{E}_G(e^{\prime}) : e^{\prime} \in \PCI(\Q G) \rangle) ] $ is infinite. But would $|q(G)|= [\langle  \U(\Z G)_{un} \rangle : \langle \mathcal{E}_G(e^{\prime}) : e^{\prime} \in \PCI(\Q G) \rangle]$ be finite, then so would be the image under $\varphi_e$. Thus indeed $q(G)$ is an infinite group and also non-torsion in view of how we used \Cref{unipotent in congruence of SL2Z}.

Finally assume that $|G| \leq 16$. Following \cite[Remark 3.12.(ii)]{JespersSun} if $G$ has also SN, then it has at most one matrix component and hence has ND. In fact looking at the classification of groups of small order one readily verifies that the only groups of order at most $16$ with more than one matrix component are $D_{12}, D_{16}, D_8 \times C_2$, the semidihedral group $D_{16}^{-} = \langle a,b \mid a^{8} = b^2 = 1, a^b =a^3 \rangle$ and 
$$G(16,3) := \langle a,b \mid a^4 = b^4 = (ab)^2 = 1, (a^2)^b = a^2 \rangle.$$
The last two groups have SmallGroup ID $[16,8]$ and $[16,3]$, respectively. The latter is sometimes given in the literature with the presentation $\langle a,b,c \mid a^2= b^2 = c^4 = [a,b] = [b,c]= 1, c^{a}= bc\rangle.$ 

In case of $D_8 \times C_2$ and $G(16,3)$ the simple matrix components are of the form $\Ma_2(\Q)$. For $D_{16}$ there is also a non-exceptional component of the form $\Ma_2(\Q(\sqrt{2}))$. As explained earlier, for each of their components $\Ma_2(\Q)$ there exists a primitive central idempotent $e$ and a triple $(g,h,Q)$ such that $\mathcal{H}(g,h,Q)e = V_{2|Q|}$ is a subgroup of finite index in $\SL_1(\Z Ge) \leq \SL_2(\Z)$. Since these groups have order $16$ one has that $\mathcal{H}(g,h,Q)e = V_4.$ For the component $\Ma_2(\Q(\sqrt{2}))$ we apply \Cref{th: upper-bound size q(G)} to find a subgroup $U_e$ in $\mc{E}_G(e)$ of the form $E_{2}(J_e)$ with $J_e$ a non-zero ideal in the ring of integers of $\Q(\sqrt{2})$. As $\Ma_2(\Q(\sqrt{2}))$ is not exceptional, $E_{2}(J_e)$ is of finite index in $\SL_1(\Z Ge)$. Summarized, in all these case we find in $\langle \mathcal{E}_G(e) \ | \ e \in \PCI(\Q G) \rangle$ a subgroup which is of finite index in $ \prod_{e \in \PCI (\Q G)_{\geq 2}} \SL_{n_e}(\O_e)$ and hence also in $\langle  \U(\Z G)_{un} \rangle$, as desired. 

For the groups $D_{12}$ and $D_{16}^{-}$ such a subgroup in $\langle \mathcal{E}_G(e) \ | \ e \in \PCI(\Q G) \rangle$ can also be constructed. In case of $D_{12}$ the matrix components are of the form $\Ma_2(\Q)$ and the required subgroups are constructed in the proof of \cite[Theorem 2]{Jespers12}. In the case of $D_{16}^{-}$ the matrix components are both exceptional, namely $\Ma_2(\Q)$ and $\Ma_2(\Q(\sqrt{-2}))$. For the $\Ma_2(\Q)$ component one can use the same argument as for the other groups of order $16$ and for $\Ma_2(\Q(\sqrt{-2}))$ the necessary subgroup is the matrix group from \cite[Theorem 2]{JesParm}.
\end{proof}

\begin{remark}
\begin{enumerate}
\item In \cite[Section 6]{JespersSun} it was stated that $q(G)$ is always torsion. The explanation given there however only yields that for every element $u \in \U(\Z G)_{un}$ there exists an integer $m$ such that $u^m \in \prod_{e \in \PCI(\Q G)}\mathcal{E}_G(e).$ As shown in \Cref{when q(G) finite for grp rings}, in general $q(G)$ is not torsion.
\item One could also consider $RG$ for $R$ an order in some number field $F$. Then $\mathcal{U}(FG) \cong \prod_{i=1}^q\GL_{n_i}(D_i)$ for some finite dimensional division algebras over $F$. Completely analogously one could define a quotient group as (\ref{def: q(G) group}), say $q(\U(RG))$. \Cref{th: upper-bound size q(G)} in fact also holds in this generality and the first part of \Cref{when q(G) finite for grp rings} also. From the description of exceptional components of type II we see that if $F$ is not $\Q$ or an imaginary quadratic extension of $\Q$, then 
$FG$ has no such exceptional components and hence $q(\U(RG))$ is finite. This conclusion for example holds if $R$ contains a primitive $m$-th root of unity with $m$ not a divisor of $4$ or $6$.
\end{enumerate}
\end{remark}

\subsection{A brief look at general semisimple algebraic groups}\label{the obstruction for general smeisimple}
To finish this section we would like to briefly point out that the group $q(G)$ can also be introduced for arithmetic subgroups of more general semisimple algebraic groups than $\mathcal{U}(\Q G)$.

Let $F$ be a number field and $S$ a non-empty finite set of places of $F$ containing the Archimedean places. Associated is the ring of $S$-integers $\mc{O}_S = \{ x\in F \mid |x|_v \leq 1 \text{ for all } v \notin S \}.$ Now consider a linear algebraic $F$-group $\mathbf{G}$ and fix an $F$-embedding $\mathbf{G} \hookrightarrow GL_n(F)$. Using this the group of $S$-integral points is defined as $\mathbf{G}(\mc{O}_S) := \mathbf{G}(F) \cap GL_n(\mc{O}_S)$. A subgroup of $\mathbf{G}(F)$ commensurable with $\mathbf{G}(\mc{O}_S)$ is called an $S$-arithmetic subgroup

Suppose now that $\mathbf{G}$ is a semisimple algebraic group which we also assume to be simply connected. In that case $\mathbf{G}$ is a direct product of simply connected almost-simple algebraic $F$-subgroups \cite[Theorem 2.6]{PlaRapBook}, say $\mathbf{G} = \prod_{i=1}^m \mathbf{G}_i$. Let $\Gamma$ be an $S$-arithmetic subgroup of $\mathbf{G}(F)$.  Analogously as in the case of $\U (\Z G)$, denote by $\Gamma^+$ the group generated by the $F$-rational unipotent elements lying in $\Gamma$ and by $\mc{E}_{\Gamma}(i)$ the subgroup generated by those unipotents projecting only non-trivially in $\mathbf{G}_i(F)$. Then one can define
\begin{equation*}
q(\Gamma) = \Gamma^+ / \prod_{i=1}^q \mc{E}_{\Gamma}(i).
\end{equation*}
Again this group measures to what extend the unipotents of $\Gamma$ have a decomposition in unipotents over $\mc{O}_S$. As in the case of $\U(\Z G)$, the size of $q(\Gamma)$ can be bounded using elementary subgroups of $\mc{E}_{\Gamma}(i)$. In this generality, for an ideal $J$ of $\mc{O}_S$, a principal congruence subgroup is a group of the form $\mathbf{G}(J) := \mathbf{G}(F) \cap \SL_n(J).$ If $U_i^+$ is the unipotent radical of a minimal parabolic $F$-subgroup of $\mathbf{G}_i(F)$ and $U_i^{-}$ the unipotent radical of an opposed (i.e. $U_i^+ \cap U_i^- = \{ 1 \}$)  minimal parabolic subgroup, then the elementary subgroup $E(J_i)$ is the group generated by $U_i^+ \cap \mathbf{G}_i(J)$ and $U_i^- \cap \mathbf{G}_i(J)$. 

The known solutions to the Subgroup Congruence Problem \cite{Ragh,Venkataramana} again yield that each $\mc{E}_{\Gamma}(i)$ contains some $E(J_i)$ which is of finite index, if $S\text{-rank}(\mathbf{G}_i(F)) = \sum_{v \in S} \text{rank}_{F_v}(\mathbf{G}_i(F))$ is at least two. Here  $F_v$ is a local field, the completion of $F$ at $v$, and $\text{rank}_{F_v}(\mathbf{G}_i(F))$ the dimension of a largest split $F_v$-torus. Finally recall that by a theorem of Borel-Tits \cite{BorelTits} $\mathbf{G}_i(F)$ contains non-trivial unipotent elements if and only if $F\text{-rank}(\mathbf{G}_i(F)) \geq 1$. In that case $\mathbf{G}_i(F)$ is called anisotropic. Thus with an analogue reasoning one can obtain the following variant of \Cref{th: upper-bound size q(G)}:
\begin{proposition}\label{if no exceptional then finite for general grp}
Consider the notations above and suppose $ S\text{-rank}(\mathbf{G}_i(F)) \geq 2$ for all anisotropic $\mathbf{G}_i(F).$ Then, $$|q(\Gamma)|< \infty. $$
In particular, in this case finiteness of $q(\Gamma)$ does not depend on the chosen $S$-arithmetic subgroup $\Gamma$.
\end{proposition}

It would be interesting to know for which other types of algebraic groups and arithmetic subgroups, the triviality and finiteness of $q(\Gamma)$ is of significance. In particular, recall that $\Gamma$ is a lattice in the Lie group $\mathbf{G}(\R)$ and the following seems relevant to obtain for example a variant of \Cref{when q(G) finite for grp rings}.

\begin{question}
Does the group $q(\Gamma)$ or its cardinality have a topological interpretation?
\end{question}

\section{Further related properties:  nilpotent decomposition, different kernels and bicyclic resistance}\label{sec:DK prop section}

In this final section we introduce and study two further properties which naturally appeared in our research on the ND property. The first is a property, having different kernels,  which turns out to hold for all groups with one matrix component, but behaves better from a structural perspective. The second property, being bicyclic resistant, can be regarded as a partial ND property. We find that for groups with SN these two new properties are equivalent which also explains some of the hardships we had to endure in the previous sections.
We then finish the paper by some remarks on the connection of bicyclic resistance with the Zassenhaus conjectures, \Cref{Zassenhaus and bicyclic resitant}, as well as in \Cref{sec:MJD} with observations on the Multiplicative Jordan Decomposition and a final question which remains open.

\subsection{Property of having different kernels}\label{section DK examples}
In this section we consider a property which turns out to be satisfied by all groups having at most one matrix component, but which is also a natural property by itself in the context of representation theory over $\Q$. To introduce this property, let $e \in \mathbb{Q}G$ be a central idempotent. Recall that in \eqref{map associated to central id} we defined the homomorphism
\[\varphi_e: G \rightarrow Ge, \ \ g \mapsto ge. \]
If $e$ is the primitive central idempotent corresponding to a given irreducible $\mathbb{Q}$-representation of $G$, then $\ker(\varphi_e)$ equals the kernel of that representation. 

\begin{definition}
A finite group $G$ is said to have the \emph{Different Kernel property}, \emph{DK} in short, if for every orthogonal pair $e,f \in \PCI(\Q G)$ one has $\ker(\varphi_e) \neq \ker(\varphi_f)$. In other words, any two non-equivalent irreducible $\mathbb{Q}$-representations of $G$ have different kernels.
\end{definition}

As we will see in \Cref{subsection on SN versus DK}, property DK is connected to property ND, but behaves better from a structural perspective. We remark that in principle one can define the DK property also over bigger fields than $\mathbb{Q}$. We do not go further in this direction, but note that the classes of groups one considers will be directly restricted by this. E.g. the cyclic group of order $3$ has DK, but does not have the corresponding property over a field containing a primitive 3rd root of unity.

One of our main motivations to introduce this property in the context of this paper is the following result.

\begin{theorem}\label{th:OMCImpliesDK}
Let $G$ be a finite group such that $\Q G$ has at most one matrix component. Then $G$ has DK.
\end{theorem}

Before proving this we make some other interesting observations.
We first reformulate DK as a condition on the set of primitive central idempotents.

\begin{proposition}\label{DK via epsilon}
Let $G$ be a finite group. Then $G$ has DK if and only if $\PCI(\Q G) \subseteq \{ \epsilon(G,N) \mid N \unlhd G \}$. In that case $\PCI(\Q G)  = \{  \epsilon(G,\ker(\varphi_e)) \mid e \in \PCI(\Q G) \}.$
\end{proposition}
\begin{proof}
Given a normal subgroup $N$ of $G$ we define the set $\mc{I}(N) = \{ e \in \PCI(\Q G) \mid N \subseteq \ker(\varphi_e) \}$ which corresponds to the irreducible $\Q$-representations containing $N$ in their kernel. Note that $\Q G \wh{N} = \bigoplus\limits_{e \in \mc{I}(N)} \Q G e$ and hence $\Q G (1 - \wh{N}) = \bigoplus\limits_{e \in \PCI(\Q G) \setminus \mc{I}(N)} \Q G e$.
Therefore, using the most right form of $\epsilon(G,N)$ in \eqref{Def epsilon idempotents}, we see that  by construction $\epsilon(G,N)$ is the central idempotent which corresponds to exactly those irreducible $\mathbb{Q}$-representations which have kernel equal to $N$. Thus $\epsilon(G,N)$ is not primitive, say $e, f \in \PCI(\Q G)$ are orthogonal summands of $\epsilon(G,N)$, if and only if $N = \ker(\varphi_e) = \ker(\varphi_f)$, i.e. if and only if $G$ does not have DK. 
\end{proof}

A direct consequence together with \cite[Corollary 3.3.3]{GRG1} is:
\begin{corollary}\label{cor:AbelianHaveDK}
Abelian groups have DK.
\end{corollary}

This also gives:

\begin{lemma}\label{lem:DKavoidsCommutativeComponents}
Let $e,f \in \PCI(\Q G)$ such that $\Q Ge$ and $\Q Gf$ are commutative. Then $\ker(\varphi_e) = \ker(\varphi_f)$ if and only if $e=f$.  Consequently, if there is a unique $e \in \PCI(\Q G)$ such that $\Q Ge$ is not commutative, then $G$ has DK.
\end{lemma}
\begin{proof}
If $e\neq f$ but $\ker(\varphi_e) = \ker(\varphi_f)$, then the group $G/G'$ would not have DK, contradicting \Cref{cor:AbelianHaveDK}.

Next assume that $e$ is the unique idempotent with $\Q Ge$ not commutative. In other words, $G'  \not \subseteq \ker(\varphi_e)$ but $G' \subseteq \ker(\varphi_f)$ for every other $f \in \PCI(\Q G)$.  Moreover the first part tells that the primitive central idempotents different from $e$ have also different kernels, hence altogether $G$ has DK.
\end{proof}

\begin{example}\label{ex:DKForQ8AndFriends}
$Q_8$, $D_8$ and $A_4$ all have DK. Indeed this follows directly from the preceding lemma as
\[\Q Q_8 \cong 4\Q \oplus \mathbb{H}_{\Q}, \ \ \Q D_8 \cong 4\mathbb{Q} \oplus M_2(\Q), \ \ \Q A_4 \cong \Q \oplus \Q(\zeta) \oplus M_3(\Q), \]
where $\mathbb{H}_{\Q}$ denotes the rational quaternions and $\zeta$ a primitive 3rd root of unity.
\end{example}

\begin{example}\label{ex:SL23}
The two smallest groups not to have DK are the symmetric group $S_4$ and $\operatorname{SL}(2,3)$. This is clear for $S_4$: the natural permutation representation from which the trivial submodule has been canceled is an integral irreducible representation. But so is its twist by the sign representation $S_4 \rightarrow \{\pm1 \}$.

 Setting $G = \operatorname{SL}(2,3) \cong Q_8 \rtimes C_3$, one has
\[\mathbb{Q}G \cong \mathbb{Q}A_4 \oplus \mathbb{H}_\Q \oplus M_2(\mathbb{Q}(\zeta)) \cong \mathbb{Q}C_3 \oplus M_3(\Q) \oplus \mathbb{H}_\Q \oplus M_2(\mathbb{Q}(\zeta)), \]
where $\mathbb{H}_\Q$ denotes the rational quaternions and $\zeta$ a primitive third root of unity. The representations of $G$ corresponding to $\mathbb{H}_\Q$ and $M_2(\mathbb{Q}(\zeta))$ are both faithful, i.e. have trivial kernel.
\end{example}

It might be tempting to attempt a proof of \Cref{th:OMCImpliesDK} by assuming $G$ does not have DK and taking orthogonal $e,f \in \PCI(\Q G)$ such that $\ker(\varphi_e) = \ker(\varphi_f)$ and such that $\Q Ge$ and $\Q G f$ are both matrix components. However, \Cref{ex:SL23} shows that this cannot be assumed in general. Hence we have to follow another strategy.

\begin{proposition}\label{prop:DKDirectProductsAndQuotients}
Let $N \unlhd G$ and $H$ be a finite groups. Then the following hold.
\begin{enumerate}
\item If $G$ has DK, then $G/N$ has DK.
\item If $G$ and $H$ have DK and the orders of $G$ and $H$ are coprime, then $G \times H$ has DK.
\end{enumerate}
\end{proposition}
\begin{proof}
Every irreducible $\Q$-representation of $G/N$ is also an irreducible $\Q$-representation of $G$, so the first item follows.

For the second claim recall that $\Q[G \times H] \cong \Q G \otimes \Q H$. If $e \in \PCI(\Q G)$ and $f \in \PCI(\Q H)$ then $e \otimes f$ is a central idempotent in $\Q[G \times H]$. If $e' \in \PCI(\Q G)$ such that $ee'= 0$, then $(e \otimes f)(e' \otimes f) = 0$, so different central idempotents obtained in this way are orthogonal. We claim that all these idempotents are in fact central primitive. Indeed the number of primitive central idempotents in $\Q(G\times H)$ is the same as the number of conjugacy classes of cyclic subgroups in $G \times H$ by \cite[Corollary 7.1.12]{GRG1}. As the orders of $G$ and $H$ are coprime, this is the same as the product of the numbers of conjugacy classes of cyclic subgroups of $G$ and $H$, so we have 
\begin{align}\label{eq:PCIDirectProduct}
|\PCI(\Q[G \times H])| = |\PCI(\Q G)|\cdot |\PCI(\Q H)|.
\end{align}
As we saw above every primitive summand of $e \otimes f$ is orthogonal with every primitive summand of $e' \otimes f$. So if $e \otimes f$ would not be primitive, this would contradict \eqref{eq:PCIDirectProduct}.

It remains to show that $\ker(\varphi_{e \otimes f}) \neq \ker(\varphi_{e' \otimes f})$ which will follows from the assumptions by showing $\ker(\varphi_{e\otimes f}) = \ker(\varphi_e) \times \ker(\varphi_f) \leq G \times H$. For this recall that the representation corresponding to $e \otimes f$ can be obtained as the Kronecker product between the representations corresponding to $e$ and $f$. Denote by $I$ the identity matrix (abusing notation it will have varying size). Now if $A \otimes B = I$ for $A$ and $B$ matrices of finite order, then $A$ and $B$ must be diagonal and $dB=I$ for any element $d$ on the diagonal of $A$. But if $A$ is a matrix coming from a representation of $G$ and $B$ a matrix coming from a representation of $H$, then the orders of $A$ and $B$ are coprime. So $dB = I$ implies $d =1$ and $B = I$, in total $A=I$ and $B=I$. 
\end{proof}

\begin{example}\label{ex:DirectProductNoDK}
We show that property DK is not closed under taking direct products when the orders of the factors are not assumed to be of coprime order and that \eqref{eq:PCIDirectProduct} is also not correct in general. For this let $G = D_{10} \times C_5$. If $G' = \langle a \rangle$ and $b$ is a central element of order $5$, then from \Cref{lem:IdempotForMetabelian} we deduce that $(\langle a,b \rangle , \langle ab \rangle)$ and $(\langle a,b \rangle , \langle a^2b \rangle)$ are strong Shoda pairs which correspond to non-equivalent faithful $\Q$-representations.
\end{example}

\begin{example}
We show that property DK is not closed under taking subgroups. First consider the following group which has GroupId [32,11] in the SmallGroupLibrary\cite{SmallGroupLibrary}:
\[H = \langle a,b,c \ | \ a^4=b^4=c^2=1, [a,b]=[b,c]=1, [a,c]=a^2b \rangle \cong (C_4 \times C_4) \rtimes C_2. \]
Using \Cref{lem:IdempotForMetabelian} we see that  in $H$ the strong Shoda Pairs $(\langle a,b \rangle, \langle a \rangle)$ and $(\langle a,b \rangle, \langle ab \rangle)$ provide two different elements in $\PCI(\Q H)$ such that both corresponding representations are faithful, i.e. $H$ does not have DK.

Now consider the group
\begin{align*}
G &= \langle a,b,c,d \ | \ a^4=b^4=c^2=d^2= [a,b]=[b,c]=1, [a,c]=a^2b, [a,d] = a^2b^2, [b,d]=b^2, [c,d] = a^2b^{-1} \rangle \\
&\cong ((C_4\times C_4) \rtimes C_2) \rtimes C_2 \cong H \rtimes C_2
\end{align*}
which has GroupId [64, 135]. As this group is metabelian, we can apply \Cref{lem:IdempotForMetabelian} with $A = \langle a, b\rangle$. We list the strong Shoda pairs which provide all the non-commutative components of $\Q G$ one obtains in this way without further details. Note that $G' = \langle a^2, b \rangle$:\\
\begin{align*}
&(\langle A, c \rangle, \langle a^2b, c \rangle), \ (\langle A, c \rangle, \langle a^2b, a^2c \rangle), \ (\langle A, d \rangle, \langle a,b^2, d \rangle), (\langle A, d \rangle, \langle ab, b^2, d \rangle),  \\  
&(\langle A, cd \rangle, \langle b, cd \rangle), \ (\langle A, cd \rangle, \langle b, a^2cd \rangle), \ (\langle A, d \rangle, \langle a \rangle).
\end{align*}
We compute the kernels of the corresponding representations by \Cref{form of SSP idempotent result}. These are $\langle a^2b, c \rangle$, $\langle a^2b, a^2c \rangle$, $\langle ad, a^2 \rangle$, $\langle abd, a^2 \rangle$, $\langle b, cd \rangle$, $\langle b, a^2cd \rangle$ and $1$ respectively. Hence all kernels are different and $G$ has DK. 
\end{example}

\begin{lemma}\label{lem:DKForCmQ}
Let $m$ be an integer, $q$ a prime not dividing $m$ and $G = C_m \rtimes Q$ a non-trivial semi-direct product where $Q$ is a $q$-group which is either abelian or generalized quaternion. Let $a \in G$ be of order $m$. In case $Q$ is generalized quaternion, assume that a maximal cyclic subgroup of $Q$ acts trivially on $\langle a \rangle$. 
Moreover, if $[a,g] \neq 1$ for some $g \in Q$, then $[\langle a \rangle, \langle g \rangle] = \langle a \rangle$ holds. Then $G$ has DK.  
\end{lemma}
\begin{proof}
$G$ is metabelian and hence we can apply \Cref{lem:IdempotForMetabelian} looking for strong Shoda pairs $(H,K)$ in $G$. By \Cref{lem:DKavoidsCommutativeComponents} we can restrict our attention to those satisfying $G' \not\subseteq K$ and we will further assume this condition. Let $A$ be a maximal abelian subgroup of $G$ containing $G'$. As $[A,\langle g \rangle] = G'$ for every $g \notin A$, we get $A = H$. When $Q$ is generalized quaternion we can write $A = \langle a \rangle \times \langle b \rangle$, where $\langle b \rangle$ is a maximal cyclic subgroup of $Q$. If $Q$ is abelian we have $A = \langle a \rangle \times (\ZZ(G) \cap Q)$. In any case, the condition that $q$ does not divide $m$ implies that every subgroup of $A$ is normal in $G$. Hence when $(A,K)$ is a strong Shoda pair and $e = e(G,A,K)$, we have $\ker(\varphi_e) = K$ by \Cref{form of SSP idempotent result}. In particular, each irreducible $\Q$-representation of $G$ is uniquely determined by its kernel and $G$ has DK.
\end{proof}

With this we can show DK for some interesting classes of groups.

\begin{corollary}\label{cor:DKForUnfaithful}
Let $G$ be an SSN group of unfaithful type. Then $G$ has DK.
\end{corollary}

\begin{proposition}\label{prop:DKForAmitsur}
Let $G$ be a finite subgroup of the multiplicative group of a division algebra in characteristic $0$. Then $G$ has DK if and only if it is not isomorphic to one of the following:
\begin{enumerate}
\item the binary octahedral group,
\item $\operatorname{SL}(2,5)$,
\item $\operatorname{SL}(2,3) \times H$ for $H$ a group of order coprime to $6$.
\end{enumerate}  
\end{proposition}
\begin{proof}
The finite subgroups of division algebras in characteristic $0$ were obtained by Amitsur, we refer to \cite[Theorem 2.1.4, 2.1.5]{ShirvaniWehrfritz} for a full account. It follows that when $G$ is not one of the three possibilities listed explicitly in the statement, then $G$ is the direct product of groups of coprime orders such that each factor has the shape given in \Cref{lem:DKForCmQ}. So by \Cref{lem:DKForCmQ} and \Cref{prop:DKDirectProductsAndQuotients} we conclude that $G$ has DK. It remains to show that this is not the case for the three cases listed.

If $G$ is the binary octahedral group, then $G/\ZZ(G) \cong S_4$, so that $G$ does not have DK by \Cref{prop:DKDirectProductsAndQuotients} and \Cref{ex:SL23}. Similarly, if $G \cong \operatorname{SL}(2,3) \times H$, then $G$ maps onto $\operatorname{SL}(2,3)$, so we can again use \Cref{prop:DKDirectProductsAndQuotients} and \Cref{ex:SL23}. For $G = \operatorname{SL}(2,5)$ we observe that $G$ maps onto a non-abelian simple group, namely $A_5$. But a non-abelian simple group can never have DK, indeed otherwise $\Q G$ would only have two components, but $G$ has certainly more than two conjugacy classes of cyclic subgroups, which would contradict \cite[Corollary 7.1.12]{GRG1}.
\end{proof}

\begin{proof}[Proof of \Cref{th:OMCImpliesDK}.]
To start we reduce the statement to the case that $G$ embeds in a division algebra of finite dimension over $\Q$. Let $G$ be a group of minimal order violating the conditions, i.e. $G$ is a group with at most one matrix component, but there exist orthogonal $e,f \in \PCI(\Q G)$ such that $\ker(\varphi_e) = \ker(\varphi_f)$. Set $N =\ker(\varphi_e)$. Then $G/N$ is also a group with at most one matrix component, which does not have DK, namely it has two non-equivalent faithful representations. By the minimality of $G$ we conclude that $N=1$. By \Cref{lem:DKavoidsCommutativeComponents} we know that neither $\Q Ge$ nor $\Q Gf$ is a field. On the other hand at most one of them, say $\Q Ge$, can be a matrix-component. Hence $\Q Gf$ is a non-commutative division algebra $D$ and as $\ker(\varphi_f) = 1$, it follows that $G$ is isomorphic to a multiplicative subgroup of $D$. 

So we assume that $G$ is a subgroup of a division algebra of characteristic $0$. By \Cref{prop:DKForAmitsur} many of those groups have DK independently from the property of having one matrix component and we will be done once we see that the three exceptions listed in the proposition do not have one matrix component. By \Cref{ex:SL23} this is true for $\operatorname{SL}(2,3)$ and also $S_4$, which is the image of the binary octahedral group. Also $\mathbb{Q}\operatorname{SL}(2,5)$ contains a direct summand isomorphic to $\Q A_5$, which has more than one matrix component. 
\end{proof}

We next show that another class of groups of interest in this paper has DK.

\begin{lemma}\label{lem:DKForNilpotentSSN}
Let $G$ be a nilpotent group with SSN. Then $G$ has DK.
\end{lemma}
\begin{proof}
Assume first that $G$ is a Dedekind group. As abelian groups have DK by \Cref{cor:AbelianHaveDK} and $Q_8$ has DK by \Cref{ex:DKForQ8AndFriends}, the property DK for $G$ follows from \Cref{prop:DKDirectProductsAndQuotients}.

So we can assume that $G$ is one of the nine classes (BJ1)-(BJ9) listed in \cite[Theorem 4.1]{JespersSun}. The groups in (BJ3) are a direct product of a quaternion group of order $8$ and a cyclic group of odd order, so they have DK by the same argument as Dedekind groups. The groups (BJ2), (BJ6), (BJ7) have one matrix component by \cite[Lemma 4.5 \& page 11]{JespersSun}, so they have DK by \Cref{th:OMCImpliesDK}. It remains to study the groups in (BJ1), (BJ4), (BJ5), (BJ8) and (BJ9). All those groups are metabelian and so we can apply \Cref{lem:IdempotForMetabelian} to show that they have DK and by \Cref{lem:DKavoidsCommutativeComponents} we can consider only strong Shoda pairs $(H,K)$ such that $K$ does not contain the commutator subgroup. For all groups we will list a full set of non-equivalent strong Shoda pairs based on \Cref{lem:IdempotForMetabelian} and the kernels of the corresponding representations which follow from \Cref{form of SSP idempotent result}. It will follow that kernels are pairwise different and the groups have DK.

\begin{itemize}
\item[(BJ4)] We have, cf. \cite[p. 120]{JespersSun},
\[G = \langle a,b,c \ | \ a^9 = b^3 = [a,b] = 1, a^c = ab, b^c = a^{-3}b, c^3 = a^3 \rangle, \]
so $G' = \langle a^3, b \rangle$. We let $A = \langle a, b \rangle \cong C_9 \times C_3$ be a maximal abelian subgroup containing $G'$. As $A$ is a maximal subgroup of $G$, we have $H=A$. The conjugacy classes of subgroups of $A$ which have cyclic quotients and do not contain $G'$, i.e. which can play the role of $K$ in the strong Shoda pair $(H,K)$, are $\{\langle b \rangle, \langle a^{-3}b \rangle, \langle a^3b \rangle \}$ and $\{\langle a \rangle, \langle ab \rangle, \langle a^{-1}b \rangle \}$. The corresponding kernels of the representations, i.e. $\text{core}_G(K)$, are $\langle a^3 \rangle$ and $1$ respectively.
\item[(BJ5)] We have
\[G = \langle a,b \ | \ a^8 = 1, a^b = a^{-1}, a^4 = b^4 \rangle, \]
so $G' = \langle a^2 \rangle$. Let $A = \langle a, b^2 \rangle \cong C_8 \times C_2$. As $A$ is a maximal subgroup of $G$, we have $H=A$. The conjugacy classes of subgroups of $A$ which have cyclic quotients and do not contain $G'$ are $\{\langle a^2b^2 \rangle, \langle a^{-2}b^2 \rangle \}$ and $\{\langle a^4b^2 \rangle \}$. The corresponding kernels of the representations are $1$ and $\langle b^2 \rangle$ respectively.
\item[(BJ8)] We have
\[G = \langle a,b,c \ | \ a^4 = b^4 = [a,b]= 1, a^c = ab^2, b^c=ba^2, c^2=a^2 \rangle, \]
so $G' = \langle a^2, b^2 \rangle$. Let $A = \langle a, b \rangle \cong C_4 \times C_4$. As $A$ is a maximal subgroup of $G$, we have $H=A$. The conjugacy classes of subgroups of $A$ which have cyclic quotients and do not contain $G'$ are $\{\langle a \rangle, \langle ab^2 \rangle \}$, $\{\langle b \rangle, \langle ba^2 \rangle \}$, $\{\langle ab \rangle\}$ and $\{\langle a^{-1}b \rangle\}$. The corresponding kernels of the representations are $\langle a^2 \rangle$, $\langle b^2 \rangle$, $\langle ab \rangle$ and $\langle a^{-1}b \rangle$ respectively.
\item[(BJ9)] We have
\[G = \langle a,b,c,d \ | \ a^4 = b^4 = [a,b]= 1, a^c = a^{-1}, b^c=b^{-1}a^2, a^d = a^{-1}b^2, b^d = b^{-1}, c^2=a^2b^2, d^2=a^2 \rangle, \]
so $G' = \langle a^2, b^2 \rangle$. Let $A = \langle a, b \rangle \cong C_4 \times C_4$. In this case $A$ is a not a maximal subgroup of $G$, but as all the proper subgroups containing it, namely $\langle A, c \rangle$, $\langle A, d \rangle$ and $\langle A, cd \rangle$, have derived subgroup $G'$, we still have $H=A$. The conjugacy classes of subgroups of $A$ which have cyclic quotients and do not contain $G'$ are $\{\langle a \rangle, \langle ab^2 \rangle \}$, $\{\langle b \rangle, \langle ba^2 \rangle \}$ and $\{\langle ab \rangle, \langle a^{-1}b \rangle\}$. The corresponding kernels of the representations are $\langle a^2 \rangle$, $\langle b^2 \rangle$ and $\langle a^2b^2 \rangle$ respectively.
\item[(BJ1)] We have for $p$ a prime, $m\geq 2$ and $n\geq 1$
\[G = \langle a,b \ | \ a^{p^m} = b^{p^n} = 1, a^b = a^{1+p^{m-1}} \rangle, \]
so $G' = \langle a^{p^{m-1}} \rangle$. Let $A = \langle a, b^p \rangle \cong C_{p^m} \times C_{p^{n-1}}$. As $A$ is a maximal subgroup of $G$, we have $H=A$. The  subgroups of $A$ which have cyclic quotients are $K = \langle a^jb^p \rangle$ for some integer $j$ such that the order of $a^j$ is at most $p^{n-1}$. If $p$ divides $j$, then $a^jb^p \in \ZZ(G)$ and $K$ is itself the kernel of the corresponding representation. If $p$ does not divide $j$, then the kernel is $\langle a^{jp}b^{p^2} \rangle$. If $k$ is also a number not divisible by $p$ such that $\langle a^{jp}b^{p^2} \rangle = \langle a^{kp}b^{p^2} \rangle$, then $j \equiv k \bmod p^{m-1}$ so that $a^jb^p$ is conjugate to $a^kb^p$ and hence the corresponding $K$ give equivalent strong Shoda pairs.
\end{itemize}
\end{proof}

\subsection{Nilpotent decomposition with specific idempotents or nilpotents}\label{subsection on SN versus DK}

The proof of \Cref{ND for non-nilp SSN grps} works by constructing a particular type of nilpotent element $n \in \mathbb{Z}G$, which we will call bicyclic nilpotent, and a central idempotent $e \in \Q G$ such that $ne \notin \mathbb{Z}G$ if and only if $G$ has more than one matrix component. We formalize this in the following way.

\begin{definition}
For elements $g,h \in G$ and $H$ a subgroup of $G$ containing $h$ we call $(1-h)g\widetilde{H}$ and $\widetilde{H}g(1-h)$ a \emph{bicyclic nilpotent} element.

We call $G$ \emph{bicyclic resistant}, if for every bicyclic nilpotent element $n \in \mathbb{Z}G$ and every central idempotent $e \in \Q G$ one has $ne \in \mathbb{Z}G$.
\end{definition}

Interestingly, a group having SN will have DK exactly when all the bicyclic nilpotent elements have a nilpotent decomposition. More precisely, in the remainder of the section we will work towards proving the following result. 

\begin{theorem}\label{When SN is DK theorem}
Let $G$ be a finite group with SN. Then the following are equivalent:
\begin{enumerate}
\item $G$ is bicyclic resistant.
\item $G$ is supersolvable or $\Q G$ has one matrix component.
\item $G$ has DK.
\end{enumerate}
\end{theorem}

The methods of the proof of \Cref{ND for non-nilp SSN grps} in fact suggest that it might be interesting to study analogues of the property ND only considering certain nilpotent elements and certain central idempotents. 

\begin{definition} 
Let $E$ be a set of central idempotents in $\Q G$ and $n \in \mathbb{Z}G$ a nilpotent element. We say that $n$ \emph{has ND with respect to} $E$, if $ne \in \mathbb{Z}G$ holds for every $e \in E$. 
\end{definition}

With this terminology at hand we can give a new characterization of property SN in terms of such kind of local ND. This characterization implies that bicyclic resistant groups have SN. 

\begin{proposition}\label{SN via epsilon idemoptent}
Let $G$ be a finite group. The following are equivalent:
\begin{enumerate}
\item $G$ has SN.
\item All bicyclic nilpotent elements have ND with respect to $\{\epsilon(G,N) \mid N \unlhd G \}$.
\item All bicyclic nilpotent elements have ND with respect to $\{\wh{N} \mid N \unlhd G \}$.
\end{enumerate}
\end{proposition}
\begin{proof}
Let $Y \leq G$, $x\in G$, $y\in Y$ and denote $n = (1-y)x\widetilde{Y}$. Remark that  $n \wh{N} = (1-y).x. \wt{\langle N, Y\rangle}.\frac{|Y \cap N|}{|N|}$. This implies that one can choose $N$ such that $0\neq n \wh{N}$ exactly when $YN$ is not normal, i.e. there exists a non-trivial $x \notin N_G(YN)$. Moreover, when $0\neq n \wh{N}$, it is in $\Z G$ exactly when $|Y \cap N| = |N|$. In other words, when $N \leq Y$. These two observations combined imply the equivalence between (1) and (3). 

To see that (2) and (3) are equivalent note first that 
\[ n\epsilon(G,N) = n\wh{N} \prod_{M/N \in  \mathcal{M}(G/N)} (1 - \wh{M}) = \sum_{M \unlhd G} a_M n\wh{M}\]
for certain integers $a_M$. If (3) holds, then $n\wh{M} \in \mathbb{Z}G$ for every $M\unlhd G$ and consequently (2) holds. To see that (2) implies (3) we argue by induction on the minimal length of a chain of normal subgroups from $N$ to $G$. For the induction start notice $n\epsilon(G,G) = n\wh{G}$. Now let $N \unlhd G$. Then
\[ n\epsilon(G,N) = n\wh{N} \prod_{M/N \in  \mathcal{M}(G/N)} (1 - \wh{M}) = n \wh{N} + \sum_{N \lneq M \unlhd G} a_M n\wh{M},\]
for certain integers $a_M$, is an element of $\mathbb{Z}G$. As $\sum_{N \lneq M \unlhd G} a_M n\wh{M}  \in \mathbb{Z}G$ by induction, we conclude $n\wh{N} \in \mathbb{Z}G$.
\end{proof}

\Cref{SN via epsilon idemoptent} combined with \Cref{DK via epsilon} now yield the following.

\begin{corollary}\label{prop SN for a DK group}
Let $G$ be a finite group with DK. Then $G$ has SN if and only if it is bicyclic resistant.
\end{corollary}

We show that some other classes of interest are also bicyclic resistant using the following lemma.

\begin{lemma}
Let $p$ and $q$ be primes and $G$ a semi-direct product $P \rtimes Q$ of a cyclic $p$-group $P$ and a cyclic $q$-group $Q$ such that $G'$ is cyclic of prime order. Then $G$ is bicyclic resistant.
\end{lemma} 
\begin{proof}

Let $a,b\in G$ such that $\langle a \rangle = P$ and $\langle b \rangle = Q$. We will separate two cases which only differ in technical details though.

Assume first that $ p = q$. Then our conditions imply that the action of $Q$ on $P$ is of order $p$, i.e. $b^p \in \ZZ(G)$.
To construct a non-trivial bicyclic nilpotent element in $\mathbb{Z}G$ we need to find $g,u \in G$ and a subgroup $U$ of $G$ containing $u$ such that $u^g \notin U$. In the present conditions the only elements which generate non-normal cyclic subgroups of $G$ are those of shape $\langle a^ib \rangle$ for some integer $i$. Any subgroup of $G$ containing $\langle a^ib \rangle$ properly will also contain $G'$ and hence be normal. So, up to left-right symmetry, the only non-trivial bicyclic nilpotent elements in $\mathbb{Z}G$ are of shape $(1-a^ib)g\widetilde{\langle a^ib \rangle}$. We fix such a generic element $n \in \mathbb{Z}G$.

The group $G$ is metabelian, so we can use \Cref{lem:IdempotForMetabelian} to construct all the elements of $\PCI(\Q G)$. Fix $A = \langle a, b^p \rangle$, a maximal abelian subgroup of $G$ containing $G'$. Assume $e \in \PCI(\Q G)$ with $e = e(G,H,K)$. If $ne \neq 0$, then $K$ does not contain $G'$. On the other hand $K$ does contain $H'$ and $H$ contains $A$, which implies $H=A$. Set $S = \langle b^p \rangle$. Then we can write $\widetilde{\langle a^ib \rangle} = g_1\widetilde{S} + ... + g_n\widetilde{S}$ for $g_1,...,g_n$ a transversal of $S$ in $\langle a^ib \rangle$. So $ne \neq 0$ implies $\widetilde{S}e \neq 0$. As $S$ is a central cyclic group and the sum of all the roots of unity of the same order equals $0$, this implies $S \leq \ker(\varphi_e)$ and so $S \leq K$. As $S$ is a maximal subgroup of $A$ among those not containing $G'$, we obtain $K=S$. So $e$ is uniquely determined by the property $ne \neq 0$. Hence for every $f \in \PCI(\Q G)$ one has $nf = 0 $ or $nf=n$. Overall, $G$ is bicyclic resistant.

Next assume $p \neq q$. Then our conditions imply that $P$ has order $p$. Similarly as in the previous case the only elements of $G$ which do not generate normal cyclic subgroups are those of shape $\langle a^ib^j \rangle$ for some integers $i$ and $j$ such that $b^j \notin \ZZ(G)$. A subgroup of $G$ containing $\langle a^ib^j \rangle$ either contains $G'$ or will be a cyclic $q$-group. So a generic bicyclic nilpotent element $n$ can be written as $(1-a^ib^j)g\widetilde{R}$, where $R$ is a cyclic $q$-group containing $a^ib^j$. Again we want to use \Cref{lem:IdempotForMetabelian}. Let $S = \ZZ(G) \cap Q$. Then $A = \langle a \rangle \times S$ is a maximal abelian subgroup of $G$ containing $G'$. As before choosing $e = e(G,H,K)$ one concludes $H=A$. Moreover we note that $S \leq R$, so we can again write $\widetilde{R} = g_1\widetilde{S} + ... + g_n\widetilde{S}$ for $g_1,...,g_n$ a transversal of $S$ in $R$. So $ne \neq 0$ implies $\widetilde{S} \neq 0$, but this is only possible if $S \leq K$. This means $e = e(G,A,S)$, so the element of $\PCI(\Q G)$ which satisfies $ne \neq 0$ is unique. 
\end{proof}

This implies on one hand that the work carried out for the proof of \Cref{th:NilpotentCase} could not be carried out using bicyclic nilpotent elements, as well as that these elements cannot serve to solve the remaining case of SSN groups of unfaithful type.

\begin{corollary}\label{the hardest cases are resistant}
The groups $G(p,m,n)$, defined in \Cref{sec:NDForNilpotent}, are bicyclic resistant. Also, the SSN groups of unfaithful type are bicyclic resistant.
\end{corollary}

We are finally ready to describe which groups with SN are bicyclic resistant.

\begin{proof}[Proof of \Cref{When SN is DK theorem}.]
Following \Cref{prop SN for a DK group} we know that (3) implies (1). Next suppose (1), i.e. $G$ is bicyclic resistant. Since SSN groups of unfaithful type and nilpotent groups are supersolvable, it remains to consider the groups dealt within \Cref{ND for non-nilp SSN grps}. The proof of \Cref{ND for non-nilp SSN grps} in fact constructs a bicyclic nilpotent element $n \in \mathbb{Z}G$ and a central idempotent $e \in \Q G$ such that $ne \notin \mathbb{Z}G$ if and only if $G$ has more than one matrix component. In other words, those groups are bicyclic resistant if and only if $G$ has one matrix component, which finishes the proof that (1) implies (2).

Now suppose (2). If $\Q G$ has one matrix component, then it has DK by \Cref{th:OMCImpliesDK}. Therefore we may assume that $\Q G$ has more than one matrix component and is supersolvable. If $G$ is even nilpotent, then by \Cref{prop:NilpotentSNNotSSN} the group $G$ has SSN and so also DK by \Cref{lem:DKForNilpotentSSN}.
It remains to consider the case that $G$ is supersolvable but not nilpotent. It is easily verified that the group $P \rtimes H$ with $H$ acting irreducibly and faithfully as in \Cref{prop:SolvSNGroups} is supersolvable if and only if $P$ is cyclic and so also $H$ is cyclic. Using \Cref{lem:DKForCmQ} we now see that   supersolvable not nilpotent SN groups have DK. 
\end{proof}

\begin{remark}\label{Zassenhaus and bicyclic resitant}
One could wonder in how far being bicyclic resistant is a property of the group ring $\mathbb{Z}G$ defined independently of the group basis $G$. In general this is not clear, but at least for those groups where a positive answer to the second Zassenhaus conjecture is known, this is the case. Recall that the second Zassenhaus conjecture asked, if it is true that when $H$ is a group of normalized units of $\mathbb{Z}G$ of the same order as $G$, there necessarily exists a unit $x \in \Q G$ such that $H^x = G$. It is clear that if such a unit exists the bicyclic nilpotent elements which can be defined using the elements of $G$ are conjugate in $\Q G$ to those which can be defined using $H$. As the central idempotents of $\Q G$ do not change under conjugation of course, it follows that in this situation being bicyclic resistant does not depend on the chosen group basis. More strongly one could even take any two units of $\mathbb{Z}G$ which generate a subgroup of finite order to construct a bicyclic nilpotent. This will also not break bicyclic resistance at least when the third Zassenhaus conjecture has a positive answer for $G$, i.e. if every finite subgroup of units in $\mathbb{Z}G$ is conjugate in the units of $\Q G$ to a subgroup of $\pm G$.

We remark that nilpotent groups are known to satisfy the third Zassenhaus conjecture \cite{Weiss1991} as well as metacyclic groups $A \rtimes B$ when $A$ and $B$ have coprime orders \cite{Valenti94}. So neither could we have constructed bicyclic nilpotent elements with respect to any finite subgroup of units of $\mathbb{Z}G$ for the groups $G(p,m,n)$ in \Cref{sec:NDForNilpotent} to prove \Cref{th:NilpotentCase}, nor will this be possible to resolve ND for SSN groups of unfaithful type. 

These observations might lead to wonder, if in fact the third Zassenhaus Conjecture might hold for all groups with DK. This is however not the case: it can be checked that the counterexample to the conjecture presented in \cite{HertweckAnother} does have property DK.
\end{remark}

\subsection{Concluding remarks on the Jordan decomposition}\label{sec:MJD}
The motivation of the work of Jespers-Sun \cite{JespersSun} was to contribute to the precise classification of groups having Multiplicative Jordan Decomposition. Though many contributions have been made here, the complete classification remains elusive. We refer to \cite{HalesPassi17} for a survey and to \cite{WangZhou, KuoSun} for the only results to have appeared since. 

Remark that a first major difference between ND and MJD is that the latter implies that the reduced degree of all simple components are at most $3$ \cite{AroHalPAs}. However there exists groups having ND with a simple component  of arbitrary large reduced degree, e.g. the groups $C_{p^{m}} \rtimes C_{p^n}$ in \cite[Theorem A]{JespersSun}.

Next, analyzing all groups for which the Multiplicative Jordan Decomposition is known to hold and those for which it remains open, using  \cite[Section 7.4]{PolMilSehgalBook} and \cite{JespersSun}, one finds first that all groups which are known to have the Multiplicative Jordan Decomposition have at most one matrix component. The only groups among those for which it remains open with more than one matrix component are the groups of type $C_p \rtimes C_{2^k}$ with $k \geq 3$ and $p \equiv 1 \bmod 8$ and where the action of the cyclic $2$-group is by inversion. Note that these groups are SSN groups of unfaithful type - so exactly from the series for which the equivalence between property ND and having at most one matrix component remains open. Hence an answer to the following might solve the Multiplicative Jordan Decomposition for a new series and provide an answer to whether the Multiplicative Jordan Decomposition for a group implies that it has at most one matrix component.

\begin{question} \label{when MJD satisfy conjecture}
Let $p$ and $q$ be primes and $G= C_p \rtimes C_{q^k}$ for some natural number $k$ such that the action of $C_{q^k}$ is not faithful. Is it true that $G$ has ND if and only if it has one matrix component?
\end{question}

The smallest group with more than one matrix component for which the Multiplicative Jordan Decomposition remains unknown is
\[ \langle x, a \ | \ x^{17} = a^8 = 1, x^a = x^{-1} \rangle  \cong C_{17} \rtimes C_8.\]
In \cite[Section 4.1]{HalesPassi17} it is called ``a challenging open case''. We can confirm it is challenging. Answering our question would also eliminate the last question mark in \cite[Figure 1]{JespersSun}.

\bibliographystyle{plain}
\bibliography{ND}

\end{document}